\documentclass{article} 
\usepackage[utf8]{inputenc}
\usepackage[T1]{fontenc}
\usepackage{lmodern}
\usepackage[a4paper]{geometry}
\usepackage{babel}
\usepackage{enumitem}
\usepackage{fullpage}

\usepackage[usenames, dvipsnames]{xcolor}
\usepackage{tikz}
\usetikzlibrary{patterns}
\usepackage[justification=centering]{caption}
\usepackage{pgfplots}
\pgfplotsset{compat=1.15}
\usepackage{mathrsfs}
\usetikzlibrary{arrows}
\usepackage{amssymb,amsmath,mathtools,amsthm,dsfont}
\usepackage{stmaryrd}

\usepackage[all]{xy}

\newtheorem{theoreme}{Théorème}[section]
\newtheorem{theorem}[theoreme]{Theorem}
\newtheorem{prop-f}[theoreme]{Proposition}
\newtheorem{prop}[theoreme]{Proposition}

\newtheorem{lemme}[theoreme]{Lemme}
\newtheorem{lemma}[theoreme]{Lemma}

\newtheorem{definen}[theoreme]{Definition}

\newtheorem{remk}[theoreme]{Remark}

\newcommand{\E}{\mathbb{E}}

\newcommand{\N}{\mathbb{N}}

\renewcommand{\P}{\mathbb{P}}

\newcommand{\R}{\mathbb{R}}

\newcommand{\Z}{\mathbb{Z}}
\renewcommand{\l}{\ell}
\renewcommand{\lll}{\ell^\Lambda}

\newcommand{\cA}{\mathcal{A}}
\newcommand{\cB}{\mathcal{B}}

\newcommand{\cD}{\mathcal{D}}
\newcommand{\cE}{\mathcal{E}}
\newcommand{\cF}{\mathcal{F}}
\newcommand{\cG}{\mathcal{G}}

\newcommand{\cM}{\mathcal{M}}
\newcommand{\cN}{\mathcal{N}}
\newcommand{\cO}{\mathcal{O}}
\newcommand{\cP}{\mathcal{P}}

\newcommand{\cT}{\mathcal{T}}

\newcommand{\fA}{\mathfrak{A}}
\newcommand{\OG}{\overline{\gamma}}

\newcommand{\gu}{\gamma^{*}}
\newcommand{\gd}{\gamma^{**}}

\newcommand{\ou}{\overline{u}}

\newcommand{\ow}{\overline{w}}

\newcommand{\td}{t^{**}}

\newcommand{\Tu}{T^{*}}
\newcommand{\Td}{T^{**}}
\newcommand{\Bd}{B^{**}}
\newcommand{\ogu}{\overline{\gamma}^{*}}
\newcommand{\ogd}{\overline{\gamma}^{**}}
\newcommand{\ud}{u^{**}}

\newcommand{\vd}{v^{**}}

\newcommand{\zd}{z^{**}}

\newcommand{\oz}{\omega}

\newcommand{\tpm}{\tilde{\pi}}

\newcommand{\ed}{\cE^{**}}
\newcommand{\Ed}{E^{**}}
\newcommand{\Ep}{E^*_+}
\newcommand{\Em}{E^*_-}
\newcommand{\Eep}{E^{**}_+}
\newcommand{\Eem}{E^{**}_-}
\newcommand{\EeM}{E^{**}_P}
\newcommand{\Bus}{B_{1,s,N}}
\newcommand{\Bds}{B_{2,s,N}}
\newcommand{\Bts}{B_{3,s,N}}
\newcommand{\Bqs}{B_{4,s,N}}
\newcommand{\tts}{t_{3,s,N}}
\newcommand{\ulm}{u^\Lambda}
\newcommand{\vlm}{v^\Lambda}
\renewcommand{\r}{t_{\min}}
\newcommand{\ST}{t_{\max}}
\newcommand{\rr}{\overline{r}}
\newcommand{\pp}{\pi_{u_2,u_3} \cup \pi_{v_3,v_2}}
\newcommand{\Sun}{S^1}
\newcommand{\Sde}{S^2}
\newcommand{\Mun}{M}

\newcommand{\Ka}{K}
\newcommand{\Kb}{K_1}

\newcommand{\Kd}{K_2}

\newcommand{\Kf}{K_3}

\newcommand{\Kh}{K''}
\newcommand{\Ki}{K'}

\newcommand{\Ca}{C_1}
\newcommand{\Cb}{C_2}
\newcommand{\Cc}{C}
\newcommand{\Cd}{C_3}
\newcommand{\Ce}{C_4}
\newcommand{\Da}{D_1}
\newcommand{\Db}{D_3}
\newcommand{\Lo}{\overline{L}}
\newcommand{\Lu}{\underline{L}}
\newcommand{\PP}{\pi^+}
\newcommand{\PPP}{\pi^{++}}

\newcommand{\aM}{a_{\max}}
\newcommand{\avant}{\text{before}}
\newcommand{\apres}{\text{after}}

\renewcommand{\1}{\mathds{1}}

\renewcommand{\1}{\mathds{1}}

\renewcommand{\epsilon}{\varepsilon}
\renewcommand{\phi}{\varphi}

\definecolor{qqwuqq}{rgb}{0,0.39215686274509803,0}
\definecolor{ccqqqq}{rgb}{0.8,0,0}

\numberwithin{equation}{section}
\newcounter{numeroexo}

\makeindex
\begin{document}
	
	\title{Geodesics in first-passage percolation cross any pattern}
	\author{Antonin Jacquet\footnote{Institut Denis Poisson, UMR-CNRS 7013, Université de Tours, antonin.jacquet@univ-tours.fr}}
	\date{}
	\maketitle
	\begin{abstract}
		In first-passage percolation, one places nonnegative i.i.d.\ random variables ($T (e)$) on the edges of $\Z^d$.  A geodesic is an optimal path for the passage times $T(e)$. Consider a local property of the time environment. We call it a pattern. We investigate the number of times a geodesic crosses a translate of this pattern. Under mild conditions, we show that, apart from an event with exponentially small probability, this number is linear in the distance between the extremities of the geodesic.
	\end{abstract}
	
	\tableofcontents
	
	\section{Introduction and main result}
	
	\subsection{Settings}\label{Settings}
	
	Fix an integer $d \ge 2$. In this article, we consider the model of first passage percolation on the hypercubic lattice $\Z^d$. We denote by $0$ the origin of $\Z^d$ and by $\cE$ the set of edges in this lattice. The edges in $\cE$ are those connecting two vertices $x$ and $y$ such that $\|x-y\|_1=1$. A finite path $\pi=(x_0,\dots,x_k)$ is a sequence of adjacent vertices of $\Z^d$, i.e.\ for all $i=0,\dots,k-1$, $\|x_{i+1}-x_i\|_1=1$. We say that $\pi$ goes from $x_0$ to $x_k$. Sometimes we identify a path with the sequence of the edges that it visits, writing $\pi=(e_1,...,e_k)$ where for $i=1,\dots,k$, $e_i=\{x_{i-1},x_i\}$. We say that $k$ is the length of $\pi$ and we denote $|\pi|=k$.
	
	The basic random object consists of a family $T=\{T(e) \, : \, e \in \cE\}$ of i.i.d.\ non-negative random variables defined on a probability space $(\Omega,\cF,\P)$, where $T(e)$ represents the passage time of the edge $e$. Their common distribution is denoted by $F$. The passage time $T(\pi)$ of a path $\pi=(e_1,\dots,e_k)$ is the sum of the variables $T(e_i)$ for $i=1,\dots,k$. 
	
	For two vertices $x$ and $y$, we define the geodesic time
	\begin{equation}
		t(x,y)= \inf \{T(\pi) \, : \, \pi \mbox{ is a path from $x$ to $y$} \}. \label{Définition geodesic time.}
	\end{equation}
	A self-avoiding path $\gamma$ such that $T(\gamma)=t(x,y)$ is called a geodesic between $x$ and $y$.
	
	For the following and for the existence of geodesics, we need some assumptions on $F$. Let $\r$ denote the minimum of the support of $F$. We recall a definition introduced in \cite{VdBK}. A distribution $F$ with support in $[0,\infty)$ is called useful if the following holds: 
	\begin{align*}
		F(\r) < p_c & \mbox{ when } \r=0, \\
		F(\r) < \overrightarrow{p_c} & \mbox{ when } \r>0,
	\end{align*}
	where $p_c$ denotes the critical probability for the Bernoulli bond percolation model on $\Z^d$ and $\overrightarrow{p_c}$ the critical probability for the oriented Bernoulli bond percolation. 
	
	In the whole article, we assume that $F$ has support in $[0,\infty)$, is useful, and that 
	\begin{equation}
		\E \min \left[ T^d_1,\dots,T^d_{2d} \right] < \infty, \label{17. Hypothèse sur l'espérance des ti.}
	\end{equation} where $T_1,\dots,T_{2d}$ are independent with distribution $F$.
	
	As $F$ is useful, $F(0)<p_c$. By Proposition 4.4 in \cite{50years}, we thus know that geodesics between any points exist with probability one. 
	
	\subsection{Patterns}\label{Sous-section patterns dans l'introduction.}
	
	For a set $B$ of vertices, we denote by $\partial B$ its boundary, this is the set of vertices which are in $B$ and which can be linked by an edge to a vertex which is not in $B$. We say that an edge $e=\{u,v\}$ is contained in a set of vertices if $u$ and $v$ are in this set.
	
	Let $L_1,\dots,L_d$ be non-negative integers. To avoid trivialities we assume that at least one of them is positive. We fix $\displaystyle \Lambda=\prod_{i=1}^d \{0,\dots,L_i\}$ and two distinct vertices $\ulm$ and $\vlm$ on the boundary of $\Lambda$. These points $\ulm$ and $\vlm$ are called endpoints. Then we fix an event $\cA^\Lambda$, with positive probability, only depending on the passage time of the edges joining two vertices of $\Lambda$. We say that $\mathfrak{P}=(\Lambda,u^\Lambda,v^\Lambda,\cA^\Lambda)$ is a pattern.
	Let $x \in \Z^d$. Define:
	\begin{itemize}
		\item for $y \in \Z^d$, $\theta_x y = y-x$,
		\item for $e=\{u,v\}$ an edge connecting two vertices $u$ and $v$, $\theta_x e = \{\theta_x u, \theta_x v\}$.
	\end{itemize}
	
	Similarly, if $\pi=(x_0,\dots,x_k)$ is a path, we define $\theta_x \pi=(\theta_x x_0,\dots, \theta_x x_k)$. Then $\theta_x T$ denotes the environment $T$ translated by $-x$, i.e. the family of random variables indexed by the edges of $\Z^d$ defined for all $e \in \cE$ by 
	\[
	\left(\theta_x T\right) (e) = T \left( \theta_{-x} e \right).
	\]
	
	Let $\pi$ be a self-avoiding path and $x \in \Z^d$. We say that $x$ satisfies the condition $(\pi ; \mathfrak{P})$ if these two conditions are satisfied:
	\begin{enumerate}
		\item $\theta_x \pi$ visits $u^\Lambda$ and $v^\Lambda$, and the subpath of $\theta_x \pi$ between $u^\Lambda$ and $v^\Lambda$ is entirely contained in $\Lambda$,
		\item $\theta_x T \in \cA^\Lambda$.
	\end{enumerate}
	Note that, if $x$ satisfies the condition $(\pi ; \mathfrak{P})$ when $\pi$ is a geodesic, then $(\theta_x \pi)_{u^\Lambda,v^\Lambda}$ is one of the optimal paths from $u^\Lambda$ to $v^\Lambda$ entirely contained in $\Lambda$ in the environment $\theta_x T$.	
	When the pattern is given, we also say "$\pi$ takes the pattern in $\theta_{-x}\Lambda$" for "$x$ satisfies the condition $(\pi;\mathfrak{P})$".
	We denote:
	\begin{equation}
		N^\mathfrak{P}(\pi)=\sum_{x \in \Z^d} \1_{\{x \text{ satisfies the condition }(\pi ; \mathfrak{P})\}}. \label{Compteur nombre de motifs empruntés introduction.}	
	\end{equation}
	Note that the number of terms in this sum is actually bounded from above by the number of vertices in $\pi$. If $N^\mathfrak{P}(\pi) \ge 1$, we say that $\pi$ takes the pattern. The aim of the article is to investigate, under reasonable conditions on $\mathfrak{P}$, the behavior of $N^\mathfrak{P}(\gamma)$ for all geodesics $\gamma$ from $0$ to $x$ with $\| x \|_1$ large. The first step is to determine these reasonable conditions, that is why we define the notion of valid patterns. 
	
	\begin{definen}
		Denote by $\{\epsilon_1,\dots,\epsilon_d\}$ the vectors of the canonical basis. An external normal unit vector associated to a vertex $z$ of the boundary of $\Lambda$ is an element $\alpha$ of the set $\{\pm\epsilon_1,\dots,\pm\epsilon_d\}$ such that $z+\alpha$ does not belong to $\Lambda$.
	\end{definen}
	
	\begin{definen}\label{Définition motif valable.}
		We say that a pattern is valid if the following two conditions hold:
		\begin{itemize}
			\item $\cA^\Lambda$ has a positive probability,
			\item the support of $F$ is unbounded or there exist two distinct external normal unit vectors, one associated with $\ulm$ and one associated with $\vlm$.
		\end{itemize} 
	\end{definen}
	
	\begin{remk}
		The existence of the two distinct vectors in the second condition of Definition \ref{Définition motif valable.} is equivalent to the fact that the endpoints of the pattern belong to two different faces.  
		Note that a real obstruction can appear when this second condition is not satisfied. For example, take $d=2$, $F = \frac{1}{2} \delta_1 + \frac{1}{2} \delta_4$, $\Lambda = \{0,1\} \times \{0,1,2,3\}$, $u^\Lambda=(0,2)$, $v^\Lambda=(0,1)$, and $\cA^\Lambda$ the event on which for all edges $e \in \Lambda$ such that $e$ is adjacent to $u^\Lambda$ or $v^\Lambda$, $T(e)=4$ and for all other edges $e$ of $\Lambda$, $T(e)=1$. This is the pattern of Figure \ref{Figure motif non valable dans l'intro.} and in what follows in this remark, we use the notations of this figure.
		The only geodesic from $\ulm$ to $\vlm$ entirely contained in the pattern is $(\ulm,\vlm)$. However, neither $(a,\ulm,\vlm)$, nor $(b,\ulm,\vlm)$, nor $(\ulm,\vlm,e)$, nor $(\ulm,\vlm,d)$ is a geodesic. Hence, every geodesic taking the pattern would contain the path $(c,\ulm,\vlm,f)$ but this path is not a geodesic.
	\end{remk}

	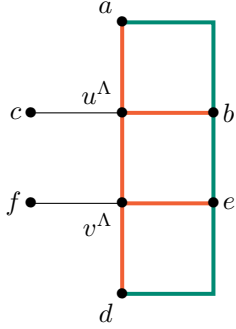
\begin{figure}
		\begin{center}
			\begin{tikzpicture}[scale=1.2]
				\draw (0,0) grid (1,3);
				\draw[color=PineGreen,line width=1.5pt] (0,3) -- (1,3) -- (1,0) -- (0,0);
				\draw[color=RedOrange,line width=1.5pt] (0,0) -- (0,3); 
				\draw[color=RedOrange,line width=1.5pt] (0,1) -- (1,1);
				\draw[color=RedOrange,line width=1.5pt] (0,2) -- (1,2);
				\draw (-1,1) -- (0,1);
				\draw (-1,2) -- (0,2);
				\draw (0,3) node[above left] {$a$};
				\draw (0,3) node {$\bullet$};
				\draw (1,2) node[right] {$b$};
				\draw (1,2) node {$\bullet$};
				\draw (-1,2) node[left] {$c$};
				\draw (-1,2) node {$\bullet$};
				\draw (0,0) node[below left] {$d$};
				\draw (0,0) node {$\bullet$};
				\draw (1,1) node[right] {$e$};
				\draw (1,1) node {$\bullet$};
				\draw (-1,1) node[left] {$f$};
				\draw (-1,1) node {$\bullet$};
				\draw (0,2) node[above left] {$\ulm$};
				\draw (0,2) node {$\bullet$};
				\draw (0,1) node[below left] {$\vlm$};
				\draw (0,1) node {$\bullet$};
			\end{tikzpicture}
			\caption{Example of a pattern which can not be taken by a geodesic. The passage times of the edges in red are equal to $4$ and those of the edges in green are equal to $1$.}\label{Figure motif non valable dans l'intro.}
		\end{center}
	\end{figure}
	
	\subsection{Main result}
	
	The main result of this paper is the following.
	
	\begin{theorem}\label{17. Théorème à démontrer.}
		Let $\mathfrak{P}=(\Lambda,u^\Lambda,v^\Lambda,\cA^\Lambda)$ be a valid pattern and assume that $F$ is useful and satisfies \eqref{17. Hypothèse sur l'espérance des ti.}. Then there exist $\alpha >0$, $\beta_1 >0$ and $\beta_2>0$ such that for all $x \in \Z^d$, \[ \P \left(\exists \mbox{ a geodesic $\gamma$ from $0$ to $x$ such that } N^\mathfrak{P}(\gamma) < \alpha \|x\|_1 \right) \le \beta_1 \mathrm{e}^{-\beta_2 \|x\|_1}. \]
	\end{theorem}
	Showing the existence of a constant $c > 0$  such that, for all large $n$,
	\begin{equation}\label{l:esperance-lineaire}
		\E[N^\mathfrak{P}(\pi(n))] \ge cn
	\end{equation}
	where $\mathfrak{P}$ is a properly designed pattern and where $\pi(n)$ is the first geodesic from $0$ to $n\epsilon_1$ (geodesics are ordered in an arbitrary way), has been a key intermediate result to show several properties in first-passage percolation.
	The first result of this kind appears in an article by van den Berg and Kesten \cite{VdBK}. Let us recall their setting. Assume that $F$ is a finite mean distribution on $[0,+\infty)$.
	Denote by $\mu(F)$ the time constant associated to $F$, that is
	\[
	\mu(F) = \lim_{n\to\infty} \frac{\E[t(0,n\epsilon_1)]}n.
	\]
	Let $\tilde F$ be another finite mean distribution on $[0,+\infty)$.
	Assume $F$ useful, $F \neq \tilde F$ and $d \ge 2$.
	If $\tilde F$ is more variable\footnote{One says that $\tilde F$ is more variable than $F$ 
		if there exists two random variables $T$ -- with distribution $F$ -- and $\tilde T$ -- with distribution $\tilde F$ -- such that $\E[\tilde T|T] \le T$.
		See Definition (2.1) and Theorem 2.6 in \cite{VdBK}.} than $F$,
	then $\mu(\tilde F) < \mu(F)$.
	This is the main result of \cite{VdBK} and estimate \eqref{l:esperance-lineaire} is the content of their Proposition 5.22. The proof relies on a modification argument. 
	
	In \cite{Nakajima}, Nakajima proves a version of \eqref{l:esperance-lineaire} to show that the number of geodesics between two vertices has an exponential growth if the distribution has an atom. The result can be deduced from Theorem \ref{17. Théorème à démontrer.} as follows. Denote by $\kappa$ an atom of the distribution. Consider the pattern $\mathfrak{P}=(\Lambda,\ulm,\vlm,\cA^\Lambda)$ where:
	\begin{itemize}
		\item $\displaystyle \Lambda=\{0,1\} \times \{0,1\} \times \prod_{i=3}^{d} \{0\}$,
		\item $\ulm=(0,\dots,0)$ and $\vlm=(1,1,0,\dots,0)$,
		\item $\cA^\Lambda$ the event on which the passage time of every edge of $\Lambda$ is equal to $\kappa$.
	\end{itemize}
	The key fact about this pattern is the following: each time a geodesic takes the pattern, the geodesic can chose any of the two optimal paths between the endpoints.
	
	Then, one of the most recent result of this kind appears in an article by Krishnan, Rassoul-Agha and Seppäläinen in \cite{KRAS} (see Theorem 5.4 and Theorem 6.2). They use \eqref{l:esperance-lineaire} for some specific geodesics in order to get results about the Euclidean length of geodesics and the strict concavity of the expected passage times as a function of the weight shifts.
	
	We explain in this section the differences between our result and the ones in \cite{KRAS}, \cite{Nakajima} and \cite{VdBK} but we give more details about \cite{KRAS} below in Section \ref{Sous-section Introduction Some applications.} as we wish to strengthen some of their results to illustrate the use of Theorem \ref{17. Théorème à démontrer.}. Theorem \ref{17. Théorème à démontrer.} is stronger on three aspects that are commented below:
	\begin{enumerate}
		\item\label{Intro 1.} It deals with general patterns while the results in \cite{KRAS}, \cite{Nakajima} and \cite{VdBK} are stated for specific patterns. 
		\item\label{Intro 2.} In the case of non-uniqueness of geodesics, it gives the result for all geodesics and not only for a specific one.
		\item\label{Intro 3.} It provides an at least linear growth of the number of crossed patterns out of an event of exponentially small probability.
	\end{enumerate}

	Since the proof given by van den Berg and Kesten, it has been clear that \eqref{l:esperance-lineaire} should be true for any reasonable pattern. As explained above, \eqref{l:esperance-lineaire} has indeed been proven for several specific patterns in \cite{KRAS}, \cite{Nakajima} and \cite{VdBK}. In some part of the proof of \eqref{l:esperance-lineaire}, one needs to design a new environment in which the geodesics have to cross the pattern. When the support of $F$ is unbounded, the argument is relatively straightforward. However, when the support of $F$ is bounded, this is more involved. Actually, in \cite{KRAS}, \cite{Nakajima} and \cite{VdBK}, when the support of $F$ is bounded, each of the proofs is technical and makes use of specific properties of the considered pattern. The extension to any reasonable pattern, while naturally expected, actually requires new arguments and is a significant difficulty in the proof of Theorem \ref{17. Théorème à démontrer.}. Thanks to Theorem \ref{17. Théorème à démontrer.}, we can for example generalize Theorem 6.2 in \cite{KRAS}. See Theorem \ref{Théorème 6.2 KRAS} below. Let us note however that the strategy developed in the bounded case in Section  \ref{Section preuve du théorème 6.2.} to remove the restriction in Assumption 6.1 in \cite{KRAS} could be used in the proof of Theorem 6.2 in \cite{KRAS}.
	
	In \cite{KRAS}, \cite{Nakajima} and \cite{VdBK}, \eqref{l:esperance-lineaire} is proven only for a specific geodesic. This has no consequence on the main results of \cite{Nakajima} and \cite{VdBK}. However obtaining a result for all geodesics enables to strengthen one of the main results of \cite{KRAS} in the bounded case. See Remark \ref{Remarque KRAS pour toute géodésique.}. Dealing with all geodesics is obtained thanks to a new idea using concentric annuli to define and localize good boxes (see Section \ref{Sous-section preuve, cas non borné}).
	
	The last difference with the results of \cite{KRAS}, \cite{Nakajima} and \cite{VdBK} is that our result is stronger than a result in expectation. However, notice that the result in expectation is sufficient for the applications in \cite{Nakajima} and \cite{VdBK}. We refer to Section \ref{Sous-section Introduction Some applications.} for comments on \cite{KRAS}.
	
	
	
	A result fulfilling items \ref{Intro 2.} and \ref{Intro 3.} above appears in an article by Andjel and Vares \cite{AndjelVares} for the number of edges with large time crossed by a geodesic.
	
	\begin{theorem}[Theorem 2.3 in \cite{AndjelVares}]
		Let $F$ be a useful distribution on $[0,+\infty)$ with unbounded support. Then, for each $M$ positive there exists $\epsilon=\epsilon(M)>0$ and $\alpha=\alpha(M)>0$ so that for all $x$, we have
		\begin{equation}
			\P \left(\exists \text{ geodesic $\pi$ from $0$ to $x$ such that } \sum_{e\in\pi} \1_{T(e) \ge M} \le \alpha \|x\|_1 \right) \le \mathrm{e}^{-\epsilon \|x\|_1}. \label{Théorème Andjel/Vares}
		\end{equation}		
	\end{theorem}
	
	Theorem \ref{17. Théorème à démontrer.} is a generalization of this theorem since, to get this result, we can take the pattern (reduced to one edge) $\mathfrak{P}=(\{\ulm,\vlm\},\ulm,\vlm,\cA^\Lambda)$ where $\ulm=(0,\dots,0)$, $\vlm=(1,0,\dots,0)$ and $\cA^\Lambda$ is the event on which the passage time of the only edge of the pattern is greater than $M$. The proof of Theorem \ref{17. Théorème à démontrer.} is partly inspired by the proof of this theorem and by the proof of \eqref{l:esperance-lineaire} in \cite{VdBK}.
	
	Even if it is stated for distributions with unbounded support, one can check that Theorem 2.3 in \cite{AndjelVares} holds for $F$ with bounded support with the same proof. As we need this extension in the proof of Theorem \ref{17. Théorème à démontrer.} we state it below.
	
	\begin{theorem}\label{Théorème Andjel/Vares dans le cas borné.}
		Let $F$ be a useful distribution on $[0,+\infty)$ with bounded support. Then, for each $M$ positive such that $F([M,+\infty))>0$, there exists $\epsilon=\epsilon(M)>0$ and $\alpha=\alpha(M)>0$ so that for all $x$, we have \eqref{Théorème Andjel/Vares}.
	\end{theorem}
	
	\subsection{Some applications}\label{Sous-section Introduction Some applications.}
	
	Several of the main results recently obtained in \cite{KRAS} are based on modification arguments leading to results of the type \eqref{l:esperance-lineaire}. We take advantage of Theorem \ref{17. Théorème à démontrer.} to slightly improve some of these results. The purpose of this section is primarily to illustrate the use of Theorem \ref{17. Théorème à démontrer.}, the details of the proofs are postponed to Section \ref{Section preuve des généralisations de KRAS.}.
	
	\subsubsection*{Euclidean length of geodesics}\label{Sous-section Théorème 6.2 dans KRAS.}
	
	Consider the following two assumptions 
	on the distribution $F$:
	\begin{enumerate}[label=(H\arabic*)]
		\item\label{Hypothèse H1} There exist strictly positive integers $k$ and $\l$ and atoms $r'_1,\dots,r'_{k+2\l},$ $s'_1,\dots,s'_k$ (not necessarily distinct) such that 
		\begin{equation}
			\sum_{i=1}^{k+2\l} r'_i = \sum_{j=1}^{k}s'_j. \label{Egalité avec les atomes}
		\end{equation}
		\item\label{Hypothèse H2} There exist strictly positive integers $k$ and $\l$ and atoms $r<s$ such that $(k+2\l)r=ks$, or $F$ has an atom in $0$.
	\end{enumerate}
	Note that \ref{Hypothèse H2} is strictly stronger than \ref{Hypothèse H1}. For $x \in \Z^d$, we denote by $\Lu_{0,x}$ (resp. $\Lo_{0,x}$) the minimal (resp. maximal) Euclidean length of self-avoiding geodesics from $0$ to $x$. In \cite{KRAS},  Krishnan, Rassoul-Agha and Seppäläinen prove the following theorem. 
	
	\begin{theorem}[Theorem 6.2 in \cite{KRAS}]\label{Vrai Théorème 6.2 KRAS.}
		Assume that $\P(T(e)=\r) < p_c$ and $\E \min \left[ T^p_1,\dots,T^p_{2d} \right] < \infty$ with $p>1$. Furthermore, assume one of the following two assumptions:
		\begin{itemize}
			\item the support of $F$ is unbounded and \ref{Hypothèse H1} is satisfied,
			\item the support of $F$ is bounded and \ref{Hypothèse H2} is satisfied.
		\end{itemize}
		Then, there exist constants $0<D,\delta,M<\infty$ such that
		\begin{equation}
			\P \left(\Lo_{0,x}-\Lu_{0,x} \ge D \|x\|_1 \right) \ge \delta \text{ for $\|x\|_1 \ge M$}. \label{Résultat vrai théorème 6.2.}
		\end{equation}
	\end{theorem} 
	
	We use Theorem \ref{17. Théorème à démontrer.} to prove the following result. It generalizes in a way Theorem \ref{Vrai Théorème 6.2 KRAS.} since in the case of bounded support, we have a less restrictive assumption and since the lower bound in \eqref{Résultat théorème 6.2.} is exponentially close to one in the distance instead of the uniform lower bound in \eqref{Résultat vrai théorème 6.2.}. However, the assumption on the moment is less restrictive in Theorem \ref{Vrai Théorème 6.2 KRAS.}.
	
	\begin{theorem}\label{Théorème 6.2 KRAS}
		Assume that $F$ is useful and $\E \min \left[ T^d_1,\dots,T^d_{2d} \right] < \infty$. Furthermore, assume \ref{Hypothèse H1}. Then there exist constants $0<\beta_1,\beta_2,D<\infty$ such that
		\begin{equation}
			\P \left(\Lo_{0,x}-\Lu_{0,x} \ge D \|x\|_1 \right) \ge 1-\beta_1 \mathrm{e}^{-\beta_2 \|x\|_1}. \label{Résultat théorème 6.2.}
		\end{equation}
	\end{theorem}
	
	The proof of this theorem is the aim of Section \ref{Section preuve du théorème 6.2.}.
	
	\subsubsection*{Strict concavity of the expected passage times as a function of the weight shifts}
	
	
	For $b \in \R$, define the $b$-shifted weights by \[T^{(b)}=\{T^{(b)}(e) \, : \, e \in \cE\} \text{ with } T^{(b)}(e)=T(e) + b \text{ for all } e \in \cE.\] Following the notations of \cite{KRAS} (see Section 2.2 in \cite{KRAS}), all the quantities associated with the passage times $T^{(b)}$ acquire the superscript. For example, $t^{(b)}(x,y)$ is the geodesic time between $x$ and $y$ defined at \eqref{Définition geodesic time.}, where the infimum is only on self-avoiding paths. Theorem A.1 in \cite{KRAS} gives the existence of a constant $\epsilon_0 > 0$ with which we have an extension of the Cox-Durett shape theorem for the shifted weights $T^{(-b)}$ for $b<\r+\epsilon_0$ (note that here the weights can be negative). Note that (ii) in Theorem A.1 in \cite{KRAS} guarantees that $\E [t^{(-b)}(0,x)]$ is finite if $b \in (0,\r + \epsilon_0)$.
	
	\begin{theorem}\label{Théorème 5.4 KRAS}
		Assume $F$ useful. Furthermore, assume that the support of $F$ is bounded and that it contains at least two strictly positive reals. Then, there exists a finite positive constant $M$ and a function $D(b)>0$ of $b>0$ such that the following bounds hold for all $b \in (0,\r+\epsilon_0)$ and all $\|x\|_1 \ge M$: 
		\begin{equation}
			\E [t^{(-b)}(0,x)] \le \E [t(0,x)] - b \E [\Lo_{0,x}]-D(b)b\|x\|_1. \label{Résultat Théorème 5.4 KRAS.}
		\end{equation} 
	\end{theorem}
	
	\begin{remk}\label{Remarque KRAS pour toute géodésique.}
		In Theorem \ref{Théorème 5.4 KRAS}, we slightly strengthen Theorem 5.4 in \cite{KRAS} in the bounded case. Indeed, $\E [\Lu_{0,x}]$ in \cite{KRAS} is replaced by $\E [\Lo_{0,x}]$ in \eqref{Résultat Théorème 5.4 KRAS.}. This strengthening is made possible by the fact that Theorem \ref{17. Théorème à démontrer.} gives a result for all geodesics and thus, in particular, for the geodesic of maximal Euclidean length. We focus on the bounded case in Theorem \ref{Théorème 5.4 KRAS} since Theorem 5.4 in \cite{KRAS} already contains \eqref{Résultat Théorème 5.4 KRAS.} in the unbounded case.
	\end{remk}
	
	
	\subsection{Sketch of the proof}\label{Section sketch of proof}
	
	In what follows, we give an informal sketch of proof of Theorem \ref{17. Théorème à démontrer.}. Fix a pattern $\mathfrak{P}$ and $x \in \Z^d$ with $\|x\|$ large. Consider the event:
	\[
	\cM = \{\text{there exists a geodesic from }0\text{ to }x\text{ which does not take the pattern}\}.
	\]
	The aim is to prove that $\cM$ has a probability small enough in $\|x\|$.
	More precisely, we want to prove
	\begin{equation}
		\P(\cM) \ll \frac{1}{\|x\|^{d-1}}. \label{Nouveau sketch of proof 4.}
	\end{equation}
	From this result, by a standard re-normalization argument, we easily get that, out of a very low probability event, every geodesic from $0$ to $x$ takes a number of patterns linear in $\|x\|$ (see Proposition \ref{17. Gros théorème à démontrer, un seul motif.} in Section \ref{Sous-section réduction dans l'introduction.} for a formal statement of \eqref{Nouveau sketch of proof 4.}).
	\paragraph*{General idea.}
	The idea is to define a suitable sequence of events $\cM(\ell)$ for $0 \le \ell \le q$ such that, for some positive constant $c<1$,
	\begin{enumerate}
		\item $q \ge c\|x\|$,
		\item $\cM \subset \cM(q) \cup \cB$ where $\displaystyle \P(\cB) \ll \frac{1}{\|x\|^{d-1}}$,
		\item for all $\ell \ge 1$, 
		\begin{equation}
			\P(\cM(\ell)) \le c \P(\cM(\ell-1)).\label{Nouveau sketch of proof 1.}
		\end{equation}
	\end{enumerate}
	If the above holds, we get $\P(\cM) \le c^{c\|x\|_1} + \P(\cB)$, which allows us to conclude.
	The event $\cM(\ell)$ is approximately "there exists a geodesic from $0$ to $x$ which does not take the pattern until a distance of order $\l$", where we have to precise the sense of "distance of order $\l$".
	The complementary event of $\cB$ is approximately "each geodesic crosses enough good boxes" and these good boxes are the boxes in which the environment and the geodesics behave in a typical way. This enables us to try to modify the environment to ensure that all geodesics from $0$ to $x$ take the pattern inside.
	We take a good definition for the event $\cM(\l)$ to have $\cM(\l) \subset \cM(\l-1)$ and thus \eqref{Nouveau sketch of proof 1.} is equivalent to the existence of a constant $\eta>0$ (by taking $\eta=\frac{1}{c}-1$) such that 
	\begin{equation}
		\P (\cM(\l-1) \setminus \cM(\l)) \ge \eta \P (\cM(\l)). \label{Nouveau sketch of proof 2.}
	\end{equation}
	To get \eqref{Nouveau sketch of proof 2.}, we would like to make a modification in an environment where $\cM(\l)$ occurs to get a new environment in which the event $\cM(\l-1) \setminus \cM(\l)$ occurs. This requires some stability in the definition of the events $\cM(\cdots)$ under the modification. This will be made clearer later. 
	
	\paragraph*{Associated geodesics.}
	We need the notion of associated paths. For the remaining of the sketch of the proof, "geodesic" means "geodesic from $0$ to $x$".
	Let $B$ be a set of vertices (it is intended to be the "selected box", i.e.\ the box where me make the modification). 
	Two paths $\pi^1$ and $\pi^2$ from $0$ to $x$ are $B$-associated if there exists $a,b \in B$ such that :
	\begin{enumerate}
		\item $\pi^1$ and $\pi^2$ visit successively $a$ and $b$.
		\item $\pi^1_{0,a} = \pi^2_{0,a}$ and $\pi^1_{b,x} = \pi^2_{b,x}$.
		\item $\pi^1_{a,b}$ and $\pi^2_{a,b}$ are entirely contained in $B$.
	\end{enumerate}
	In particular the two paths coincide outside $B$.
	With this definition, we can clearly enumerate the properties we want after the modification. Imagine we have a geodesic $\gamma$ "selected" in a certain way and a "good" box $B$ (where we want to make the modification) such that $\gamma$ crosses $B$. The aim is to modify the environment in $B$ such that:
	\begin{enumerate}
		\item Every geodesic in the new environment takes the pattern in $B$.
		\item Every geodesic in the new environment is $B$-associated with a geodesic in the original environment.
		\item The geodesic $\gamma$ is $B$-associated with at least one geodesic in the new environment. 
	\end{enumerate}
	If we identify two geodesics $B$-associated, the last two properties can be rephrased as follows: we have not won new geodesics, we have not lost the geodesic $\gamma$. It is one of the keys of the stability.
	Getting these properties in the unbounded case is elementary but it is a significant difficulty in the bounded case (see Section \ref{Sous-section schéma de preuve dans le cas borné.} where we give the main ideas of the modification).
	
	\paragraph*{Definition of the sequence $\cM(\l)$.}
	We would like to get a result on every geodesic (in the case where there is no uniqueness, which is a case we do not want to eliminate). A definition of the type:
	\[
	\cM(\ell) \approx \{\text{the first geodesic (in the lexicographical order) does not take the pattern in its first } \ell \text{ good boxes}\}
	\]
	does not provide a result for all geodesics. 
	However, it is possible to get a result for the first geodesic in the lexicographical order with this definition.
	
	If every box crossed were a good box, we could define a sequence of large concentric annuli centered in the origin and use a definition of the type:
	\[
	\cM(\ell) \approx \{\text{there exists a geodesic which does not take the pattern in the } \ell \text{ first annuli}\}.
	\]
	Thus, we could choose one of these geodesics (let us denote it by $\gamma$) and choose a box crossed by $\gamma$ in the $\l$-th annulus (let us denote it by $B$.)
	By making a modification giving the three properties stated above in the paragraph about the associated geodesics in the box $B$, in the new environment, the event $\cM(\ell-1) \setminus \cM(\ell)$ would occur.
	Indeed, $\cM(\l)$ would not occur since every geodesic in the new environment would take the pattern in $B$ and thus in the $\l$-th annulus. 
	However, $\cM(\l-1)$ would occur: this crucially uses the fact that $\gamma$ is $B$-associated with a geodesic in the new environment (and the fact that the environment outside $B$, and thus outside the $\l$-annulus, is not modified).
	
	We could use the above definition if every geodesic crossed a good box in every annulus. Since it is not the case, we have to use a definition of the type:
	\[
	\cM(\ell) \approx \{\text{there exists a geodesic $\gamma$ which does not take the pattern in the union of the first $a_\ell(\gamma)$ annuli}\}
	\]
	where $a_\ell(\gamma)$ is the index of the $\ell$-th annulus in which $\gamma$ crosses a good box (see Section \ref{Sous-section preuve, cas non borné} and Section \ref{Sous-section preuve, cas borné}).
	
	\paragraph*{Modification.}
	Fix a positive integer $\ell$.
	Now, the aim is to define a general plan to prove \eqref{Nouveau sketch of proof 2.}.
	Here we do not discuss the modification itself (see Section \ref{Sous-section schéma de preuve dans le cas borné.} in the bounded case).
	Recall that we define the sequence $\cM(\l)$ to have
	\begin{equation}
	\cM(\ell) \subset \cM(\ell-1). \label{Nouveau sketch of proof 3.}
	\end{equation}
	We denote by $T$ the environment (i.e.\ the family of passage times on the edges) and by $T'$ an independent copy of $T$. The basic idea consists in creating a modified environment $T^*=\phi(T,T')$ where $T^*(e)=T(e)$ for some edges (whose passage times do not change) and $T^*(e)=T'(e)$ for the other edges (whose passage times are re-sampled).
	In other words, we define a random set of edges $R(T)$ (the edges we want to re-sample) and we set $T^* = \phi(T,T')=\phi_{R(T)}(T,T')$ where, for every set of edges $r$,
	$\phi_r(T,T')(e)$ is equal to $T'(e)$ if $e \in r$ and to $T(e)$ else.
	
	In a utopian situation, imagine that we could define a new environment $T^*=\phi(T,T')$ and an event $\cA$ (ensuring the success of the modification) such that:
	\begin{itemize}
		\item $\Tu$ and $T$ have the same distribution,
		\item $\eta := \P(T' \in \cA) > 0$,
		\item and $\{T \in \cM(\ell) \text{ and } T' \in \cA\} \subset \{T^* \in \cM(\ell-1) \setminus \cM(\ell)\}$.
	\end{itemize}
	Then we would have 
	\[
	\P(T \in \cM(\ell-1) \setminus \cM(\ell)) = \P(T^* \in  \cM(\ell-1) \setminus \cM(\ell)) \ge \P(T \in \cM(\ell) \cap T' \in \cA) = \eta \P(T \in \cM(\ell)),
	\]
	which is \eqref{Nouveau sketch of proof 2.}.
	
	However, this situation is unrealistic since the set of re-sampled edges $R(T)$ depends on $T$, and thus the distribution of $\Tu$ is different from the distribution of $T$.
	But when $r$ is fixed, $\phi_r(T,T')$ and $T$ have the same distribution.
	It is possible to rely on this fact as soon as, observing only the modified environment, we can guess approximately in which box we performed the modification 
	(see the use of the $S^1$-variables and $S^2$-variables in Lemma \ref{17. cas non borné, lemme pour les inégalités avec les indicatrices dans la modification.} and Lemma \ref{17. ancien lemme 12.9.}). 
	
	In the case where the passage times are bounded, we use two independent copies of $T$ and we make a two-steps modification. The way we actually perform the modification in the bounded case is sketched in Section \ref{Sous-section schéma de preuve dans le cas borné.}. The ideas described in this paragraph can be adapted to this two-steps modification without difficulties.
		
	\paragraph*{Comparison with the plan of the proof of Proposition 5.22 in \cite{VdBK}.}
	
	Let us compare the above strategy with the plan used in \cite{VdBK} to prove Proposition 5.22. Fix $x$ in $\Z^d$. In \cite{VdBK}, van den Berg and Kesten also start by associating with some specific geodesic $\gamma$ some sequence of $q=C\|x\|_1$ good boxes. By simple geometric arguments, they then get some family $\cB$ of boxes such that \[\E [\text{number of boxes of $\cB$ which are good and crossed by $\gamma$}] \ge c\|x\|_1\] where $c$ is a positive constant. Fix some box $B \in \cB$. Then they also define a new environment $\Tu$ by resampling the times of the edges in $B$. It is then sufficient (this is the technical part of the proof in the bounded case) to prove 
	\[ \P(\text{every geodesic in $\Tu$ crosses the pattern in $B$}|\text{in the environment $T$, $\gamma$ crosses $B$ and $B$ is good}) \ge \eta \] for some positive constant $\eta>0$. In particular, and contrary to what happens in our framework, it is not necessary in this setting to control what happens to geodesic(s) outside the considered box when we resample the times of the edges in the box. This is not a problem if the geodesic in the new environment completely changes.
	
	\paragraph*{Comparison with the plan of the proof of Theorem 2.3 in \cite{AndjelVares}.}
	
	In \cite{AndjelVares}, the main difference with the strategy described above is the use of penalized geodesics. Indeed, Andjel and Vares only consider geodesics which do not take edges whose passage time is greater than $M$ and it allows them to get a result on all geodesics from $0$ to $x$ thanks to the modification argument. However, it seems difficult to use penalized geodesics with the patterns, that is why we use the strategy of concentric annuli developped in Section \ref{Sous-section preuve, cas non borné}. 
	
	
	\subsection{Reduction}\label{Sous-section réduction dans l'introduction.}
	
	One can check that, using a standard re-normalization argument, Theorem \ref{17. Théorème à démontrer.} is a simple consequence of the following proposition (see for example the proof of Theorem 2.3 in \cite{AndjelVares}).
	
	\begin{prop}\label{17. Gros théorème à démontrer, un seul motif.}
		Let $\mathfrak{P}=(\Lambda,u^\Lambda,v^\Lambda,\cA^\Lambda)$ be a valid pattern and assume that $F$ is useful and satisfies \eqref{17. Hypothèse sur l'espérance des ti.}. Then there exist $C>0$ and $D>0$ such that for all $n \ge 0$, for all $x$ such that $\|x\|_1=n$, 
		\begin{equation}
			\P \left( \mbox{$\exists$ a geodesic $\gamma$ from $0$ to $x$ such that } N^\mathfrak{P}(\gamma)=0 \right) \le D \mathrm{e}^{-C n^\frac{1}{d}}. \label{Résultat gros héorème à démontrer, un seul motif.}
		\end{equation}
	\end{prop}
	
	Thus, the aim of the paper is now to prove Proposition \ref{17. Gros théorème à démontrer, un seul motif.}.
	Although they share some similarities, the proofs of Proposition \ref{17. Gros théorème à démontrer, un seul motif.} differ according to whether the support of $F$ is bounded or unbounded. As the proof is easier in the unbounded case, we decide to first give the proof in the unbounded case in Section 2 and then give the proof in the bounded case in Section 3.
	
	\subsection{Some tools and notations}\label{Sous-section Some tools and notations.}
	
	In this subsection, we recall some results and fix some notations. First, we denote by $\N$ the set of all non-negative integers, by $\N^*$ the set $\N \setminus \{0\}$, and by $\R_+$ the set of all $x \in \R$ such that $x \ge 0$. 
	
	For a self-avoiding\footnote{The definition can be extended to not necessarily self-avoiding paths by saying that a vertex $x$ is visited by $\pi$ before $y$ if there exists $i_0 \in \{0,\dots,k\}$ such that $x_{i_0}=x$ and for all $j \in \{0,\dots,k\}$, $x_j=y$ implies that $j>i_0$.} path $\pi=(x_0,...,x_k)$ going from $x_0$ to $x_k$, we say that $x_i$ is visited by $\pi$ before $x_j$ if $i<j$; we say that an edge $\{x_i,x_{i+1}\}$ is visited before an edge $\{x_j,x_{j+1}\}$ if $i<j$. A subpath of $\pi$ going from $x_i$ to $x_j$ (where $i,j \in \{0,\dots,k\}$ and $i<j$) is the path $(x_i,\dots,x_j)$ and is denoted by $\pi_{x_i,x_j}$.
	
	For a set $B$ of vertices, we denote by $\partial B$ its boundary, this is the set of vertices of $B$ which can be linked by an edge to a vertex which is not in $B$. Note that when we define a set of vertices of $\Z^d$, sometimes we also want to say that an edge is contained in this set. So, we say that an edge $e=\{u,v\}$ is contained in a set of vertices if $u$ and $v$ are in this set. Since now a subset $B$ of $\Z^d$ can be seen as a set of vertices or as a set of edges, we denote by $|B|_v$ the number of vertices of $B$ and by $|B|_e$ its number of edges. 
	
	Then, we define different balls in $\Z^d$ or $\R^d$. For all $c \in \Z^d$ and $r \in \R_+$, we denote 
	\begin{align*}
		B_\infty (c,r) & = \{u \in \Z^d \, : \, \|u-c\|_\infty \le r \}, \\
		B_1 (c,r) & = \{u \in \Z^d \, : \, \|u-c\|_1 \le r \},
	\end{align*}
	and for $n \in \N^*$, we denote by $\Gamma_n$ the boundary of $B_1(0,n)$, i.e.\ \begin{equation}
		\Gamma_n=\{u \in \Z^d \, : \, \|u\|_1=n \}. \label{définition Gamman}
	\end{equation}
	Also for $c \in \Z^d$ and $r \in \R_+$, we denote by $B(c,r)$ the random ball \[B(c,r) = \{ u \in \Z^d \, : \, t(c,u) \le r \}. \] Then, for $x$ and $y$ in $\R^d$, we define $t(x,y)$ as $t(x',y')$ where $x'$ is the unique vertex in $\Z^d$ such that $x \in x' + [0,1)^d$ (similarly for $y'$). For $c \in \Z^d$ and $r \in \R_+$, we denote by $\tilde{B}(c,r)$ the random ball \[\tilde{B}(c,r) = \{ y \in \R^d \, : \, t(c,y) \le r \}. \]
	Let $x$ in $\R^d$. Thanks to \eqref{17. Hypothèse sur l'espérance des ti.}, we can define
	\begin{equation}
		\mu(x) = \lim\limits_{n \to \infty} \frac{t(0,nx)}{n} \, a.s. \label{Définition constante de temps.}
	\end{equation}  Thanks to the hypothesis \eqref{17. Hypothèse sur l'espérance des ti.} and since $F(0)<p_c(\Z^d)$, for all $x \in \R^d \setminus \{0\}$, we have $\mu(x) \in (0, \infty)$. Furthermore, $\mu$ is a norm on $\R^d$ and describes the first order of approximation of $\tilde{B}(0,r)$ when $r$ goes to infinity. For $c \in \Z^d$ and $r \in \R_+$, we denote \[B_\mu(c,r)=\{ y \in \R^d \, : \, \mu(c-y) \le r\}.\]
	Fix $\textbf{B} = B_\mu(0,1)$, then the Cox-Durett shape theorem (Theorem 2.16 in \cite{50years}) guarantees that for each $\epsilon > 0$, 
	\begin{equation}
		\P \left( (1-\epsilon) \textbf{B} \subset \frac{\tilde{B}(0,t)}{t} \subset (1+\epsilon) \textbf{B} \text{ for all large $t$} \right) = 1. \label{17. Théorème de forme asymptotique.}
	\end{equation}
	One can easily prove that the result \eqref{17. Théorème de forme asymptotique.} is equivalent to
	\begin{equation}
		\lim\limits_{\mu(x) \to \infty} \frac{t(0,x)-\mu(x)}{\mu(x)} = 0 \, \, \, \, a.s. \label{Autre version du théorème de forme asymptotique}
	\end{equation}
	Since $\mu$ is a norm, we can fix two constants $c_\mu > 0$ and $C_\mu > 0$ such that for all $y$ in $\R^d$, \[c_\mu \|y\|_1 \le \mu (y) \le C_\mu \|y\|_1. \]
	Finally, since $F$ is useful, by Lemma 5.5 in \cite{VdBK}, there exist $\delta = \delta(F) > 0$ and $D_0 = D_0(F)$ fixed for the remaining of the article such that for all $u, \, v \in \Z^d$,
	\begin{equation}
		\P (t(u,v) \le (\r+\delta) \|u-v\|_1) \le \mathrm{e}^{-D_0 \|u-v\|_1}. \label{17. définition de delta}
	\end{equation}
	Notice that, using the Borel-Cantelli lemma with this result, we get that for all $u \in \Z^d$ different from $0$, 
	\begin{equation}
		\mu(u) \ge (\r+\delta)\|u\|_1. \label{Lien entre mu et delta.}
	\end{equation}
	
	\section{Unbounded case}
	\subsection{Proof of Proposition \ref{17. Gros théorème à démontrer, un seul motif.} in the unbounded case}\label{Sous-section preuve, cas non borné}
	
	Let $\mathfrak{P}=(\Lambda,u^\Lambda,v^\Lambda,\cA^\Lambda)$ be a valid pattern. It is convenient to reduce to the case where there exists an integer $\lll > 0$, fixed for the remaining of the proof, such that\footnote{We make a very slight abuse of notation: we also consider patterns where $0$ is in the center of $\Lambda$.} $\Lambda=B_\infty(0,\lll)$. There is no loss of generality (see Lemma \ref{Lemme annexe A.1 sur les overlapping pattern.} in Appendix \ref{Annexe sur l'hypothèse sur les motif dans le cas non borné.}).
	Let us begin with the definitions of a typical box and of a successful box. To this end, we have to fix some constants. 
	
	\paragraph*{Boxes and constants.}
	
	Recall that $\r$, and $\delta$ have been fixed in the introduction. The minimum of the support of $F$ is denoted by $\r$ and $\delta$ comes from \eqref{17. définition de delta}. 
	Since \[\lim\limits_{M \to \infty} \P \left( \cA^\Lambda \mbox{ is realized and for all edges $e \in \Lambda$, } T(e) \le M  \right) = \P ( \cA^\Lambda )  >0, \] there exists a positive constant $M^\Lambda$ fixed for the rest of the proof such that \[\P \left( \cA^\Lambda \mbox{ is realized and for all edges $e \in \Lambda$, } T(e) \le M^\Lambda \right) >0.\]
	Even if it means replacing $\cA^\Lambda$ by $\cA^\Lambda \cap \{\forall e \in \Lambda, \, T(e) \le M^\Lambda\}$, we can assume that 
	\begin{equation}
		\cA^\Lambda \subset \{\forall e \in \Lambda, \, T(e) \le M^\Lambda\}.\label{On fixe M Lambda dans la preuve du cas non borné.}
	\end{equation}
	We fix
	\begin{equation}
		\tau^\Lambda = M^\Lambda \|u^\Lambda-v^\Lambda\|_1, \label{On fixe tau Lambda dans le cas non borné.}
	\end{equation}  which is an upper bound for the travel time of an optimal path (for the passage time) going from $u^\Lambda$ to $v^\Lambda$ and entirely contained in $\Lambda$ on the event $\cA^\Lambda$.
	
	For $i \in \{1,2,3\}$, $B_{i,s,N}$ is the ball in $\Z^d$ of radius $r_i N$ for the norm $\|.\|_1$ centered at the point $sN$ where the constants $r_i$ are defined as follows. We fix $r_1=d$. Denote by $\Ka$ the number of edges in $B_\infty(0,\lll+3)$. 
	Then, fix $r_2$ an integer such that 
	\begin{equation}
		r_2 \delta - r_1 (\r+\delta) - \Ka \r - \tau^\Lambda > 0. \label{17. r2 cas non borné}
	\end{equation}
	Let $r_{2,3}$ be a positive real such that \[B_{2,0,1} \subset B_\mu \left(0,\frac{r_{2,3}}{2} \right) \cap \Z^d ,\] then we fix $r_3$ an integer such that \[B_\mu (0,9 r_{2,3}) \cap \Z^d \subset B_{3,0,1}.\]

	We use the word "box" to talk about $\Bts$. Recall that we denote by $\partial B_{i,s,N}$ the boundary of $B_{i,s,N}$, that is the set of points $z \in \Z^d$ such that $\|z-sN\|_1=r_iN$. For $u$ and $v$ two vertices contained in $\Bts$, we denote by $\tts(u,v)$ the  minimum of the times of all paths entirely contained in $\Bts$ and going from $u$ to $v$.  
	
	\paragraph*{Crossed boxes and weakly crossed boxes.} 
	
	We say that a path
	\begin{itemize}
		\item crosses a box $\Bts$ if it visits a vertex in $\Bus$,
		\item weakly crosses a box $\Bts$ if it visits a vertex in $\Bds$.
	\end{itemize}

	\paragraph*{Paths associated in a box.}\label{Chemins associés dans une boîte.}
	
	We say that two paths $\gamma$ and $\gamma'$ from $0$ to the same vertex $x$ are associated in a box $\Bts$ if there exist two distinct vertices $s_1$ and $s_2$ such that the following conditions hold:
	\begin{itemize}
		\item $\gamma$ and $\gamma'$ visit $s_1$ and $s_2$,
		\item $\gamma_{0,s_1}=\gamma'_{0,s_1}$,
		\item $\gamma_{s_1,s_2}$ and $\gamma'_{s_1,s_2}$ are entirely contained in $\Bts$,
		\item $\gamma_{s_2,x}=\gamma'_{s_2,x}$.
	\end{itemize}
	In particular, these two paths coincide outside $\Bts$. 
	
	\paragraph*{Typical boxes.}
	
	We define a sequence $\left( \nu (N) \right)_{N \in \N^*}$ such that for all $N \in \N^*$, $\nu(N) > M^\Lambda$ and 
	\begin{equation}
		\lim\limits_{N \to \infty} \P \left( \sum_{e \in B_{2,0,N}} T(e) \ge \nu (N) \right) =0. \label{17. Définition de nu(N).}
	\end{equation}
	Note that $F((\nu(N),+\infty))>0$ for all $N \in \N^*$ since the support of $F$ is unbounded. 
	
	We can now define typical boxes. A box $\Bts$ is typical if it verifies the following properties:
	\begin{enumerate}[label=(\roman*)]
		\item $\cT(s;N)$ is realized, where $\cT(s;N)$ is the following event: \[ \left\{\sup_{z \in \Bds} \tts (Ns,z) \le r_{2,3} N \right\} \cap \left\{\inf_{z \in \partial \Bts} \tts (Ns,z) \ge 4 r_{2,3} N \right\},  \]
		\item every path $\pi$ entirely contained in $\Bts$ from $u_\pi$ to $v_\pi$ with $\|u_\pi-v_\pi\|_1 \ge (r_2-r_1) N$ has a passage time verifying: 
		\begin{equation}
			t(\pi) \ge (\r+\delta) \, \|u_\pi-v_\pi\|_1, \label{17. chemin pas anormalement courts version non bornée}
		\end{equation}
		\item $\displaystyle \sum_{e \in B_{2,0,N}} T(e) < \nu (N)$.
	\end{enumerate}
	
	\begin{lemma}\label{Gros lemme typical boxes}
		We have these three properties about typical boxes. 
		\begin{enumerate}
			\item Let $s \in \Z^d$ and $N \in \N^*$. If $\Bts$ is a typical box, for all points $u_0$ and $v_0$ in $\Bds$, every geodesic from $u_0$ to $v_0$ is entirely contained in $\Bts$. 
			\item Let $s \in \Z^d$ and $N \in \N^*$. The typical box property only depends on the passage times of the edges in $\Bts$.
			\item We have \[\lim\limits_{N \to \infty} \P (B_{3,0,N} \mbox{ is a typical box}) = 1. \]
		\end{enumerate}
	\end{lemma}
	
	This lemma guarantees that the properties of a typical box are indeed typical ones and that they are also local ones. Its proof is in Section \ref{Typical boxes crossed, cas non borné}. Let us now introduce some further definitions. 
	
	\paragraph*{Successful boxes.}
	
	For a fixed $x \in \Z^d$, a box $\Bts$ is successful if every geodesic from $0$ to $x$ takes a pattern which is entirely contained in $\Bds$, i.e.\ if for every geodesic $\gamma$ going from $0$ to $x$, there exists $x_\gamma \in \Z^d$ satisfying the condition $(\gamma;\mathfrak{P})$ and such that $B_\infty(x_\gamma,\lll)$ is contained in $\Bds$. Note that the notion of successful box depends on some fixed $x \in \Z^d$.
	
	\paragraph*{Annuli.}
	
	Now, we define the annuli mentioned in Section \ref{Section sketch of proof}. Fix \begin{equation}
		r=2(r_1+r_3+1), \label{on fixe r.}
	\end{equation} and for all integers $i \ge 1$, let us define \[A_{i,N}=\left\{ y \in \Z^d \, : \, \|y\|_1 \in [(i-1)rN,irN) \right\}. \] For any annulus $A_{i,N}$, we call $\{ y \in \Z^d \, : \, \|y\|_1=(i-1)rN\}$ its inner sphere and $\{ y \in \Z^d \, : \, \|y\|_1=irN\}$ its outer sphere. Then, we give two definitions about these annuli which are useful in the proof of Lemma \ref{17. cas non borné, lemme pour les inégalités avec les indicatrices dans la modification.}.
	
	\begin{itemize}
		\item For $i>1$, we say that a path from $0$ to a vertex of $\Z^d$ crosses (resp. weakly crosses) a box $\Bts$ in the annulus $A_{i,N}$ if the two following conditions are satisfied:
		\begin{itemize}
			\item it crosses (resp. weakly crosses) this box before it visits for the first time the outer sphere of $A_{i,N}$,
			\item $\Bts$ is entirely contained in the annulus, i.e.\ every vertex of $\Bts$ is in $A_{i,N}$.
		\end{itemize}
		\item We also say that a path takes a pattern in the annulus $A_{i,N}$ if it takes a pattern which is entirely contained in $A_{i,N}$, i.e.\ if every vertex of this pattern is in $A_{i,N}$. Here, we do not require that the path takes a pattern before it visits the outer sphere of $A_{i,N}$ for the first time. 
	\end{itemize}
	Note that the choice of $r$ guarantees that every path passing through an annulus has to cross a box in this annulus. 
	
	For all integer $p \ge 2$ and all $N \in \N^*$, we denote by $\cG^p(N)$ the event on which for all $x$ in the outer sphere of the $p$-th annulus, every geodesic from $0$ to $x$ crosses a typical box in at least $\left\lfloor \frac{p}{2} \right\rfloor$ annuli. The following lemma, whose proof is given in Section \ref{Typical boxes crossed, cas non borné}, gives us an exponential decrease of the probability of the complement of $\cG^p(N)$.
	
	\begin{lemma}\label{17. Majoration de la proba du complémentaire de G.}
		There exist two positive constants $\Ca$ and $\Da$, and an integer $N_0 \ge 1$ such that for all $p \ge 1$ and $N \ge N_0$, \[\P\left( \cG^p(N)^c \right) \le \Da \mathrm{e}^{-\Ca p^{\frac{1}{d}} }. \]
	\end{lemma}
	
	\paragraph*{Setup for the proof of Proposition \ref{17. Gros théorème à démontrer, un seul motif.}.}
	
	For the rest of the proof, we fix $\Ca$, $\Da$ and $N_0$ given by Lemma \ref{17. Majoration de la proba du complémentaire de G.}. Recall that $\Ka$ is the number of edges in $B_\infty(0,\lll+3)$. Then we fix $\delta'>0$ such that
	\begin{equation}
		r_2 (\delta - \delta')-r_1(\r+\delta) - \Ka (\r+\delta') - \tau^\Lambda > 0. \label{17. inégalité delta', cas non bornée.}
	\end{equation} Note that it is possible since we have taken $r_2$ large enough (see \eqref{17. r2 cas non borné}). Then, fix
	
	\begin{equation}
		N \ge \max(N_0,\lll + 3), \, n \ge 2rN \mbox{ and } x \in \Gamma_n, \label{constantes N, n et x fixées.}
	\end{equation}
	(where $\Gamma_n$ is defined at \eqref{définition Gamman}). Fix $\displaystyle p = \left\lfloor \frac{n}{rN} \right\rfloor$ and $\displaystyle q = \left\lfloor \frac{p}{2} \right\rfloor$. Note that $x$ belongs to the $(p+1)$-th annulus. 
	
	\paragraph*{$\Mun$-sequences.}

	Let us now define some random sets and variables which are useful for stability questions for the modification argument. Unless otherwise specified, in the remaining of this section, we write geodesic as a shorthand for geodesic from $0$ to $x$.
	
	First, let us associate a sequence of $0$ to $p-1$ typical boxes to every geodesic from $0$ to $x$. For a geodesic $\gamma$, the deterministic construction is what follows. 
	
	Initialize the sequence as an empty sequence. For $j$ from $1$ to $p-1$, do:
	\begin{itemize}
		\item let $a_j(\gamma)$ be the index of the first annulus such that $\gamma$ crosses a typical box in this annulus and such that $a_j(\gamma)>a_{j-1}(\gamma)$ (where $a_0(\gamma)=1$). If there is no such annulus, then we stop the algorithm. 
		\item Add the first typical box crossed\footnote{If a path crosses two boxes $B_{3,s_1,N}$ and $B_{3,s_2,N}$, we say that it crosses $B_{3,s_1,N}$ before $B_{3,s_2,N}$ if it visits a vertex of $B_{1,s_1,N}$ before one of $B_{1,s_2,N}$.} by $\gamma$ in the annulus $A_{a_j(\gamma),N}$ to the sequence. Note that, since this typical box is crossed before $\gamma$ leaves $A_{a_j(\gamma),N}$ by the outer sphere for the first time, the $j$-th box of the sequence is crossed by $\gamma$ after the $(j-1)$-th one.
	\end{itemize}
	
	So, we get a sequence of at most $p-1$ boxes crossed by the geodesic. These boxes are all in different annuli. Furthermore, every box of this sequence is crossed by $\gamma$ before $\gamma$ leaves the annulus containing it for the first time by the outer sphere. If the event $\cG^p(N)$ occurs, we know that all these sequences have at least $q$ elements. For $j \in \{1,\dots ,p-1 \}$, we define a set of geodesics $\Gamma^j$. A geodesic $\gamma$ from $0$ to $x$ belongs to $\Gamma^j$ if:
	\begin{itemize}
		\item the length of its sequences defined above is greater than or equal to $j$,
		\item $\gamma$ does not take the pattern in any annuli $A_{k,N}$ with $k \le a_j(\gamma)$. 
	\end{itemize} 
	We call the sequences defined above the $\Mun$-sequences. 
	
	\paragraph*{Selected geodesic and $\Sun$-variables.}
	
	Then, for $j \in \{1,\dots,p-1\}$, if $\Gamma^j$ is not empty, we define the selected geodesic among the geodesics of $\Gamma^j$ as the one which minimizes the index of the annulus containing the $j$-th box of its sequence. If there are several such geodesics, the selected one is the first in the lexicographical order. Then, the random variable $\Sun_j$ is equal to the vertex $s$ such that the box $\Bts$ is the $j$-th box in the $\Mun$-sequence of the selected geodesic. When $j$ is fixed, we say that the box $B_{3,\Sun_j,N}$ is the selected box. Finally, if $\Gamma^j$ is empty, set $\Sun_j=0$ and there is no selected geodesic.
	
	\paragraph*{$\Sde$-variables.}
	
	After a modification in the box of the selected geodesic, we have to select a new box in the new environment $\Tu$ defined below. Since we can not guarantee that the selected box in the initial environment is still typical in the new environment, we have to select it in a different way. That is why we introduce the $\Sde$-variables. Let $j \in \{1,\dots,p-1\}$, we define $\Sde_j$ as follows. For every geodesic $\gamma \in \Gamma^{j-1}$, we define, if it is possible, $\Sde_j(\gamma)$ as the vertex $s$ corresponding to the box $\Bts$ where $\Bts$ is the first successful box crossed by $\gamma$ in an annulus (in the sense given with the definition of the annuli above). If it is not possible, $\Sde_j(\gamma)=0$. We denote by $a'_j(\gamma)$ the index of the annulus containing $B_{3,\Sde_j(\gamma),N}$. Then, $\Sde_j$ is equal to the vertex $\Sde_j(\OG)$ where $\OG$ satisfies the three following conditions:
	\begin{itemize}
		\item $a'_j(\OG)>1$,
		\item for all geodesic $\gamma \in \Gamma^{j-1}$ such that $a'_j(\gamma) \neq 1$, we have $a'_j(\gamma) \ge a'_j(\OG)$,
		\item $\OG$ is the first geodesic in the lexicographical order among the geodesics $\gamma$ such that $a'_j(\gamma) = a'_j(\OG)$.
	\end{itemize}
	If it is not possible, $\Sde_j=0$.
	
	\paragraph*{Modification argument.}
	
	Finally, for $j \in \{1,...,p-1\}$, we define $\cM(j)$ as the event on which every geodesic from $0$ to $x$ has at least $j$ typical boxes in its $\Mun$-sequence and there exists one geodesic which does not take the pattern in any annuli $A_{k,N}$ with $k \le a_j(\gamma)$. We also define $\cM(0)$ as the event on which there exists a geodesic from $0$ to $x$. Its probability is equal to $1$ (see Section \ref{Settings}). Now, the aim is to bound from above $\P(\cM(q))$ independently of $x$ since we have:
	\begin{equation}
		\P \left( \mbox{there exists a geodesic $\gamma$ from $0$ to $x$ such that } N^\mathfrak{P}(\gamma)=0 \right) \le \P(T \in \cG^p(N)^c) + \P(T \in \cM(q)). \label{inégalité preuve de la proposition principale.}
	\end{equation}
	
	In the sequel, we introduce an independent copy $T'$ of the environment $T$, the two being defined on the same probability space. It is thus convenient to refer to the considered environment when dealing with the objects defined above. To this aim, we shall use the notation $\{T \in \cM(j)\}$ to denote that the event $\cM(j)$ holds with respect to the environment $T$. In other words, $\cM(j)$ is now seen as a subset of $(\R_+)^\cE$, where $\cE$ is the set of all the edges. Similarly, for $i \in \{1,2\}$ we denote by $S^i_j(T')$ the random variables defined above but in the environment $T'$.
	
	Fix $\l \in \{1,\dots,q\}$. On $\{T \in \cM(\l)\}$, $\Gamma^\l \neq \emptyset$ and $B_{3,\Sun_\l(T),N}$ is a typical box crossed by the selected geodesic. We get a new environment $\Tu$ defined for all edge $e$ by: 
	\[\Tu(e) = \left\{
	\begin{array}{ll}
		T(e) & \mbox{if } e \notin B_{2,\Sun_\l(T),N} \\
		T'(e) & \mbox{else.}
	\end{array}
	\right.\]
	For $y$ and $z$ in $\Z^d$, we denote by $t^*(y,z)$ the geodesic time between $y$ and $z$ in the environment $\Tu$. Note that $T$ and $\Tu$ do not have the same distribution. 
	
	\begin{lemma}\label{17. cas non borné, lemme pour les inégalités avec les indicatrices dans la modification.}
		There exists $\eta=\eta(N) > 0$ such that for all $\l$ in $\{1,\dots,q\}$, there exist measurable functions $\Ep \, : \, \R_+^\cE \mapsto \cP(\cE)$ and $\Em \, : \, \R_+^\cE \mapsto \cP(\cE)$ such that:
		
		\begin{enumerate}[label=(\roman*)]
			\item $\Ep(T) \cap \Em(T) = \emptyset$ and $\Ep(T) \cup \Em(T) \subset B_{2,\Sun_\l (T),N}$,
			\item on the event $\{T \in \cM(\l) \}$, we have $\P \left( T' \in \cB^*(T) | T \right) \ge \eta$ where $\{T' \in \cB^*(T)\}$ is a shorthand for \[\{\forall e \in \Ep(T), \, T'(e) \ge \nu(N), \, \forall e \in \Em(T), T'(e) \le \r+\delta', \, \theta_{N \Sun_\l(T)} T' \in \cA^\Lambda\},\]
			\item $\{T \in \cM(\l) \} \cap \{T' \in \cB^*(T)\} \subset \{\Tu \in \cM(\l-1) \setminus \cM(\l)\}$ and $\|\Sde_\l(\Tu)-\Sun_\l(T)\|_1 \le 2r_3$.
		\end{enumerate}
	\end{lemma}
	
	\begin{remk}
		We have $\Ep(T) \cup \Em(T) \cup (N \Sun_\l(T) + \Lambda ) = B_{2,\Sun_\l(T),N}$. Note that since $r_1=d$ and $N \ge \lll + 3$, we have that $(N \Sun_\l(T) + \Lambda ) \subset B_{1,\Sun_\l(T),N}$. 
	\end{remk}
	
	Lemma \ref{17. cas non borné, lemme pour les inégalités avec les indicatrices dans la modification.} is a consequence of Lemma \ref{17. gros lemme modification, cas non borné.} below whose proof is given in Section \ref{Sous-section modification argument, cas non borné}. Recall the definition of associated paths given page \pageref{Chemins associés dans une boîte.}. 
	
	\begin{lemma}\label{17. gros lemme modification, cas non borné.}
		There exists $\eta=\eta(N) > 0$ such that for all $\l$ in $\{1,\dots,q\}$, there exist measurable functions $\Ep \, : \, \R_+^\cE \mapsto \cP(\cE)$ and $\Em \, : \, \R_+^\cE \mapsto \cP(\cE)$ such that $(i)$ and $(ii)$ of Lemma \ref{17. cas non borné, lemme pour les inégalités avec les indicatrices dans la modification.} are satisfied and such that if the event $\{T \in \cM(\l) \}\cap \{T' \in \cB^*(T)\} \mbox{ occurs,}$ then  we have the following properties: 
		\begin{enumerate}[label=(\roman*)]
			\item  in the environment $\Tu$, every geodesic from $0$ to $x$ takes the pattern inside $B_{2,\Sun_\l(T),N}$,
			\item for all geodesic $\ogu$ from $0$ to $x$ in the environment $\Tu$, there exists a geodesic $\OG$ from $0$ to $x$ in the environment $T$ such that $\OG$ and $\ogu$ are associated in $B_{3,\Sun_\l(T),N}$,
			\item there exists a geodesic $\gu$ in the environment $\Tu$ from $0$ to $x$ such that $\gu$ and the selected geodesic $\gamma$ in the environment $T$ are associated in $B_{3,\Sun_\l(T),N}$.
		\end{enumerate}
	\end{lemma}
	
	\begin{proof}[Proof of Lemma \ref{17. cas non borné, lemme pour les inégalités avec les indicatrices dans la modification.}]
		Let $\l \in \{1,...,q\}$. Consider $\Ep$ and $\Em$ given by Lemma \ref{17. gros lemme modification, cas non borné.}. Let $s \in \Z^d$ and assume that the event $\{T \in \cM(\l) \}\cap \{\Sun_\l(T)=s\} \cap \{T' \in \cB^*(T)\} \mbox{ occurs}$. To prove that the event $\{\Tu \in \cM(\l-1) \setminus \cM(\l)\} \mbox{ occurs}$ and that $\|\Sde_\l(\Tu)-s\|_1 \le 2r_3$, it is sufficient to prove that we have the four following points in the environment $\Tu$:   
		\begin{enumerate}
			\item every geodesic from $0$ to $x$ has at least $\l-1$ typical boxes in its $\Mun$-sequence,
			\item there exists a geodesic from $0$ to $x$ which does not take the pattern in the annuli until the one containing its $(\l-1)$-th box, 
			\item every geodesic from $0$ to $x$ whose $\Mun$-sequence contains at least $\l$ elements takes the pattern in an annulus whose index is smaller than or equal to the one containing its $\l$-th box,
			\item there exists $s'$ such that $\Sde_\l(\Tu)=s'$ and $\|s-s'\|_1 \le 2r_3$. 
		\end{enumerate}
		
		Let us start with a few remarks. We denote by $a_\l$ the index of the annulus which contains $B_{3,\Sun_\l(T),N}$.
		\begin{enumerate}[label=(\alph*)]
			\item The environments $T$ and $\Tu$ coincides outside the box $B_{2,\Sun_\l(T),N}$. As this box is included in the annulus $A_{a_\l,N}$, the environments $T$ and $\Tu$ are the same in all the other annuli. In particular, any box contained in an annulus $A_{i,N}$ for $i \neq a_\l$ is typical in $T$ if and only if it is typical in $\Tu$. 
			\item Similarly, every path $\pi$ takes a pattern which is outside the box $B_{3,\Sun_\l(T),N}$ in the environment $T$ if and only if it takes this pattern in the environment $\Tu$. In particular, for any $i \neq a_\l$, $\pi$ takes		 the pattern in the annulus $A_{i,N}$ in the environment $T$ if and only if it takes the pattern in the annulus $A_{i,N}$ in the environment $\Tu$. 
			\item Let $\OG$ and $\ogu$ be as in item $(ii)$ of Lemma \ref{17. gros lemme modification, cas non borné.}. Then $\OG$ and $\ogu$ coincide except maybe for the part between $s_1$ and $s_2$. 
			To sum up: 
			\begin{equation}
				\OG_{0,s_1}=\ogu_{0,s_1} \mbox{ and } \OG_{s_2,x}=\ogu_{s_2,x} \mbox{ and } \OG_{s_1,s_2} \subset B_{3,\Sun_\l(T),N} \subset A_{a_\l,N} \mbox{ and } 	\ogu_{s_1,s_2} \subset B_{3,\Sun_\l(T),N} \subset A_{a_\l,N}. \label{remarque (c) preuve de la stabilité.}
			\end{equation}
			Furthermore, by remark (b), we have that for any $i \neq a_\l$, $\OG$ takes the pattern in the annulus $A_{i,N}$ in the environment $T$ if and only if $\ogu$ takes the pattern in the annulus $A_{i,N}$ in the environment $\Tu$. The same property holds for the selected geodesic $\gamma$ and for the associated geodesic $\gu$ (in the environment $\Tu$) given by item $(iii)$ of Lemma \ref{17. gros lemme modification, cas non borné.}. 
			\item Let again $\OG$ and $\ogu$ be as in item $(ii)$ of Lemma \ref{17. gros lemme modification, cas non borné.}. Let us compare the $M$-sequence of $\OG$ (which is built in the environment $T$) with the $M$-sequence of $\ogu$ (which is built in the environment $\Tu$). By $(a)$ and \eqref{remarque (c) preuve de la stabilité.}, we get that any box which belongs to the $M$-sequence of $\OG$, with the possible exception of a box contained in $A_{a_\l,N}$, also belongs to the $M$-sequence of $\ogu$. The same property holds for the selected geodesic $\gamma$ and for the associated geodesic $\gu$ (in the environment $\Tu$) given by item $(iii)$ of Lemma \ref{17. gros lemme modification, cas non borné.}.  In particular the first $\l-1$ elements of the $M$-sequence of $\gamma$ (which is built in the environment $T$) are the same as the first $\l-1$ elements of the $M$-sequence of $\gu$ (which is built in the environment $\Tu$).
		\end{enumerate}
		
		Let us now proceed to the proof of item 1. We assume $\l \ge 2$ otherwise there is nothing to prove. Let $\ogu$ be a geodesic from $0$ to $x$ in the environment $\Tu$. Let $\OG$ be the associated geodesic in the environment $T$ given by item $(ii)$ of Lemma \ref{17. gros lemme modification, cas non borné.}. Since the event $\{T \in \cM(\l)\}$ occurs, the $M$-sequence (in the environment $T$) of $\OG$ contains at least $\l$ typical boxes. By remark (d) above, the $M$-sequence (in the environment $\Tu$) of $\ogu$ contains at least $\l-1$ typical boxes. 
		
		Let us consider item 2. We can again assume $\l \ge 2$. Let $\gu$ be the geodesic given by item $(iii)$ of Lemma \ref{17. gros lemme modification, cas non borné.}. Recall that $\gamma$ is the selected geodesic and that $\gamma \in \Gamma^\l$. In particular, we have the following properties: its $M$-sequence contains at least $\l$ boxes; the $\l$-th box of its $M$-sequence belongs to $A_{a_\l,N}$; $\gamma$ does not take the pattern in any annulus whose index is smaller than or equal to $a_\l$. Therefore, by remark (d) above, the first $\l-1$ boxes of the $M$-sequence of $\gamma$ and $\gu$ are the same. By remark (c) above, we conclude that $\gu$ does not take the pattern (in the environment $\Tu$) in any annulus whose index is smaller than the one containing its $(\l-1)$-th box. 
		
		Let us prove item 3. Let $\ogu$ be such a geodesic. Assume that the $\l$-th box of the $\Mun$-sequence of $\ogu$ is in an annulus whose index is strictly smaller than $a_\l$. Let $\OG$ be a geodesic in the environment $T$ given by item $(ii)$ of Lemma \ref{17. gros lemme modification, cas non borné.}. Assume, aiming at a contradiction, that $\ogu$ does not take the pattern in an annulus until the one containing its $\l$-th box. By remark (d), the $\l$ first boxes of the $M$-sequences of $\ogu$ and $\OG$ are the same. By remarks (b) and (c), $\OG$ does not take the pattern until the annulus containing its $\l$-th box. This contradicts the definition of $\Sun_\l$, so it is impossible. Thus the $\l$-th box of the $\Mun$-sequence of $\ogu$ is in an annulus whose index is greater than or equal to $a_\l$. By item $(i)$ of Lemma \ref{17. gros lemme modification, cas non borné.}, $\ogu$ takes the pattern in the annulus whose index is $a_\l$. Therefore it takes the pattern in an annulus whose index is smaller than or equal to the one containing its $\l$-th box and the third point is satisfied.
		
		Finally, let us prove item 4. Note that since $\Bts$ is a successful box, $\Sde_\l(\Tu) \neq 0$. There are two steps. First, we prove that $N \Sde_\l(\Tu)$ is the center of a box contained in the annulus $A_{a_\l,N}$ and then we prove that $\Bts \cap B_{3,\Sde(\Tu),N} \neq \emptyset$, which gives the result. Assume that $N \Sde_\l(\Tu)$ is not the center of a box contained in $A_{a_\l,N}$. Let again $\gu$ be the geodesic given by item $(iii)$ of Lemma \ref{17. gros lemme modification, cas non borné.} and recall that $\gamma$ is the selected geodesic in the environment $T$. Since $\Bts$ is successful, we have $a'_\l(\gamma) \le a_\l$ and thus the definition of $\Sde_\l(\Tu)$ implies that $N \Sde_\l(\Tu)$ is the center of a box contained in an annulus $A_{k,N}$ such that $k < a_\l$. So $\gu$ takes the pattern in this annulus in the environment $\Tu$. By remark (c), $\gamma$ takes the pattern in this annulus in the environment $\Tu$ and then in the environment $T$ by remark (b). This is impossible since $\gamma$ does not take the pattern in any annuli $A_{k,N}$ with $k \le a_\l$. 
		
		To conclude, assume, aiming at a contradiction, that $\Bts \cap B_{3,\Sde(\Tu),N} = \emptyset$. It implies that $\gu$ takes a pattern outside $\Bts$ in the environment $\Tu$ and the portion of $\gu$ taking this pattern is a portion of $\gu_{0,s'_1}$ or $\gu_{s'_2,x}$. By remark (b), the portion of $\gu$ taking this pattern also takes this pattern in the environment $T$. By remark (c), it implies that $\gamma$ takes the pattern in the environment $T$ in the annulus $A_{a_\l,N}$, which is a contradiction.
	\end{proof}
	
	Now, thanks to Lemma \ref{17. cas non borné, lemme pour les inégalités avec les indicatrices dans la modification.}, we can adapt Lemma 3.8 from \cite{AndjelVares}. 
	
	\begin{lemma}\label{17. Majoration par lambda puissance q}
		There exists $\lambda \in (0,1)$, which does not depend on $n$ and on $x \in \Gamma_n$, such that \[\P \left( T \in \cM(q) \right) \le \lambda^q. \]
	\end{lemma}
	
	\begin{proof}[Proof]
		Let $\l$ be in  $\{1,...,q\}$. For every $s \in \Z^d$, let us consider the environment $\Tu_s$ defined for all edge $e$ by: 
		\[\Tu_s(e) = \left\{
		\begin{array}{ll}
			T(e) & \mbox{if } e \notin B_{2,s,N} \\
			T'(e) & \mbox{else.}
		\end{array}
		\right.\]
		Thus $\Tu_s$ and $T$ have the same distribution and on the event $\{T \in \cM(\l)\} \cap \{\Sun_\l(T)=s\}$, $\Tu=\Tu_s$. So, using this environment and writing with indicator functions the result of Lemma \ref{17. cas non borné, lemme pour les inégalités avec les indicatrices dans la modification.}, we get: 
		
		\[\1_{\{T \in \cM(\l) \}} \1_{\{\Sun_\l(T)=s\}} \1_{\{T' \in \cB^*(T)\}} \le \1_{\{\Tu_s \in \cM(\l-1) \setminus \cM(\l)\}} \1_{\bigcup_{s' \approx s} \{\Sde_\l(\Tu_s)=s'\}}, \] where $s' \approx s$ if $\|s-s'\|_1 \le 2r_3$.
		We compute the expectation on both sides. The right side yields \[\P \left( \Tu_s \in \cM(\l-1) \setminus \cM(\l), \bigcup_{s' \approx s} \{\Sde_\l(\Tu_s)=s'\} \right) = \P \left( T \in \cM(\l-1) \setminus \cM(\l), \bigcup_{s' \approx s} \{\Sde_\l(T)=s'\} \right). \]
		For the left side, we have \[\E \left[ \1_{\{T \in \cM(\l) \}} \1_{\{\Sun_\l(T)=s\}} \1_{\{T' \in \cB^*(T)\}} \right] = \E \left[\1_{\{T \in \cM(\l) \}} \1_{\{\Sun_\l(T)=s\}}  \E \left[ \left. \1_{\{T' \in \cB^*(T)\}} \right| T \right] \right]. \] 
		Since on the event $\{T \in \cM(\l) \} \cap \{\Sun_\l(T)=s\}$, we have $\P \left( T' \in \cB^*(T) | T \right) \ge \eta$, the left side is bounded from below by $\eta \P (T \in \cM(\l),\Sun_\l(T)=s)$. 
		Then, by summing on all $s \in \Z^d$ and writing $K$ a constant which bounds from above for all $s' \in \Z^d$ the number of vertices $s \in \Z^d$ such that $s' \approx s$, we get 
		\[\frac{\eta}{K} \P(T \in \cM(\l)) \le \P(T \in \cM(\l-1) \setminus \cM(\l)).\]
		Now, since $\cM(\l) \subset \cM(\l-1)$,
		
		\[ \P(T \in \cM(\l-1) \setminus \cM(\l))= \P \left( T \in \cM(\l-1)\right) - \P \left( T \in \cM(\l) \right).\]
		Thus, \[\P(T \in \cM(\l)) \le \lambda \P(T \in \cM(\l-1) ),\] where $\displaystyle \lambda=\frac{1}{1+\frac{\eta}{K}} \in (0,1)$ does not depend on $x$. 
		Hence, using $\P(T \in \cM(0)) =1$, we get by induction \[\P(T \in \cM(q)) \le \lambda^q.\]
	\end{proof} 
	
	\begin{proof}[Proof of Proposition \ref{17. Gros théorème à démontrer, un seul motif.}]
		Recall that $N$, $x$ (and then $n$ and $p$) are fixed at \eqref{constantes N, n et x fixées.} but that $\Ca$, $\Da$ and $\lambda$ does not depend on $x$, $n$ and $p$. Then, by Lemma \ref{17. Majoration de la proba du complémentaire de G.} and Lemma \ref{17. Majoration par lambda puissance q}, using the inequality \eqref{inégalité preuve de la proposition principale.},
		
		\begin{align*}
			\P \left( \mbox{there exists a geodesic $\gamma$ from $0$ to $x$ such that } N^\mathfrak{P}(\gamma)=0 \right) & \le \P(T \in \cG^p(N)^c) + \P(T \in \cM(q)) \\
			& \le  \Da \mathrm{e}^{-\Ca p^{\frac{1}{d}} }+\lambda^{\left\lfloor \frac{p}{2} \right\rfloor}. \\
		\end{align*}
		As $C_1>0$ and $\lambda \in (0,1)$, and as this inequality holds for any $n \ge 2rN$ and any $x \in \Gamma_n$, we get the existence of two constants  $C>0$ and $D>0$ such that for all $n$, for all $x \in \Gamma_n$, 
		\[\P\left(  \mbox{there exists a geodesic $\gamma$ from $0$ to $x$ such that } N^\mathfrak{P}(\gamma)=0 \right) \le D \exp(-C n^{\frac{1}{d}}).\]
	\end{proof}
	
	\subsection{Typical boxes crossed by geodesics}\label{Typical boxes crossed, cas non borné}
	
	Let us first begin with the proof of the lemma stated in the paragraph of typical boxes in Section \ref{Sous-section preuve, cas non borné}.
	
	\begin{proof}[Proof of Lemma \ref{Gros lemme typical boxes}]
		\,
		\begin{enumerate}
			\item Let $\Bts$ be a typical box. Then the event $\cT(s,N)$ occurs. Let $u_0$ and $v_0$ be two vertices in $\Bds$. We have \[\tts(u_0,v_0) \le 2 \sup_{z \in \Bds} \tts (Ns,z)  \le 2 r_{2,3}N. \]
			
			Let $\pi_0$ be a path from $u_0$ to $v_0$ which is not entirely contained in $\Bts$. Let $z_0$ denote the first vertex on the boundary of $\Bts$ visited by $\pi_0$. Then 
			\begin{align*}
				T(\pi_0) & \ge \tts (u_0,z_0) \ge \tts(z_0,Ns)-\tts(u_0,Ns) \\
				& \ge \inf_{z \in \partial \Bts} \tts (Ns,z) - \sup_{z \in \Bds} \tts (Ns,z) \ge 3 r_{2,3} N > 2 r_{2,3} N \ge \tts(u_0,v_0).
			\end{align*}
			
			Hence, every geodesic from $u_0$ to $v_0$ has to be entirely contained in $\Bts$. 
			\item The properties $(ii)$ and $(iii)$ only depend on the time of edges in $\Bds$. The event $\cT(s;N)$ only depends on edges in $\Bts$ by the definition of $\tts$. 
			\item First, \[\lim\limits_{N \to \infty} \P(\cT(0;N)) =1.\]
			
			Indeed, by \eqref{17. Théorème de forme asymptotique.}, 
			\[\P \left( B_\mu \left(0, \frac{r_{2,3}}{2} N \right) \subset \tilde{B}(0,r_{2,3} N) \mbox{ for all large $N$} \right) = 1, \] and 
			\[\P \left( \tilde{B}(0,4 r_{2,3} N) \subset B_\mu (0,8r_{2,3}N) \mbox{ for all large $N$} \right) = 1. \]
			
			Thus, since $\displaystyle B_{2,0,1} \subset B_\mu \left(0,\frac{r_{2,3}}{2} \right) \cap \Z^d$ and $\displaystyle B_\mu (0,9 r_{2,3}) \cap \Z^d \subset B_{3,0,1}$, almost surely there exists $N_0 \in \N^*$ such that for all $N \ge N_0$, \[ B_{2,0,N} \subset B(0,r_{2,3} N), \, B(0,4 r_{2,3}N) \subset B_{3,0,N} \mbox{ and for all } y \in \partial B_{3,0,N}, \, y \notin B(0,4r_{2,3}N).\] So, for all $N \ge N_0$, 
			\[\sup_{z \in \partial B_{2,0,N}} t_{3,0,N} (0,z) \le r_{2,3} N \quad \mbox{ and } \inf_{z \in \partial B_{3,0,N}} t_{3,0,N} (0,z) \ge 4 r_{2,3} N.\] Note that, for the first inequality, we use the fact that for all $z \in B_{2,0,N}$, $t_{3,0,N}(0,z)=t(0,z)$ thanks to the first point of Lemma \ref{Gros lemme typical boxes} proved above. 
			
			The probability that $(iii)$ is satisfied by $B_{3,0,N}$ goes to $1$ by \eqref{17. Définition de nu(N).}. Then, let us prove that the probability that $(ii)$ is satisfied by $B_{3,0,N}$ goes to $1$. Let $|B_{3,0,N}|$ denote the number of vertices in $B_{3,0,N}$ and $\Pi_0$ denote the set of self-avoiding paths entirely contained in $B_{3,0,N}$. Then, using \eqref{17. r2 cas non borné}, we have that $r_2>r_1$, and by \eqref{17. définition de delta},
			
			\begin{align*}
				& \P( B_{3,0,N} \mbox{ does not satisfy } (ii) ) \\ \le & \sum_{\substack{u_\pi,v_\pi \in B_{3,0,N} \\ \|u_\pi-v_\pi\|_1 \ge (r_2-r_1) N}} \P \left( \mbox{\eqref{17. chemin pas anormalement courts version non bornée} is not satisfied by a path of $\Pi_0$ whose endpoints are $u_\pi$ and $v_\pi$ } \right) \\
				\le & \sum_{\substack{u_\pi,v_\pi \in B_{3,0,N} \\ \|u_\pi-v_\pi\|_1 \ge (r_2-r_1) N}} \P \left( \mbox{\eqref{17. chemin pas anormalement courts version non bornée} is not satisfied by a path whose endpoints are $u_\pi$ and $v_\pi$ } \right) \\
				\le & |B_{3,0,N}|^2  \mathrm{e}^{-D_0 (r_2-r_1) N} \xrightarrow[N \to \infty]{} 0, 
			\end{align*}
			since $|B_{3,0,N}|$ is bounded by a polynomial in $N$. 
		\end{enumerate}
	\end{proof}
	
	\begin{proof}[Proof of Lemma \ref{17. Majoration de la proba du complémentaire de G.}.]
		To begin this proof, one need an upper bound on the Euclidean length of geodesics. Using Theorem 4.6 in \cite{50years}, we have two positive constants $\Kb$ and $\Cb$ such that for all $y \in \Z^d$, \[\P \left( m(y) \ge \Kb \|y\|_1 \right) \le \mathrm{e} ^{-\Cb \|y\|_1^{\frac{1}{d}}},  \] where $m(y) = \max \left\{ | \sigma |_e \, : \, \sigma \mbox{ is a geodesic from $0$ to $y$} \right\}$ and where for a path $\sigma$, $|\sigma|_e$ means the number of different edges taken by $\sigma$. For all $p \in \N^*$, we define the event $\cN^p(N)$ on which every geodesic from $0$ to the outer sphere of the $p$-th annulus takes less than $\Kb prN$ distinct edges. Note that $r=2(r_1+r_3+1)$ is fixed at \eqref{on fixe r.} and $rN$ corresponds to the widths of the annuli. Then, 
		\[\P(\cN^p(N)^c) \le \sum_{y \, : \, \|y\|_1 = prN} \P \left(m(y) \ge \Kb \| y \|_1 \right) \le (2prN+1)^d \mathrm{e}^{-\Cb \left( prN \right)^\frac{1}{d}}. \]
		Hence, we obtain two positive constants $\Cd$ and $\Db$ only depending on $r$, $d$ and $F$ such that for all $p \in \N^*$, for all $N \in \N^*$, \[\P(\cN^p(N)^c) \le \Db \mathrm{e}^{-\Cd p^\frac{1}{d}}. \]
		Now, we assume that the event $\cN^p(N) \cap \cG^p(N)^c$ occurs. So, every geodesic from $0$ to the outer sphere of the $p$-th annulus takes a number of distinct edges which is between $prN$ and $\Kb prN$. Let us consider a re-normalized model. We introduce the meta-cubes \[B^\infty_{s,N}= \left\{w \in \Z^d \, : \, \left(s-\frac{1}{2} \right)N \le w < \left(s+\frac{1}{2} \right)N \right\}, \, \mbox{ for all $s \in \Z^d$},\] (where $v \le w$ means $v_i \le w_i$ for $1 \le i \le d$ and $v < w$ means $v_i < w_i$ for $1 \le i \le d$.)
		These meta-cubes form a partition of $\Z^d$. Furthermore, the meta-cubes and the boxes defined above have the same centers (which are the vertices $Ns$ for $s \in \Z^d$), and for all $s \in \Z^d$, $B^\infty_{s,N} \subset \Bus$. So, we can define typical meta-cubes. A meta-cube $B^\infty_{s,N}$ is typical if $\Bts$ is a typical box. 
		
		For a geodesic $\gamma$ from $0$ to the outer sphere of the $p$-th annulus, we associate the set of meta-cubes visited by $\gamma$, that is \[\mathfrak{A}(\gamma)=\{B^\infty_{s,N} \, | \, \gamma \text{ visits at least one vertex of $B^\infty_{s,N}$}\}.\]
		This set can be identify with the subset of the re-normalized graph $N\Z^d$: \[\mathfrak{A}^R_v(\gamma)=\{sN \, | \, B^\infty_{s,N} \in \mathfrak{A}(\gamma)\}.\]
		Note that, if we consider the set $\mathfrak{A}^R_e(\gamma)$ of edges of $N\Z^d$ linking vertices which are both in $\mathfrak{A}^R_v(\gamma)$, then the pair of sets $(\mathfrak{A}^R_v(\gamma),\mathfrak{A}^R_e(\gamma))$ forms a lattice animal, denoted by $\mathfrak{A}^R(\gamma)$. Recall that a lattice animal $\mathfrak{A}$ in $N\Z^d$ is a finite connected sub-graph of $N\Z^d$ that contains $0$. We denote by $\cA^R$ the set of lattice animals in $N\Z^d$ associated with a geodesic going from $0$ to the outer sphere of the $p$-th annulus.		
		
		Let us bound the size of these lattice animals. By the size of a lattice animal $\mathfrak{A}^R$, denoted by $|\mathfrak{A}^R|_v$, we mean its number of vertices in the re-normalized model. Recall that, since the event $\cN^p(N)$ occurs, every geodesic from $0$ to the outer sphere of the $p$-th annulus takes a number of distinct edges which is between $prN$ and $\Kb prN$. Then, in the meta-cube set $\fA(\gamma)$ associated to such a geodesic $\gamma$, since $r_1=d$ and thanks to the choice of $r$, there are $p-1$ meta-cubes associated to boxes crossed by $\gamma$ in distinct annuli. In particular (considering also the meta-cube centered at the origin), the size of every lattice animal $\fA^R \in \cA^R$ is bounded from below by $p$.
		For an upper bound, we use the inequality 
		\begin{equation}
			|\gamma|_e \ge N \left( \frac{|\fA^R(\gamma)|_v}{3^d}-1 \right), \label{Majoration de la taille des animaux.}
		\end{equation}
		for all geodesic $\gamma$ from $0$ to the outer sphere of the $p$-th annulus, where $|\gamma|_e$ still denotes the number of edges taken by $\gamma$ and where $|\fA^R(\gamma)|_v$ is the number of vertices of $\fA^R(\gamma)$. Let us prove this inequality. Let $\gamma$ be a geodesic from $0$ to the outer sphere of the $p$-th annulus and denote by $\gamma=(v_0,\dots,v_m)$ the sequence of vertices visited by $\gamma$. For all $v \in \Z^d$, denote by $s(v)$ the unique $s\in\Z^d$ such that $v$ belongs to $B^\infty_{s,N}$. We define by induction a strictly increasing sequence $i_0,\dots,i_{\kappa}$ by setting $\kappa=0$ and $i_0=0$ and then applying the following algorithm:
		\begin{enumerate}[label=(\alph*)]
			\item If there exists $i \in \{i_{\kappa}+1,\dots,m\}$ such that $s(v_i)$ is at distance at least $2$ for the norm $\|.\|_\infty$ from $s(v_{i_{\kappa}})$, we denote by $i_{\kappa+1}$ the smallest of these $i$, then we increment $\kappa$ and go back to (a).
			\item Otherwise we stop the algorithm.
		\end{enumerate}
		Then, we necessarily have $3^d(\kappa+1) \ge |\fA^R(\gamma)|_v$. Furthermore, for all $k\in\{0,\dots,\kappa-1\}$, we have $\|v_{i_{k+1}}-v_{i_k}\|_1 \ge N$. Hence, 
		\[|\gamma|_e \ge N \kappa \ge N \left( \frac{|\fA^R(\gamma)|_v}{3^d}-1 \right), \] and \eqref{Majoration de la taille des animaux.} is proved. 
		
		Now, using \eqref{Majoration de la taille des animaux.}, writing $\Kd = \lceil 3^d(\Kb r + 1)\rceil$ (which does not depend on $p$ and $N$), for every lattice animal $\fA^R \in \cA^R$, $|\fA^R|_v$ is bounded from above by $\Kd p$.
		Furthermore, for $j \in \{p,\dots,\Kd p\}$, using (4.24) in \cite{Grimmett}, we have that
		\begin{equation}
			\left| \left\{  \fA^R \in \cA^R \, : \, |\fA^R|_v=j \right\} \right| \le |\{\text{lattice animals in $\Z^d$ of size $j$}\}| \le 7^{dj}. \label{inégalité Grimmett}
		\end{equation}
		Let us consider the random variables $(X^N_\l)_{\l \in \Z^d}$ such that $X^N_\l=1$ if the meta-cube $B^\infty_{\l,N}$ is typical and $X^N_\l=0$ otherwise. By Lemma \ref{Gros lemme typical boxes}, there exists a positive constant $\Kf$ such that $X^N_\l$ is independent from the sigma-algebra generated by $\{X^N_k, \, k \in \Z^d \, : \, \|k-\l\|_1\ge \Kf\}$. Furthermore, also by Lemma \ref{Gros lemme typical boxes}, \[\lim\limits_{N\to\infty} \P(X^N_\l =1)=1.\] Thus, by Corollary 1.4 in \cite{LSS}, there exists $\eta_1=\eta_1(N)>0$ such that \[\eta_1(N) \xrightarrow[N \to \infty]{} 0,\] and there exist i.i.d.\ random variables $(Y^N_\l)_{\l \in \Z^d}$ such that $(X^N_\l)_\l\ge (Y^N_\l)_\l$ and $Y^N_0$ has a Bernoulli distribution of parameter $(1-\eta_1(N))$. 
		Finally, 
		we have \[\P(\cN^p(N) \cap \cG^p(N)^c)  \le \P \left( \exists \, \fA^R \in \cA^R \mbox{ such that } p \le |\fA^R|_v \le \Kd p \text{ and } \sum_{\l \in \fA^R_v} X^N_{\l} \le |\fA^R|_v + 1 - \frac{p}{2} \right).\] Indeed, on $\cG^p(N)^c$, there exists a geodesic $\gamma$ from $0$ to the outer sphere of the $p$-th annulus which crosses a typical box in strictly less than $\displaystyle \left\lfloor \frac{p}{2} \right\rfloor$ annuli, and thus there are strictly more than $\displaystyle \left\lceil \frac{p}{2} \right\rceil -1$ annuli $A_{i,N}$ with $i>1$ such that $\gamma$ does not cross a typical box in them. Furthermore, there are $p-1$ meta-cubes in $\fA (\gamma)$ such that each of them is associated to a box crossed by $\gamma$ in one of the $p-1$ distinct annuli $A_{i,N}$ with $1<i \le p$. Thus, there are strictly more than $\displaystyle \left\lceil \frac{p}{2} \right\rceil -1$ of these specified meta-cubes which are not typical. Hence the number of typical meta-cubes in $\fA(\gamma)$ is strictly smaller than $|\fA^R|_v-\left\lceil\frac{p}{2}\right\rceil+1$. Then, using the random variables $(Y^N_\l)_{\l \in \Z^d}$,
		\begin{align}
			\P(\cN^p(N) \cap \cG^p(N)^c) &  \le \P \left( \exists \, \fA^R \in \cA^R \mbox{ such that } p \le |\fA^R|_v \le \Kd p \text{ and }  \sum_{\l \in \fA^R_v} Y^N_{\l} \le |\fA^R|_v - \left\lceil\frac{p}{2}\right\rceil + 1 \right) \nonumber \\
			& \le \sum_{p \le j \le \Kd p} |\{\text{lattice animals in $\Z^d$ of size $j$}\}| \P \left( \mathrm{binomial}(j,\eta_1) \ge \left\lceil \frac{p}{2} \right\rceil -1 \right) \nonumber \\
			& \le \sum_{p \le j \le \Kd p} 7^{dj} \P \left( \mathrm{binomial}(j,\eta_1) \ge \left\lceil \frac{p}{2} \right\rceil -1 \right) \text{ (by \eqref{inégalité Grimmett})} \nonumber \\
			& \le \Kd p 7^{d\Kd} \P \left( \mathrm{binomial}(\Kd p,\eta_1) \ge \left\lceil \frac{p}{2} \right\rceil -1 \right). \nonumber
		\end{align}
		Then, for $p \ge 4$ and $N$ large enough to have $\displaystyle \eta_1(N) < \frac{1}{4 \Kd}$, using a Chernov bound for the binomial distribution (see Section 2.2 in \cite{Boucheron}), we get 
		\[	\P \left( \mathrm{binomial}(\Kd p,\eta_1) \ge \left\lceil \frac{p}{2} \right\rceil -1 \right) \le \P \left( \mathrm{binomial}(\Kd p,\eta_1) \ge \frac{p}{4} \right) \le \exp \left( - \Kd p h_{\eta_1} \left( \frac{1}{4 \Kd} \right) \right), \]
		where for $x \in (\eta_1,1)$, \[ h_{\eta_1} \left( x \right) = \left( 1 - x \right) \ln \left( \frac{1-x}{1-\eta_1} \right) + x \ln \left( \frac{x}{\eta_1} \right).\]
		Thus, since we can take $\eta_1$ as small as we want by taking $N$ large enough,
		\begin{align*}
			\P(\cN^p(N) \cap \cG^p(N)^c) & \le \Kd p \left[ 7^{d\Kd} \exp\left(  - \Kd h_{\eta_1} \left( \frac{1}{4 \Kd} \right) \right) \right]^p \\
			& \le \Kd p \exp (-2p) \text{ for $N$ large enough} \\
			& \le \exp(- \Ce p).
		\end{align*}
		Finally, we have a constant $N_0$ such that for all $p \ge 4$, for all $N \ge N_0$, \[	\P(\cG^p(N)^c)  \le \P(\cN^p(N) \cap \cG^p(N)^c) + \P(\cN^p(N)^c) \le  \mathrm{e}^{-\Ce p} + \Db \mathrm{e}^{-\Cd p^\frac{1}{d}}.\]
		So, there exist two positive constants $\Ca$ and $\Da$ such that for all $p \ge 1$, for all $N \ge N_0$, \[\P(\cG^p(N)^c) \le \Da \mathrm{e}^{-\Ca p^\frac{1}{d}}. \]
	\end{proof}
	
	\subsection{Modification argument}\label{Sous-section modification argument, cas non borné}
	
	The aim of this subsection is to prove Lemma \ref{17. gros lemme modification, cas non borné.}. Let $\l \in \{1,\dots,q\}$. On $\{T \notin \cM(\l) \}$, we set $\Ep(T) = \emptyset$ and $\Em(T) = \emptyset$. Let $s$ be in $\Z^d$. We now define $\Ep$ and $\Em$ on the event $\{T \in \cM(\l) \}\cap \{\Sun_\l(T)=s\}$. So assume that this event occurs. On the event $\{T \in \cM(\l) \}$, $\Gamma^\l$ is not empty and thus there is a selected geodesic. We denote this selected geodesic by $\gamma$. 
	We define the entry point (resp. the exit point) of a self-avoiding path in a set of vertices as the first (resp. the last) vertex of this path belonging to this set. Let $u$ denote the entry point of $\gamma$ in $\Bds$ and $v$ the exit point.
	
	We call entry point and exit point of the pattern (centered at $0$) the endpoints denoted by $u^\Lambda$ and $v^\Lambda$ in the introduction. Note that, if a self-avoiding path takes the pattern, its entry and exit points in the set $B_\infty(0,\lll)$ are not necessarily the entry and exit points of the pattern (as it can visit the set before and after taking the pattern).
	
	Here, we want to put the pattern centered at $sN$. The vertex $s$ being fixed, we keep the notation $u^\Lambda$ and $v^\Lambda$ to designate the entry and the exit points of the pattern centered at $sN$.
	
	\paragraph*{Construction of $\pi$.}
	\,
	
	We have the following inclusions:
	\begin{itemize}
		\item $B_\infty(sN,\lll) \subset B_\infty(sN,\lll+3) \subset \Bus$ since $r_1 = d$ and $N \ge \lll+3$ (see \eqref{constantes N, n et x fixées.}),
		\item $\Bus \subset \Bds$ since $r_2 > r_1$ by \eqref{17. r2 cas non borné}.
	\end{itemize}
	
	For the modification, we need a path $\pi$, constructed in a deterministic way and satisfying several properties, whose existence is guaranteed by the following lemma. 
	
	\begin{lemma}\label{Lemme construction de pi dans le cas non borné.}
		We can construct a path $\pi$ in a deterministic way such that :
		\begin{enumerate}[label=(\roman*)]
			\item $\pi$ goes from $u$ to $u^\Lambda$ without visiting a vertex of $B_\infty(sN,\lll)$, then goes from $u^\Lambda$ to $v^\Lambda$ in a shortest way for the norm $\|.\|_1$ (and thus being contained in $B_\infty(sN,\lll)$) and then goes from $v^\Lambda$ to $v$ without visiting a vertex of $B_\infty(sN,\lll)$,
			\item $\pi$ is entirely contained in $\Bds$ and does not have vertices on the boundary of $\Bds$ except $u$ and $v$, 
			\item $\pi$ is self-avoiding,
			\item the length of $\pi_{u,u^\Lambda}  \cup \pi_{v^\Lambda,v}$ is bounded from above by $2r_2 N+\Ka$, where $\Ka$ is the number of edges in $B_\infty(0,\lll+3)$.
		\end{enumerate}
	\end{lemma}
	
	The proof of this lemma is given in Appendix \ref{Annexe construction de pi pour le cas non borné.} but the idea is to construct two paths, one from $u$ to $sN$ and the other from $sN$ to $v$ which minimize the distance for the norm $\|.\|_1$ and such that the only vertex belonging to both paths is $sN$. Then, we denote by $u_0$ the first vertex of $B_\infty(sN,\lll+3)$ visited by the path from $u$ to $sN$ and $v_0$ the last vertex of $B_\infty(sN,\lll+3)$ visited by the path from $sN$ to $v$. We construct two paths entirely contained in $B_\infty(sN,\lll+3)$ from $u_0$ to $u^\Lambda$ and from $v^\Lambda$ to $v_0$ which do not take vertices of $B_\infty(sN,\lll)$ except $u^\Lambda$ and $v^\Lambda$ and which have no vertices in common and we consider the concatenation of the path from $u$ to $u_0$, the one from $u_0$ to $u^\Lambda$, a path from $u^\Lambda$ to $v^\Lambda$ in a shortest way, the path from $v^\Lambda$ to $v_0$ and the one from $v_0$ to $v$ (see Figure \ref{Figure du chemin pi.}).
	
	Let $\pi$ be the path given by Lemma \ref{Lemme construction de pi dans le cas non borné.}.
	
	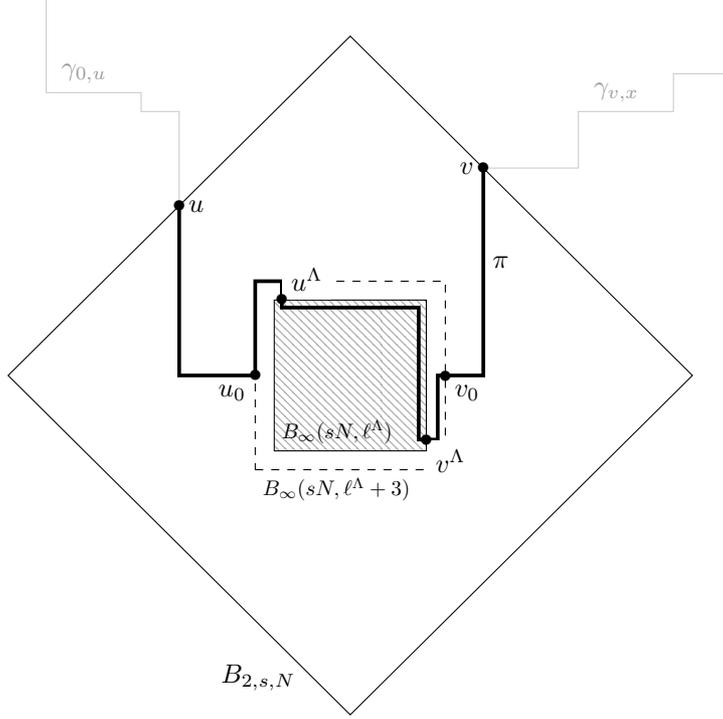
\begin{figure}
		\begin{center}
			\begin{tikzpicture}[scale=0.5]
				\draw (0,9) -- (9,0) -- (0,-9) -- (-9,0) -- cycle;
				\draw (-1.2,-8) node[left] {$\Bds$};
				\draw[pattern= north west lines, pattern color=gray!70] (-2,-2) rectangle (2,2);
				\draw (-2,-2) node[above right,scale=0.8] {$B_\infty(sN,\lll)$};
				\draw[dashed] (-2.5,-2.5) rectangle (2.5,2.5);
				\draw (-2.5,-2.5) node[below right,scale=0.8] {$B_\infty(sN,\lll+3)$};
				\draw[line width=1.4pt] (-4.5,4.5) -- (-4.5,0) -- (-2.5,0) -- (-2.5,2.5) -- (-1.8,2.5) -- (-1.8,2) -- (-1.8,1.8) -- (1.8,1.8) -- (1.8,-1.7) -- (2,-1.7) -- (2.3,-1.7) -- (2.3,0) -- (3.5,0) -- (3.5,5.5);
				\draw (-4.5,4.5) node[right] {$u$};
				\draw (3.5,5.5) node[left] {$v$};
				\draw (-2.5,0) node[below left] {$u_0$};
				\draw (2.5,0) node[below right] {$v_0$};
				\draw (-1.8,2) node[above right,fill=white] {$u^\Lambda$};
				\draw (2,-1.7) node[below right,fill=white] {$v^\Lambda$};
				\draw (3.5,3) node[right] {$\pi$};
				\draw[color=gray!40] (-9,10)--(-8,10) -- (-8,7.5) -- (-5.5,7.5) -- (-5.5,7) -- (-4.5,7) -- (-4.5,4.5);
				\draw[color=gray!80] (-7,7.5) node[above] {$\gamma_{0,u}$};
				\draw[color=gray!40] (3.5,5.5) -- (6,5.5) -- (6,7) -- (8.5,7) -- (8.5,8) -- (10,8);
				\draw[color=gray!80] (7,7) node[above] {$\gamma_{v,x}$};
				\draw (-4.5,4.5) node {$\bullet$};
				\draw (3.5,5.5) node {$\bullet$};
				\draw (-2.5,0) node {$\bullet$};
				\draw (2.5,-0.02) node {$\bullet$};
				\draw (-1.8,2) node {$\bullet$};
				\draw (2,-1.72) node {$\bullet$};
			\end{tikzpicture}
			\caption{Example of the construction of $\pi$ in dimension $2$.}\label{Figure du chemin pi.}
		\end{center}
	\end{figure}
	
	\paragraph*{Definition of $\Ep$, $\Em$ and $\cB^*$.}
	
	Define $\Em(T)$ as the set of edges $e$ such that $e \in \pi \setminus B_\infty(sN,\lll)$ and $\Ep(T)$ as the set of edges which are in $\Bds$ but which are not in $B_\infty(sN,\lll) \cup \pi$. Recall that $\{T' \in \cB^*(T)\}$ is a shorthand for 
	\[\{\forall e \in \Ep(T), \, T'(e) \ge \nu(N), \, \forall e \in \Em(T), T'(e) \le r+\delta', \, \theta_{N \Sun_\l(T)} T' \in \cA^\Lambda\}.\]
	Fix $\eta=\tilde{p}^{|\Bds|}\P(T \in \cA^\Lambda)$, where $\tilde{p}=\min(F([\r,\r+\delta']),F([\nu(N),\ST]))$. Thus, $\eta$ only depends on $F$, the pattern and $N$ and we have
	\[\P \left( T' \in \cB^*(T) | T \right) \ge \tilde{p}^{|\Bds|}\P(T \in \cA^\Lambda)=\eta,\] 
	
	\paragraph*{Consequences of the modification.} 
	
	We denote by $\gu$ the path $\gamma_{0,u} \cup \pi \cup \gamma_{v,x}$. Note that $\gu$ is a self-avoiding path. 
	
	\begin{lemma}\label{17. Cas non borné, comparaison temps de gamma et de gamma étoile.}
		We have $\Tu(\gu) < T(\gamma)$.
	\end{lemma}
	
	\begin{proof}[Proof]
		We have that $\gamma_{u,v}$ visits at least one vertex in $\Bus$. Denote by $w$ the first of these vertices. Then, $\gamma_{u,w}$ and $\gamma_{w,v}$ are two geodesics, both between two vertices in $\Bds$. Using item 1 in Lemma \ref{Gros lemme typical boxes}, $\gamma_{u,w}$ and $\gamma_{w,v}$ are entirely contained in $\Bts$. Thus, since $\Bts$ is a typical box, using \eqref{17. chemin pas anormalement courts version non bornée} and the fact that $\|u-w\|_1 \ge (r_2-r_1)N$ and $\|v-w\|_1 \ge (r_2-r_1)N$, we have \[T(\gamma_{u,v}) \ge 2N(r_2-r_1)(\r+\delta).\] Then, by the construction of $\pi$ and of $\cB^*(T)$, \[\Tu(\pi) \le (2r_2N+\Ka)(\r+\delta') + \tau^\Lambda,\] where $\tau^\Lambda$ is fixed at \eqref{On fixe tau Lambda dans le cas non borné.}.
		Thus, \[T(\gamma) - \Tu(\gu) \ge 2N(r_2(\delta-\delta')-r_1(\r+\delta))- \Ka(\r+\delta') - \tau^\Lambda. \]
		By \eqref{17. inégalité delta', cas non bornée.} and since $2N \ge 1$, we get $T(\gamma) - \Tu(\gu) > 0$.
	\end{proof}
	
	\begin{lemma}\label{Conséquences modification, cas non borné.}
		Let $\ogu$ be a geodesic from $0$ to $x$ in the environment $\Tu$. Then $\ogu$ weakly crosses the box $\Bts$ and the first vertex of $\Bds$ visited by $\ogu$ is $u$ and the last is $v$.
		Furthermore, $\ogu$ takes the pattern in $B_\infty(sN,\lll)$, $\ogu_{u,u^\Lambda}=\pi_{u,u^\Lambda}$ and $\ogu_{v^\Lambda,v}=\pi_{v^\Lambda,v}$.
	\end{lemma}
	
	\begin{proof}[Proof]
		Let $\ogu$ be a geodesic from $0$ to $x$ in the environment	$\Tu$. By Lemma \ref{17. Cas non borné, comparaison temps de gamma et de gamma étoile.}, $\Tu(\ogu) < T(\ogu)$. Thus $\ogu$ takes an edge of $\Bds$ and by item $(iii)$ of the definition of a typical box and since there is no edge whose time has been modified outside $\Bds$, $\ogu$ can not take any edge of time greater than $\nu(N)$ in $\Bds$. 
		Indeed, assume that $\ogu$ takes an edge $e$ such that $\Tu(e) \ge \nu(N)$. Then, denoting by $\cE(\ogu)$ the edges of $\ogu$ and using \eqref{17. Définition de nu(N).},
		\begin{align*}
			\Tu(\ogu) & = \sum_{f \in \cE(\ogu)\cap\Bds} \Tu(f) + \sum_{f \in \cE(\ogu)\cap\Bds^c} \Tu(f) \ge \nu(N) + \sum_{f \in \cE(\ogu) \cap \Bds^c} T(f) \\
			& > \sum_{f \in \Bds} T(f) + \sum_{f \in \cE(\ogu)\cap\Bds} T(f) \ge \sum_{f \in \cE(\ogu)\cap\Bds} T(f) + \sum_{f \in \cE(\ogu)\cap\Bds^c} T(f) = T(\ogu),
		\end{align*}
		which is impossible. Hence, $\ogu$ has to take edges of $\pi$ or of the pattern and can not take other edges of $\Bds$. 
		
		Since $\pi$ does not visit any vertex on the boundary of $\Bds$ except $u$ and $v$, $\ogu$ has to visit $u$ and $v$ and to follow $\pi$ between $u$ and $u^\Lambda$ and between $v^\Lambda$ and $v$. If $\ogu_{u^\Lambda,v^\Lambda}$ leaves the pattern, it takes an edge whose time is greater than $\nu(N)$, which is impossible. So, $\ogu_{u^\Lambda,v^\Lambda}$ is a path entirely contained in $B_\infty(sN,\lll)$ and is optimal for the passage time since $\ogu$ is a geodesic. 
		
		To conclude, let us show that $u$ is visited by $\ogu$ before $v$.  Assume that it is not the case. Then, there exists $\ogu_1$ a geodesic from $0$ to $v$ and $\ogu_2$ a geodesic from $u$ to $x$ in the environment $\Tu$ which does not take any edge in $\Bds$. Thus, there are also geodesics in the environment $T$. Then, \[T(\ogu_1) + T(\ogu_2) \le t^*(0,x) < t(0,x) \mbox{ by Lemma \ref{17. Cas non borné, comparaison temps de gamma et de gamma étoile.}}.\] By concatenating $\gamma_{0,u}$ and $\ogu_2$, we obtain a path from $0$ to $x$. Thus, \[T(\gamma_{0,u})+T(\ogu_2) \ge t(0,x).\] So $T(\ogu_1) < T(\gamma_{0,u})$, which implies \[T(\ogu_1)+T(\gamma_{v,x})<T(\gamma_{0,u})+T(\gamma_{v,x}) \le t(0,x),\] which is impossible since $\ogu_1 \cup \gamma_{v,x}$ is a path from $0$ to $x$. 
	\end{proof}
	
	Using Lemma \ref{Conséquences modification, cas non borné.}, we can prove Lemma \ref{17. gros lemme modification, cas non borné.}. Indeed, by this previous lemma, every geodesic from $0$ to $x$ takes the pattern inside $\Bds$ and the first item holds. For the third item, one can check that the concatenation of $\gamma_{0,u}$, $\pi_{u,u^\Lambda}$, one of the optimal paths for the passage time between $u^\Lambda$ and $v^\Lambda$ entirely contained in $B_\infty(sN,\lll)$, $\pi_{v^\Lambda,v}$ and $\gamma_{v,x}$ gives a geodesic $\gu$ which is associated to $\gamma$ in $\Bts$ (with $s_1=u$ and $s_2=v$). Finally, let us prove the second item. If $\ogu$ is a geodesic from $0$ to $x$ in the environment $\Tu$, then $\ogu_{u,v}$ is contained in $\Bds$. Furthermore, \[T(\gamma_{0,u}) = \Tu(\gamma_{0,u}) \ge \Tu (\ogu_{0,u}) = T(\ogu_{0,u}) \ge T(\gamma_{0,u}), \] so $T(\ogu_{0,u})=T(\gamma_{0,u})$ and $\ogu_{0,u}$ is a geodesic in the environment $T$. Similarly, $\ogu_{v,x}$ is a geodesic in the environment $T$. Hence, we get a not necessarily self-avoiding optimal path for the passage time by considering $\pi'=\ogu_{0,u} \cup \gamma_{u,v} \cup \ogu_{v,x}$. We get a geodesic $\gamma'$ which satisfies the properties of this item by cutting the loops of $\pi'$ using a standard process\footnote{Note that $\ogu_{0,u}$, $\gamma_{u,v}$ and $\ogu_{v,x}$ are three self-avoiding paths and that $\ogu_{0,u}$ and $\ogu_{v,x}$ do not have vertices in common. So we can consider $s_1$ the last vertex belonging to both path $\ogu_{0,u}$ and $\gamma_{u,v}$ in the order in which they are visited by $\gamma_{u,v}$ and $s_2$ the first vertex belonging to both paths $\gamma_{s_1,v}$ and $\ogu_{v,x}$ in the order in which they are visited by $\gamma_{u,v}$. Thus, since $\gamma_{u,v}$ is entirely contained in $\Bts$, $s_1$ and $s_2$ are two vertices contained in $\Bts$ and we can take $\gamma'=\ogu_{0,s_1} \cup \gamma_{s_1,s_2} \cup \ogu_{s_2,x}$.}.
	
	\section{Bounded case}\label{Section cas borné.}
	
	In this section, we assume that the support of $F$ is bounded. In this case, Theorem \ref{17. Théorème à démontrer.} also follows from Proposition \ref{17. Gros théorème à démontrer, un seul motif.}. 
	Our proof of Proposition \ref{17. Gros théorème à démontrer, un seul motif.} still follows the strategy given in the preceding section, but the modification argument is more involved. Let $\mathfrak{P}=(\Lambda,u^\Lambda,v^\Lambda,\cA^\Lambda)$ be a valid pattern. We set $\ST=\sup (\text{support}(F))$. Remark that, because of \eqref{17. définition de delta}, we have $\r + \delta < \ST$. 
	
	\subsection{Oriented pattern}\label{Sous-section surmotif}
	
	The proof in the bounded case uses a modification argument in which we have to connect the pattern to a straight path in a given direction. It is convenient to show the feasibility of this construction before starting the modification. The following lemma, whose proof is in Appendix \ref{Annexe overlapping pattern}, shows that it is indeed feasible by proving that a pattern can be associated to $d$ patterns (with a larger size), each having endpoints aligned in a distinct direction, and each having the original pattern as a sub-pattern. By direction, we mean one of the $d$ directions of the canonical basis which is denoted by $\{\epsilon_1,\dots,\epsilon_d\}$. Recall $\displaystyle \Lambda = \prod_{i=1}^d \{0,\dots,L_i\}$. 
	
	\begin{lemma}\label{17. Lemme motif valable}
		There exists $\l_0>\max(L_1,\dots,L_d)$ such that the following holds. Set $\Lambda_0=\{-\l_0,\dots,\l_0\}^d$. For all $j \in \{1,\dots,d\}$, there exists a pattern $\mathfrak{P}^j=(\Lambda_0,-\l_0 \epsilon_j,\l_0 \epsilon_j,\cA^{\Lambda_0}_j)$ such that:
		\begin{itemize}
			\item $\P\left(\cA^{\Lambda_0}_j\right)$ is positive, 
			\item on $\cA^{\Lambda_0}_j$, any path from $-\l_0 \epsilon_j$ to $\l_0 \epsilon_j$ optimal for the passage time among the paths entirely inside $\Lambda_0$ contains a subpath from $u^\Lambda$ to $v^\Lambda$ entirely inside $\Lambda$,
			\item $\cA^{\Lambda_0}_j \subset \cA^\Lambda$.
		\end{itemize}
	\end{lemma}
	
	We fix $\l_0$, $\Lambda_0$ and the patterns $\mathfrak{P_j}$, for $j=1,\dots,d$, for the remaining of the proof. 
	By definition, for all $j \in \{1,\dots,d\}$,
	$N^\mathfrak{P}(\pi) \ge N^\mathfrak{P_j}(\pi),$
	and actually
	\[ N^\mathfrak{P}(\pi) \ge \sum_{x \in \Z^d} \1_{\{\text{there exists $j \in \{1,\dots,d\}$, }x \text{ satisfies the condition }(\pi ; \mathfrak{P_j})\}}. \] 
	From now on, we forget the original pattern and only consider the oriented patterns $\mathfrak{P_j}$ for all directions $j\in\{1,\dots,d\}$. We talk about oriented pattern when its orientation is specified and we simply say "pattern" when talking about one of the oriented patterns $\mathfrak{P_j}$. Consistently with these conventions and to lighten notations, we write $\lll$ instead of $\lll_0$, $\Lambda$ instead of $\Lambda_0$ and $\cA^{\Lambda}_j$ instead of $\cA^{\Lambda_0}_j$. 
	
	Now, several parameters related to the distribution $F$ have to be introduced. First, we fix a positive real $\nu$ such that :
	\begin{itemize}[label=\textbullet]
		\item $\r+\delta \le \nu \le \ST$,
		\item $F([\nu,+\infty))>0$,
		\item the event $\cA^\Lambda \cap \{\forall e \in \Lambda, \, T(e) \le \nu\}$ has a positive probability. 
	\end{itemize}
	Notice that, if $F$ has an atom, one could have $\nu=\ST$. 
	Even if it means replacing $\cA^{\Lambda}_j$ by $\cA^{\Lambda}_j \cap \{\forall e \in \Lambda, \, T(e) \le \nu\}$, we can assume that \[\cA^{\Lambda}_j \subset \{\forall e \in \Lambda, \, T(e) \le \nu\}. \] 
	Further, we set $\tau^\Lambda=2\lll\nu$, which can be interpreted as an upper bound for the passage time of an optimal path from $\lll \epsilon_j$ to $-\lll \epsilon_j$ on the event $\cA^{\Lambda}_j$. Finally, we denote by $T^\Lambda$ the constant $K^\Lambda (\ST-\r)$ where $K^\Lambda$ is the number of edges in an oriented pattern. We will use it as an upper bound for the time that a path can save using edges of a pattern after a modification.
	
	\subsection{Proof of Proposition \ref{17. Gros théorème à démontrer, un seul motif.} in the bounded case}\label{Sous-section preuve, cas borné}
	
	We keep the overall plan of the unbounded case. Unlike in the unbounded case, we cannot use edges of prohibitive time and thus the modification argument is more elaborate here. This section follows the structure of Section \ref{Sous-section preuve, cas non borné} but the one step modification is replaced by a two-steps modification. To this aim, we slightly change the structure of our boxes and our definition of typical boxes. Let us begin by fixing some constants.
	
	\paragraph*{Constants.}
	
	Note that we keep the notations introduced in Section \ref{Sous-section Some tools and notations.} and that $\tau^\Lambda$ and $T^\Lambda$ are fixed in Section \ref{Sous-section surmotif}.
	In the next two paragraphs, we introduce the constants used in the proof. Before looking at their definitions below, one can keep in mind that we fix $\epsilon \ll 1$ and $\displaystyle r_1 \ll \frac{1}{\epsilon} \ll \nabla \ll r_2 \ll r_3 \ll r_4$, where "$\ll$" means that the ratio is large enough and only depends on the dimension $d$ and on the distribution $F$. 
	\begin{itemize}[label=\textbullet]
		\item We fix $\displaystyle \delta'=\min \left(\frac{\delta}{4}, \frac{\delta}{1+d} \right)$.
		\item We fix $L_1$ given by Lemma \ref{Lemme construction de pi.} only depending on $d$ and $\lll$, and $L_2 = L_1+(10+d)\lll$.
		\item Using Theorem \ref{Théorème Andjel/Vares dans le cas borné.} with $M=\r+\delta$, we get two constants $\alpha > 0$ and $\Cc$, fixed for the rest of the proof, such that for all $n \in \N^*$ and $u,v \in \Z^d$ such that $\|u-v\|_1=n$, 
		\begin{equation}
			\P \left( \exists \mbox{ a geodesic $\OG$ from $u$ to $v$ such that } \sum_{e \in \OG} \1_{\{T(e) \ge \r + \delta \}} \le \alpha n \right) \le \mathrm{e}^{-\Cc n}. \label{on définit alpha.} 
		\end{equation} 
		\item Fix $\epsilon > 0$ such that \[\epsilon < \min \left( \frac{1}{11}, \, \frac{\delta}{24 C_\mu}\right). \]
		\item  Fix $\nabla$ such that
		\begin{equation}
			\nabla > \max \left( \frac{4(1 + \ST) C_\mu}{\epsilon c_\mu}, \, 6dL_2C_\mu, \, \frac{8 C_\mu T^\Lambda}{3 \delta}, \, 4C_\mu(2\ST+\tau^\Lambda) \right). \label{On fixe t0 ou nabla0.}
		\end{equation}  
		We give here other lower bounds for $\nabla$ that we need for the sequel and which are consequences of \eqref{On fixe t0 ou nabla0.}. 
		\begin{itemize}
			\item Using the fact that $c_\mu \le C_\mu$, we get $\displaystyle \nabla > \frac{4C_\mu}{\epsilon c_\mu}> \frac{1}{\epsilon}$.
			\item From the inequality $\displaystyle \nabla > \frac{4(1 + \ST) C_\mu}{\epsilon c_\mu}$, using the fact that $\displaystyle \epsilon < \frac{1}{11}$, we have $1-\epsilon > 1 - 3\epsilon > 1 - 10 \epsilon > \epsilon$ and then $\displaystyle \nabla > \frac{1 + 2 \ST}{1-\epsilon}$, $\displaystyle \nabla > \frac{4(1+\ST)}{1-10\epsilon}$ and $\displaystyle \nabla > \frac{3 + 2 \ST}{1-3\epsilon}$.
			\item Since $\displaystyle \epsilon < \frac{\delta}{24 C_\mu}$ and $\displaystyle \nabla > \frac{1}{\epsilon}$, we get $\displaystyle \nabla > \frac{24 C_\mu}{\delta}$.
			\item Finally, since $\displaystyle \delta-\delta' \ge \frac{3 \delta}{4}$, we have from $\displaystyle \nabla > \frac{8 C_\mu T^\Lambda}{3\delta}$ that $\displaystyle \nabla > \frac{2 C_\mu T^\Lambda}{\delta-\delta'}$.
		\end{itemize}
	\end{itemize}
	
	\paragraph*{Boxes.}
	
	With theses constants, we can now define boxes. For $i \in \{1,2,3,4\}$, as in the unbounded case, $B_{i,s,N}$ is the ball of radius $r_i$ for the norm $\|.\|_1$ centered at the point $sN$ with:
	
	\begin{itemize}[label=\textbullet]
		\item $r_1=d$,
		\item $r_2$ an integer such that
		\begin{equation}
			r_2 > \max \left( r_1 + \frac{2(\nabla+2)}{c_\mu}, r_1 + L_1 + \frac{3\nabla}{c_\mu} + \frac{2\ST(1+(1+d)\lll)}{\nu} \right), \label{on fixe r2, cas borné}
		\end{equation}
		\item $r_3$ an integer such that \[ r_3 > \frac{7r_2(4\ST+\alpha \delta)}{\alpha \delta}, \]
		Note that $r_3 \ge r_2 + 1$.
		\item $r_4$ an integer such that \begin{equation}
			r_4 > \frac{r_3(\r+\delta+\ST)}{\r+\delta}. \label{On fixe r4.}
		\end{equation}
		Note that $r_4 \ge r_3 + 1$.
	\end{itemize}
	We use the word "box" to talk about $\Bqs$. Recall that we denote by $\partial B_{i,s,N}$ the boundary of $B_{i,s,N}$, that is the set of points $z \in \Z^d$ such that $\|z-sN\|_1=r_iN$. 
	
	\paragraph*{Crossed boxes and weakly crossed boxes.} 
	
	We say that a path
	\begin{itemize}
		\item crosses a box $\Bqs$ if it visits a vertex in $\Bus$,
		\item weakly crosses a box $\Bqs$ if it visits a vertex in $\Bts$.
	\end{itemize}

	\paragraph*{Paths associated in a box.}
	
	We say that two paths $\gamma$ and $\gamma'$ from $0$ to the same vertex $x$ are associated in a box $\Bqs$ if there exist two distinct vertices $s_1$ and $s_2$ such that the following conditions hold:
	\begin{itemize}
		\item $\gamma$ and $\gamma'$ visit $s_1$ and $s_2$,
		\item $\gamma_{0,s_1}=\gamma'_{0,s_1}$,
		\item $\gamma_{s_1,s_2}$ and $\gamma'_{s_1,s_2}$ are entirely contained in $\Bqs$,
		\item $\gamma_{s_2,x}=\gamma'_{s_2,x}$.
	\end{itemize}
	In particular, these two paths coincide outside $\Bqs$. 
	
	\paragraph*{Typical boxes.}
	
	$\Bqs$ is called a typical box if it verifies the following properties:
	
	\begin{enumerate}[label=(\roman*)]
		\item every geodesic $\gamma_{u,v}$ from $u$ to $v$ entirely contained in $\Bts$ with $\|u-v\|_1 \ge N$ has at least $\alpha \|u-v\|_1$ edges whose time is greater than or equal to $\r+\delta$,
		\item every path $\pi$ from $u$ to $v$ entirely contained in $\Bqs$ with $\|u-v\|_1 \ge N$ has a passage time verifying: 
		\begin{equation}
			t(\pi) \ge (\r+\delta) \, \|u-v\|_1, \label{17. chemin pas anormalement courts}
		\end{equation}
		\item for all $u$ and $v$ in $\Bts$, we have \[(1-\epsilon)\mu(u-v)- N \le t(u,v) \le (1+\epsilon)\mu(u-v)+ N.\]
	\end{enumerate}
	
	As in the unbounded case, we need properties which are guaranteed with the definition of typical boxes. We state them in the following lemma whose proof is given in Section \ref{Sous-section Properties of a typical box, cas borné.}. 
	
	\begin{lemma}\label{17. Propriétés typical box cas borné.}
		We have these three properties about typical boxes. 
		\begin{enumerate}
			\item If $\Bqs$ is a typical box, for all points $u_0$ and $v_0$ in $\Bts$, every geodesic from $u_0$ to $v_0$ is entirely contained in $\Bqs$. 
			\item The typical box property only depends on the time of the edges in $\Bqs$.
			\item We have \[ \lim\limits_{N\to \infty} \P \left( B_{4,0,N} \mbox{ is a typical box} \right) =1. \]
		\end{enumerate}
	\end{lemma}
	
	\paragraph*{Successful boxes.}
	
	For a fixed $x \in \Z^d$, a box $\Bqs$ is successful if every geodesic from $0$ to $x$ takes a pattern which is entirely contained in $\Bds$, i.e.\ if for every geodesic $\gamma$ going from $0$ to $x$, there exist $j \in \{1,\dots,d\}$ and $x_\gamma \in \Z^d$ satisfying the condition $(\gamma;\mathfrak{P}^j)$ such that $B_\infty(x_\gamma,\lll)$ is contained in $\Bds$. 
	
	\paragraph*{Annuli.}
	
	Following the proof in the unbounded case, we define the annuli $A_{i,N}$ with $r = 2(r_1+r_4+1)$ and $\cG^p(N)$ as in Section \ref{Sous-section preuve, cas non borné} but with the definitions of crossed and typical boxes defined here in Section \ref{Sous-section preuve, cas borné}. The bound on $\P(\cG^p(N)^c)$ of Lemma \ref{17. Majoration de la proba du complémentaire de G.} also holds here. The proof is exactly the same in this case thanks to Lemma \ref{17. Propriétés typical box cas borné.}. For the rest of the proof, we fix $C_1$, $D_1$ and $N_0$ given by Lemma \ref{17. Majoration de la proba du complémentaire de G.}.
	
	\paragraph*{Modification argument.}
	
	Fix $\Ki = T^\Lambda+2(C_\mu L_1+\ST(\lll+1))$. Then, fix 
	\begin{equation}
		N > \max \left( N_0, \frac{12 C_\mu \Ki}{\delta \nabla} \right), \, n \ge 2rN \text{ and } x \in \Gamma_n, \label{On fixe N, n et x dans le cas borné.}
	\end{equation} (where $\Gamma_n$ is defined at \eqref{définition Gamman}). Fix $\displaystyle p= \left\lfloor \frac{n}{rN} \right\rfloor$ and $\displaystyle q = \left\lfloor \frac{p}{2} \right\rfloor$. For $j \in \{1,\dots,q\}$, we define $\Gamma^j$, $\Sun_j$, $\Sde_j$ and $\cM(j)$ as in Section \ref{Sous-section preuve, cas non borné} but with the notions of typical and successful boxes defined here in Section \ref{Sous-section preuve, cas borné}. As in the unbounded case, the aim is to bound from above $\P(T \in \cM(q))$ independently of $x$. For the sequel, we use a two-steps modification. So we introduce two independent copies $T'$ and $T''$ of the environment $T$, the three being defined on the same probability space.
	
	Fix $\l \in \{1,...,q\}$. On $\{ T \in \cM(\l) \}$, $B_{4,\Sun_\l(T),N}$ is a typical box crossed by the selected geodesic. From this configuration, as a first step, we shall associate a set of edges $\Ep(T)$ which is contained in $B_{3,\Sun_\l(T),N} \setminus B_{2,\Sun_\l(T),N}$. It corresponds to the edges for which we want to reduce the time. Then, we get a new environment $\Tu$ defined for all edge $e$ by:
	\begin{equation}
		\Tu(e) = \left\{
		\begin{array}{ll}
			T'(e) & \mbox{if } e \in \Ep(T) \\
			T(e) & \mbox{else.}
		\end{array}
		\right.	\label{Définition de T*, cas borné.}
	\end{equation}
	
	From this environment, as a second step, we get three new subsets $\Eep(T,T')$, $\Eem(T,T')$ and $\EeM(T,T')$ of edges of $B_{2,\Sun_\l(T),N}$ which are respectively the edges for which we want to reduce the time, to increase the time and the edges of the location where we want to put the pattern. We get a third environment $\Td$ defined for all edge $e$ by: 
	\begin{equation}
		\Td(e) = \left\{
		\begin{array}{ll}
			T''(e) & \mbox{if } e \in \Eep(T,T') \cup \Eem(T,T') \cup \EeM(T,T') \\
			\Tu(e) & \mbox{else.}
		\end{array}
		\right. \label{Définition de T**, cas borné.}
	\end{equation}
	Note that $\Td$ and $T$ do not have the same distribution. For $y$ and $z$ in $\Z^d$, we denote by $t^*(y,z)$ (resp. $\td(y,z)$) the geodesic time between $y$ and $z$ in the environment $\Tu$ (resp. $\Td$). Similarly, we define for $c \in \Z^d$ and $t \in \R_+$:
	\begin{equation}
		B^*(c,t) = \{ u \in \Z^d \, : \, t^*(c,u) \le t \} \mbox{ and } \Bd(c,t) = \{ u \in \Z^d \, : \, \td(c,u) \le t \}. \label{Définition des boules des temps dans les environnements Tu et Td.}
	\end{equation}
	We formalize this modification in the next lemma and we will describe precisely the construction of $\Ep$, $\Eep$, $\Eem$ and $\EeM$ in the next subsection.
	
	\begin{lemma}\label{17. ancien lemme 12.9.}
		There exists $\eta = \eta (N)$ such that for all $\l$ in $\{1,\dots,q\}$, there exist measurable functions $\Ep \, : \, (\R_+)^\cE \mapsto \cP(\cE)$, $\Eep \, : \, (\R^\cE_+)^2\mapsto \cP(\cE)$, $\Eem \, : \, (\R^\cE_+)^2 \mapsto \cP(\cE)$, $\EeM \, : \, (\R^\cE_+)^2 \mapsto \cP(\cE)$ and $\cO \; : \, (\R^\cE_+)^2 \mapsto \{1,\dots,d\}$ such that:
		\begin{enumerate}[label=(\roman*)]
			\item $\Ep(T)$, $\Eep(T,T')$, $\Eem(T,T')$ and $\EeM(T,T')$ are pairwise disjoint and are contained in $B_{3,\Sun_\l(T)}$, 
			\item on the event $\{T \in\cM(\l)\}$, $\P \left( \left. T' \in \cB^*(T) \right| T \right) \ge \eta$ and on the event $\{T \in\cM(\l)\} \cap \{T' \in \cB^*(T)\}$, $\P \left( \left. T'' \in \cB^{**}(T,T') \right| T,T' \right) \ge \eta$, where $\{T' \in \cB^*(T)\}$ is a shorthand for \[ \{\forall e \in \Ep(T), \, T'(e) \le \r+\delta' \}, \] and $\{T'' \in \cB^{**}(T,T')\}$ is a shorthand for
			\[ \left\{ \forall e \in \Eep(T,T'), \, T''(e) \le \r+\delta', \, \forall e \in \Eem(T,T'), \, T''(e) \ge \nu, \, \theta_{N \Sun_\l(T)} T'' \in \cA^{\Lambda}_{\cO (T,T')} \right\}, \]
			\item $\{T \in \cM(\l) \} \cap \{T' \in \cB^*(T) \} \cap \{T'' \in \cB^{**}(T,T')\} \subset \{\Td \in \left(\cM(\l-1) \setminus \cM(\l) \right) \}$ and $\|\Sde_\l(\Td)-\Sun_\l(T)\|_1 \le 2r_4$.
		\end{enumerate}
	\end{lemma}
	
	The proof of Lemma \ref{17. ancien lemme 12.9.} is left to the reader. It is the same as the proof of Lemma \ref{17. cas non borné, lemme pour les inégalités avec les indicatrices dans la modification.}, replacing the use of Lemma \ref{17. gros lemme modification, cas non borné.} by the following one.
	
	\begin{lemma}\label{17. gros lemme modification.}
		There exists $\eta = \eta (N)$ such that for all $\l$ in $\{1,\dots,q\}$, there exist measurable functions $\Ep \, : \, (\R_+)^\cE \mapsto \cP(\cE)$, $\Eep \, : \, (\R^\cE_+)^2 \mapsto \cP(\cE)$, $\Eem \, : \, (\R^\cE_+)^2 \mapsto \cP(\cE)$, $\EeM \, : \, (\R^\cE_+)^2 \mapsto \cP(\cE)$ and $\cO \; : \, (\R^\cE_+)^2 \mapsto \{1,\dots,d\}$ such that items $(i)$ and $(ii)$ of Lemma \ref{17. ancien lemme 12.9.} are satisfied and such that if the event $\{T \in \cM(\l) \} \cap \{T' \in \cB^*(T) \} \cap \{T'' \in \cB^{**}(T,T')\}$ occurs, then we have the following properties: 
		\begin{enumerate}[label=(\roman*)]
			\item  in the environment $\Td$, every geodesic from $0$ to $x$ takes the pattern inside $B_{2,\Sun_\l(T),N}$,
			\item for all geodesic $\ogd$ from $0$ to $x$ in the environment $\Td$, there exists a geodesic $\OG$ from $0$ to $x$ in the environment $T$ such that $\OG$ and $\ogd$ are associated in $B_{4,\Sun_\l(T),N}$,
			\item there exists a geodesic $\gd$ in the environment $\Td$ from $0$ to $x$ such that $\gd$ and the selected geodesic $\gamma$ in the environment $T$ are associated in $B_{4,\Sun_\l(T),N}$.
		\end{enumerate}
	\end{lemma}
	
	The proof of Lemma \ref{17. gros lemme modification.} is the aim of Section \ref{Sous-section modification argument, cas borné}. We now conclude the proof of Proposition \ref{17. Gros théorème à démontrer, un seul motif.} in the bounded case.  The following lemma is the counterpart of Lemma \ref{17. Majoration par lambda puissance q} in this case.
	
	\begin{lemma}\label{17. Majoration par lambda puissance q, cas borné}
		There exists $\lambda \in (0,1)$ which does not depend on $n$ and on $x \in \Gamma_n$ such that \[\P \left( T \in \cM(q) \right) \le \lambda^q. \]
	\end{lemma}
	
	\begin{proof}[Proof]
		Let $\l$ be in  $\{1,...,q\}$. For every $s\in\Z^d$ and $\ed$ subset of edges of $\Bts$, let us consider the environment $\Td_{s,\ed}$ defined for all edge $e$ by:
		\begin{equation*}
			\Td_{s,\ed} = \left\{
			\begin{array}{ll}
				T''(e) & \mbox{if } e \in  \ed \cap \Bds \\
				T'(e) & \mbox{if } e \in \ed \cap (\Bts \setminus \Bds) \\
				T(e) & \mbox{else.}
			\end{array}
			\right. 
		\end{equation*}
		We define $\Ed(T,T')=\Ep(T) \cup \Eep(T,T') \cup \Eem(T,T') \cup \EeM(T,T')$. Thus, for every $s$ and $\ed$, $\Td_{s,\ed}$ and $T$ have the same distribution and on the event $\{T \in \cM(\l)\} \cap \{\Sun_\l(T)=s\} \cap \{\Ed(T,T')=\ed\}$, $\Td=\Td_{s,\ed}$.
		Using this environment and writing with indicator functions the result $(iii)$ of Lemma \ref{17. ancien lemme 12.9.}, we have:
		\[ \1_{\{T \in\cM(\l)\}} \1_{\left\{\Sun_\l(T)=s\right\}} \1_{\{T' \in \cB^*(T)\}} \1_{\left\{\Ed(T,T')=\ed\right\}} \1_{\{T'' \in \cB^{**}(T,T')\}} \le \1_{\left\{\Td_{s,\ed} \in \cM(\l-1) \setminus \cM(\l) \right\}} \1_{\bigcup_{s' \approx s} \left\{\Sde_\l(\Td_{s,\ed})=s'\right\}}.\]
		We take the expectation on both sides. The right side yields
		\[\P \left( \Td_{s,\ed} \in \cM(\l-1) \setminus \cM(\l), \bigcup_{s' \approx s} \{\Sde_\l(\Td_{s,\ed})=s'\} \right) = \P \left( T \in \cM(\l-1) \setminus \cM(\l), \bigcup_{s' \approx s} \{\Sde_\l(T)=s'\} \right). \]
		For the left side, we have 
		\begin{align*}
			& \E \left[ \1_{\{T \in\cM(\l)\}} \1_{\left\{\Sun_\l(T)=s\right\}} \1_{\{T' \in \cB^*(T)\}} \1_{\left\{\Ed(T,T')=\ed\right\}} \1_{\{T'' \in \cB^{**}(T,T')\}} \right] \\
			= & \E \left[ \E \left[ \left. \E \left[ \left. \1_{\{T \in\cM(\l)\}} \1_{\left\{\Sun_\l(T)=s\right\}} \1_{\{T' \in \cB^*(T)\}} \1_{\left\{\Ed(T,T')=\ed\right\}} \1_{\{T'' \in \cB^{**}(T,T')\}} \right| T,T' \right] \right| T \right] \right] \\
			= & \E \left[ \1_{\{T \in\cM(\l)\}} \1_{\left\{\Sun_\l(T)=s\right\}} \E \left[ \left. \1_{\{T' \in \cB^*(T)\}} \1_{\left\{\Ed(T,T')=\ed\right\}} \P \left( \left. T'' \in \cB^{**}(T,T') \right| T,T' \right) \right| T \right] \right].
		\end{align*}
		Since on the event $\{T \in\cM(\l)\} \cap \left\{\Sun_\l(T)=s\right\} \cap \{T' \in \cB^*(T)\} \cap \left\{\Ed(T,T')=\ed\right\}$, $\P \left( \left. T'' \in \cB^{**}(T,T') \right| T,T' \right)$ is bounded from below by $\eta$, the left side is bounded from below by \[\eta \E \left[ \1_{\{T \in\cM(\l)\}} \1_{\left\{\Sun_\l(T)=s\right\}} \E \left[ \left. \1_{\{T' \in \cB^*(T)\}} \1_{\left\{\Ed(T,T')=\ed\right\}} \right| T \right] \right]. \]
		Then, by summing on all $\ed$ subsets of edges of $\Bts$ and writing $K$ a constant which does not depend on $s$ and which bounds from above the number of different subsets of edges of $\Bts$, we get
		\[\frac{\eta}{K} \E \left[ \1_{\{T \in\cM(\l)\}} \1_{\left\{\Sun_\l(T)=s\right\}} \P \left( \left. T' \in \cB^*(T) \right| T \right) \right] \le \P \left( T \in \cM(\l-1) \setminus \cM(\l), \bigcup_{s' \approx s} \{\Sde_\l(T)=s'\} \right).\]
		Now, on the event $\{T \in\cM(\l)\} \cap \{\Sun_\l(T)=s\}$, $\P \left( \left. T' \in \cB^*(T) \right| T \right)$ is bounded from below by $\eta$, so the left side is bounded from below by \[\frac{\eta^2}{K} \P \left( T \in\cM(\l), \Sun_\l(T)=s  \right) . \]
		Then, by summing on all $s \in \Z^d$ and writing $K'$ a constant which bounds from above for all $s \in \Z^d$ the number of vertices $s' \in \Z^d$ such that $s' \approx s$, we get 
		\[\frac{\eta^2}{KK'} \P(T \in \cM(\l)) \le \P(T \in \cM(\l-1) \setminus \cM(\l)).\]
		Now, since $\cM(\l) \subset \cM(\l-1)$,
		
		\[ \P(T \in \cM(\l-1) \setminus \cM(\l))= \P \left( T \in \cM(\l-1)\right) - \P \left( T \in \cM(\l) \right).\]
		Thus, \[\P(T \in \cM(\l)) \le \lambda \P(T \in \cM(\l-1) ),\] where $\displaystyle \lambda=\frac{1}{1+\frac{\eta^2}{KK'}} \in (0,1)$ does not depend on $x$. 
		Hence, using $\P(T \in \cM(0)) =1$, we get by induction \[\P(T \in \cM(q)) \le \lambda^q.\]
	\end{proof}
	
	From Lemma \ref{17. Majoration par lambda puissance q, cas borné}, the proof of Proposition \ref{17. Gros théorème à démontrer, un seul motif.} is the same as in the unbounded case.
	
	\subsection{Properties of a typical box}\label{Sous-section Properties of a typical box, cas borné.}
	
	In this section, we state and prove the following lemma, which gives us properties of typical boxes useful for the modification argument, and the proof of Lemma \ref{17. Propriétés typical box cas borné.}.
	
	\begin{lemma}\label{17. lemme technique partie argument de modification}
		If $\Bqs$ is a typical box, we have the following properties.
		\begin{enumerate}[label=(\roman*)]
			\item For all $u$ and $v$ in $\Bts$ with $\|u-v\|_1 \ge \Kh N$ where $\displaystyle \Kh = \frac{1}{\epsilon c_\mu} >0$, \[ (1-2\epsilon) \mu (u-v) \le t(u,v) \le (1+2\epsilon) \mu(u-v). \]
			\item For all $z \in \Bts$, for all $\rr > 0$, for all $N \in \N^*$, \[\Bts \cap B(z,N\rr) \subset \Bts \cap B_\mu \left(z,\frac{N(\rr+1)}{1-\epsilon} \right),\] and if $\displaystyle \rr \ge \frac{1}{\epsilon}-2$, \[\Bts \cap B_\mu \left(z,\frac{N(\rr+1)}{1-\epsilon} \right) \subset \Bts \cap B_\mu \left(z,\frac{N\rr}{1-2\epsilon} \right).\]
			\item For all $z \in \Bts$, for all $\rr > 0$, for all $N \in \N^*$, \[\Bts \cap B_\mu(z,N\rr) \subset \Bts \cap B \left(z,N((1+\epsilon)\rr + 1) \right), \] and if $\displaystyle \rr \ge \frac{1}{\epsilon}$, \[\Bts \cap B \left(z,N((1+\epsilon)\rr + 1) \right) \subset \Bts \cap B \left(z,(1+2\epsilon)N\rr \right). \]
		\end{enumerate}
	\end{lemma}
	
	\begin{proof}[Proof]
		\begin{enumerate}[label=(\roman*)]
			\item Let $u$ and $v$ be in $\Bts$ with $\|u-v\|_1 \ge \Kh N$. Then, since $\Bqs$ is a typical box, we have \[(1-\epsilon)\mu(u-v) - N \le t(u,v) \le (1+\epsilon)\mu(u-v)+N.\] The requirement on $u$ and $v$ implies \[\frac{N}{c_\mu} \le \epsilon \|u-v\|_1,\] so $\epsilon \mu(u-v) \ge N$. 
			\item and (iii) In both cases, it is easy to check that the first inclusion follows from property (iii) of a typical box and an easy computation using (i) shows the second inclusion. 
		\end{enumerate}
	\end{proof}

	\begin{proof}[Proof of Lemma \ref{17. Propriétés typical box cas borné.}] \,
		\begin{enumerate}
			\item Let $\Bqs$ be a typical box and $u_0$ and $v_0$ two points of $\Bts$. Then, taking paths minimizing the distance for the norm $\|.\|_1$ between $u_0$ and $sN$ and between $v_0$ and $sN$, we have $t(u_0,v_0) \le 2 r_3 N \ST$. Then, if a geodesic $\gamma_{u_0,v_0}$ takes an edge which is not in $\Bqs$, since $r_4 \ge r_3 + 1$, using the item $(ii)$ of the definition of a typical box leads to \[T(\gamma_{u_0,v_0}) \ge 2(r_4-r_3) N (\r+\delta), \] which is impossible since $\displaystyle r_4 > \frac{r_3(\r+\delta+\ST)}{\r+\delta}$. Note that we only use item (ii) of the definition of a typical box to prove this property.
			\item It is clear that the first property only depends on the time of edges in $\Bts$ and the second only depends on the time of edges in $\Bqs$. For the third property, by the preceding item, we know that for all points $w_1$ and $w_2$ in $\Bts$, the knowledge of the time of all edges in $\Bqs$ allows us to determine $t(w_1,w_2)$, so to know if the two inequalities are satisfied.  
			\item For each item of the definition of a typical box, we show that the probability that $B_{4,0,N}$ satisfies this item goes to $1$. To show that the probability that item $(ii)$ of the definition of a typical box is satisfied goes to $1$, we use the same proof as for $(ii)$ in the proof of Lemma \ref{Gros lemme typical boxes}, replacing $r_2-r_1$ by 1 and $\Bts$ by $\Bqs$.
			Further, using \eqref{on définit alpha.} and a similar computation as for item $(ii)$ in the proof of Lemma \ref{Gros lemme typical boxes}, we get: \[\P (B_{4,0,N} \mbox{ satisfies } (i)) \xrightarrow[N \to \infty]{} 1. \]
			
			To prove that \[\P (B_{4,0,N} \mbox{ satisfies } (iii)) \xrightarrow[N \to \infty]{} 1,\] recall that $\epsilon$ is fixed in Section \ref{Sous-section preuve, cas borné} and fix $\displaystyle \rho = \frac{1}{2d((1+\epsilon)C_\mu + \ST)}$. Let us consider the following property for $N$ large enough to have $\lfloor \rho N \rfloor \neq 0$:
			
			\begin{equation}
				\forall u',v' \in \lfloor \rho N \rfloor \Z^d \cap B_{3,0,N}, \, |t(u',v') - \mu(u'-v') | \le \epsilon \mu(u'-v'). \label{17. propriété trois bis anneaux typiques.}
			\end{equation}
			
			By \eqref{Autre version du théorème de forme asymptotique}, by stationarity, and since $\left| \left\lfloor \rho N \right\rfloor \Z^d \cap B_{3,0,N} \right|$ is uniformly bounded in $N$, we get  \[ \P \left(\eqref{17. propriété trois bis anneaux typiques.} \mbox{ holds} \right) \xrightarrow[N \to \infty]{} 1.\]
			
			Finally, the proof is completed by showing that \eqref{17. propriété trois bis anneaux typiques.} implies that $B_{4,0,N}$ satisfies $(iii)$. Assume \eqref{17. propriété trois bis anneaux typiques.} and let $u$ and $v$ be two vertices in $B_{3,0,N}$. The aim is to show 
			\begin{equation}
				|t(u,v)-\mu(u-v)| \le \epsilon \mu(u-v)+N. \label{Inégalité preuve section boîtes typiques, cas borné.}
			\end{equation}
			Let $u',v' \in \left\lfloor \rho N \right\rfloor \Z^d \cap B_{3,0,N}$ such that $\|u-u'\|_1 \le d\rho N$ and $\|u-u'\|_1 \le d\rho N$. Thus,
			\[|t(u,v)-\mu(u-v)| \le |t(u,v) - t(u',v')| + |t(u',v') - \mu(u'-v')| + |\mu(u'-v')-\mu(u-v)|. \]
			By \eqref{17. propriété trois bis anneaux typiques.}, $|t(u',v') - \mu(u'-v')| \le \epsilon \mu(u'-v')$. Furthermore,
			\[|\mu(u-v) - \mu(u'-v')| \le \mu(u-u') + \mu(v-v') \le (\|u-u'\|_1 + \|v-v'\|_1) C_\mu \le 2 d \rho N C_\mu.\]  
			Similarly \[|t(u,v) - t(u',v')| \le t(u,u') + t(v,v') \le (\|u-u'\|_1 + \|v-v'\|_1) \ST \le 2 d \rho N \ST,\]
			Thus, we get 
			\[|t(u,v)-\mu(u-v)| \le \epsilon\mu(u-v) + 2 d \rho N(\ST + (1+\epsilon)C_\mu). \]
			We get \eqref{Inégalité preuve section boîtes typiques, cas borné.} thanks to the choice of $\rho$.
		\end{enumerate}
	\end{proof}
	
	\subsection{Modification argument}\label{Sous-section modification argument, cas borné}
	
	In this section, we prove Lemma \ref{17. gros lemme modification.}.
	
	\subsubsection{Mains ideas of the proof}\label{Sous-section schéma de preuve dans le cas borné.}
	
	\paragraph{Framework.} 
	
	Before proving in detail Lemma \ref{17. gros lemme modification.}, we give the main ideas of the proof of the modification argument. We consider a geodesic $\gamma$ crossing a typical box. Recall that this box is composed of four concentric balls of different size, denoted\footnote{We denote $B_i$ instead of $B_{i,s,N}$ to lighten the notations in this subsection which is less formal than the proof.} by $B_1$, $B_2$, $B_3$ and $B_4$, and that $\gamma$ visits at least one vertex in $B_1$. Recall also that the radii of the balls satisfy $r_1 \ll r_2 \ll r_3 \ll r_4$. The aim is to modify the environment to get the following properties:
	\begin{enumerate}
		\item Every geodesic in the new environment takes the pattern in $B_2$.
		\item Every geodesic in the new environment is associated in $B_4$ with a geodesic in the first environment.
		\item The geodesic $\gamma$ is associated in $B_4$ with at least one geodesic in the new environment.
	\end{enumerate}
	
	\paragraph{First modification.} Let us begin with some notations.
	\begin{itemize}
		\item We denote by $\cD_\avant$ the set of edges $e$ verifying the following three conditions:
		\begin{itemize}
			\item $\gamma$ takes $e$ between its first entry in $B_3$ and its first entry in $B_2$.
			\item $e$ belongs to $B_3$ but $e$ does not belong to $B_2$.
			\item $T(e) > \r + \delta$.
		\end{itemize} 
		We denote by $s_1$ the first vertex in $\cD_\avant$ visited by $\gamma$.
		\item $\cD_\apres$ and $s_2$ are symmetrically defined. In particular, $s_2$ is the last vertex of $\cD_\apres$ visited by $\gamma$.
		\item We set $\cD = \cD_\avant \cup \cD_\apres$.
	\end{itemize}
	The first modification simply consists in reducing the passage times of edges in $\cD$ below $\r+\delta'$ (with $0 < \delta' < \delta$ properly chosen).
	This provides a localization of the geodesics in the new environment $\Tu$: they all take all edges of $\cD$ (see Lemma \ref{17. Première modification.} where $\cD$ is called $\Ep(T)$).
	Note also that $\gamma$ remains a geodesic in the environment $\Tu$. 
	One could easily deduce that any geodesic in this new environment visits $s_1$ and $s_2$ and then that properties $2$ and $3$ about associated geodesics hold. The point is that these properties will be preserved by the second modification whose influence on time is negligible with respect to the first modification.
	
	\paragraph{Second modification.} 
	
	\subparagraph*{Framework and overall plan.}
	The selected geodesic $\gamma$ visits at least one vertex of $B_1$, let us denote by $c_0$ the first of them.
	We denote by $u_1$ the first vertex of $\gamma$ such that $T(\gamma_{u_1,c_0}) \le \nabla N$ and $v_1$ the last vertex of $\gamma$ such that $T(\gamma_{c_0,v_1}) \le \nabla N$ (see \eqref{On introduit u1 et v1.}).
	Since the box is typical and since $r_1$ and $\nabla$ are small enough compared with $r_2$, all that we consider (here and in what follows) takes place in $B_2$ where the environments $T$ and $\Tu$ coincide (see Lemma \ref{17. minoration de mu(u1-v1) et arêtes de la zone rouges contenues dans B4.}).
	
	The rough plan is to modify the environment between $u_1$ and $v_1$:
	\begin{enumerate}
		\item We consider an oriented path $\pi$ from $u_1$ to $v_1$ and we make the passage times on its edges very small: this is our highway.
		\item We put the pattern somewhere on $\pi$.
		\item We make the passage times on the other edges in a certain neighborhood (to be defined) very large: these are our walls.
	\end{enumerate}
	
	\subparagraph*{Requirements for $\pi$.} 
	It is quite simple and it is stated in Lemma \ref{Lemme construction de pi.}. 
	We want it to be close to the segment $[u_1,v_1]$ of $\R^d$: this will allow us to have good estimates on the relevant times.
	We also require that the path contains towards its middle a sequence of steps in the same direction: this will allow us to place the pattern there.
	We call this sequence of steps the central segment of $\pi$.
	There is, however, a small difficulty: we can not choose the orientation of the central segment.
	That is why we need to be able to place the original pattern in an overlapping pattern of arbitrary orientation (this is the purpose of Lemma \ref{17. Lemme motif valable}).
	
	\subparagraph*{Requirements for the neighborhood of the walls.}
	This point is less simple.
	We do not want to put walls on edges of $\gamma \setminus \gamma_{u_1,v_1}$ but we also do not want that these edges affect our estimates.
	Ideally, we would like the only relevant edges, apart those from $\pi$ and the pattern-location, to be walls. 
	This can be done in a non technical way as follows.
	Set
	\[
	B^*(0) = B^*(0, T^*(0,u_1)) \text{ and } B^*(x) = B^*(x, T^*(x,v_1)).
	\]
	There are three types of edges:
	\begin{itemize}
		\item The "before" edges: the ones belonging to $B^*(0)$. 
		\item The "after" edges: the ones belonging to $B^*(x)$.
		\item The "intermediate" edges: all other edges.
	\end{itemize}
	Since $\gamma$ is a geodesic in the environment $\Tu$, $\gamma$ takes "before" edges, then "intermediate" edges and then "after" edges.
	Note that the edges of $\cD_\avant$ are "before" edges and that the edges of $\cD_\apres$ are "after" edges.
	
	We can then define the neighborhood on which we put walls: this is the set of "intermediate" edges of $B_2$ which do not belong to $\pi$ and to the pattern-location. 
	We call these edges the "wall" edges.
	The idea is that, as explained at \eqref{e:chronologie}, if $\overline\pi^{**}$ is a part of a geodesic from $0$ to $x$ in the environment $\Td$ linking two vertices outside $B^*(0)$ and $B^*(x)$, then 
	$\overline\pi^{**}$ only takes "intermediate" edges.
	Hence, if in addition $\overline\pi^{**}$ does not take edges of $\pi$ and is entirely contained in $B_2$, 
	then $\overline\pi^{**}$ only takes "wall" edges or edges of the pattern-location: this will ensure in the end that every geodesic takes the pattern.
	
	\subparagraph*{It makes sense.}
	Recall that $\gamma$ is still a geodesic in the environment $\Tu$ and that it visits $u_1$ and then $c_0$. 
	Using the definition of $B^*(0)$ and the triangle inequality, we get
	\[
	T^*(B^*(0),c_0) \ge T^*(u_1,c_0) = T(u_1,c_0) \approx \nabla N.
	\]
	In particular, the distance for the norm $\|\cdot\|_1$ between $B^*(0)$ and $c_0$ is at least of order $\nabla N$.
	The same applies to the distance for the norm $\|\cdot\|_1$ between $B^*(x)$ and $c_0$.
	Thus there is enough room for the central segment and the pattern since their size is of order $1$ (see Lemma \ref{17. Egalité des boules des deuxième et troisième configurations, et aucune arête du motif dans ces boules.}). 
	
	\subparagraph*{The second modification.}
	\begin{itemize}
		\item We put the pattern somewhere on the "central segment".
		\item With the exception of the edges of $\pi$ connected to $B^*(0)$ or $B^*(x)$, we put the time of every "intermediate" edge of $\pi$ which is not in the pattern below $\r+\delta'$.
		\item We put the time of every other "intermediate" edge of $B_2$ which does not belong to $\pi$ and whose time is lower than $\nu$ greater than or equal to $\nu$ (recall that in particular we take $\nu$ such that, in the pattern, the passage time of every edge is lower than or equal to $\nu$).
	\end{itemize}
	
	\subparagraph*{Consequences.}  In what follows, we denote by $\overline\gamma^{**}$ a geodesic in the environment $T^{**}$.
	\begin{enumerate}
		\item Since the passage times of the "before" and "after" edges have not been modified and since the passage times of the "intermediate" edges touching $B^*(0)$ or $B^*(x)$ have not been reduced, $B^*(0)$ remains a ball centered in $0$ for the passage times of the environment $\Td$ and similarly for $B^*(x)$ (see again Lemma \ref{17. Egalité des boules des deuxième et troisième configurations, et aucune arête du motif dans ces boules.}).
		Hence,
		\begin{equation}\label{e:chronologie}
			\overline\gamma^{**} \text{ takes "before" edges, then "intermediate" edges and then "after" edges.}
		\end{equation}
		\item The creation of the highway on $\pi$ makes the passage time from $0$ to $x$ lower (as stated in Lemma \ref{17. Temps gagné pendant la deuxième modification.}, it allows to save a time of order $\nabla N$; it can be seen by taking a geodesic from $0$ to the last vertex of $\pi$ belonging to $B^*(0)$, then following $\pi$, then taking a geodesic from the first vertex of $\pi$ belonging to $B^*(x)$ to $x$). Hence $\overline\gamma^{**}$ takes an edge whose time has been reduced during the second modification and thus
		\begin{equation}\label{e:pimotif}
			\begin{split}
				\overline\gamma^{**} \text{ takes an edge of } \pi \text{ whose time has been reduced or an edge} \\
				\text{of the pattern (and thus an "intermediate" edge of } B_2).
			\end{split}
		\end{equation}
		\item The time saved during the first modification (of order $r_3$) is such that every geodesic in the environment $\Td$ from $0$ to a vertex $w$ in $B_2$ has to take an edge of $\cD$ 
		(i.e.\ an edge whose time has been reduced during the first modification).
		If it was not the case, a path following $\gamma$ until $B_2$ and then taking edges of $B_2$ to go to $w$ would have a passage time smaller (we use here that $r_3 \gg r_2)$.
		This is similar for a geodesic between a vertex of $B_2$ and $x$.
		Hence, as stated in Lemma \ref{17. Toute géodésique touche gamma avant et après B4sN.},
		\begin{equation}\label{e:DavantB2}
			\overline\gamma^{**} \text{ takes an edge of } \cD \text{ before and after visiting } B_2.
		\end{equation}
		\item Recall that $\cD=\cD_\avant \cup \cD_\apres$, that the edges of $\cD_\avant$ are "before" edges and that those of $\cD_\apres$ are "after" edges.
		Combining \eqref{e:chronologie}, \eqref{e:pimotif} and \eqref{e:DavantB2} we get that
		\begin{equation}\label{e:Davantemotifapres}
			\begin{split}
				\overline\gamma^{**} \text{ takes an edge of } \cD_\avant \text{ then an edge of $\pi$ whose time has been reduced } \\ 
				\text{ or an edge of the pattern and then an edge of } \cD_\apres.
			\end{split}
		\end{equation}
		We can easily deduce from this that $\overline\gamma^{**}$ visits $s_1$ just before taking the first edge whose time has been modified and visits $s_2$ just after taking the last edge whose time has been modified (see Lemma \ref{17. Toute géodésique passe par s1 et s2 dans le bon ordre.}). Then we can deduce the desired properties about associated geodesics. It remains to prove that $\overline\gamma^{**}$ takes the pattern in $B_2$.
		\item Let us now prove that
		\begin{equation}\label{e:pitoucheavantapres}
			\begin{split}
				\overline\gamma^{**} \text{ takes an edge of $\pi$ whose time has been reduced } \text{ before (with respect to $\pi$)} \\
				\text{the pattern and after (with respect to $\pi$) the pattern}.
			\end{split}	
		\end{equation}
		This is the purpose of Lemma \ref{17. Toute géodésique touche pi avant et après le motif.}. By symmetry it is sufficient to prove the first part of this property. To this aim, it is sufficient to consider a vertex $w$ in the pattern or on $\pi$ between the pattern and $B^*(x)$ and to prove
		\[
		T^{**}(B^*(0),w) < T^*(B^*(0),w).
		\]
		But $T^*(B(0),w) \ge T^*(B^*(0),B^*(x)) - T^*(B^*(x),w)$ and 
		\[
		T^*(B^*(0),B^*(x)) = T^*(u_1,v_1) = T(u_1,v_1) \approx \mu(u_1-v_1)
		\]
		while 
		\[
		T^*(B^*(x),w) \le T^*(v_1,w) = T(v_1,w) \approx \mu(v_1-w).
		\]
		Thus (since the three vertices are roughly aligned in the correct order):
		\[
		T^*(B(0),w) \gtrsim \mu(u_1-v_1)-\mu(v_1-w) \approx \mu(u_1-w).
		\]
		But (following $\pi$) we have
		\[
		T^{**}(B^*(0),w)  \lesssim (\r+\delta') \|w-u_1\|
		\]
		and the result holds.
		\item By \eqref{e:pitoucheavantapres} and the remark made after the definition of walls, we can easily deduce that $\overline\gamma^{**}$ takes the pattern we have put on $\pi$ (see Lemma \ref{17. Toute géodésique qui touche pi avant et après le motif emprunte le motif.}).
	\end{enumerate}
	
	\subsubsection{First modification}
	
	Let us go back to the proof. For the rest of Section \ref{Sous-section modification argument, cas borné}, we fix $\l \in \{1,\dots,q\}$ and $s \in \Z^d$, and we assume that the event $\{T \in \cM(\l) \} \cap \left\{ \Sun_\l(T)=s \right\}$ occurs. In particular, $\Bqs$ is a typical box in the environment $T$.
	
	As explained above, the first modification is used to get that any geodesic from $0$ to $x$ in the new environment goes from $0$ to a specific vertex in $\Bqs$, then follows $\gamma$ to another specific vertex in $\Bqs$ and then goes to $x$. It is useful to get the properties about associated geodesics after the second modification. We denote by $\gamma$ the selected geodesic in the environment $T$.
	Recall the definition of entry and exit points given at the beginning of Section \ref{Sous-section modification argument, cas non borné}.	
	Let $u$ denote the entry point of $\gamma$ in $\Bts$ and $v$ the exit point, and let $u_0$ denote the entry point of $\gamma$ in $\Bds$ and $v_0$ its exit point. Then, we set
	\begin{equation}
		\Ep(T)=\left\{ \text{edges of $\Bts$ that belong to $\gamma_{u,u_0}$ or $\gamma_{v_0,v}$ and satisfy $T(e)>\r+\delta$} \right\}. \label{On définit Ep(T)}
	\end{equation}
	We denote the first edge of $\gamma$ that belongs to $\Ep(T)$ by $e_1$ and the last one by $e_2$. Moreover, $s_1$ denote the first vertex of $e_1$ visited by $\gamma$ and $s_2$ the last vertex of $e_2$ visited by $\gamma$. Assume that the event  \[\{T \in \cM(\l) \} \cap \left\{ \Sun_\l(T)=s \right\} \cap \{T' \in \cB^*(T) \} \mbox{ occurs},\] where $\cB^*(T)$ is defined in $(ii)$ of Lemma \ref{17. ancien lemme 12.9.}. Recall the definition of $\Tu$ in \eqref{Définition de T*, cas borné.}. Note that:
	\begin{itemize}
		\item only the time of the edges of $\Ep(T)$ have been modified,
		\item these edges belong to $\gamma$ and to $\Bts \setminus \Bds$,
		\item for all edges in $\Ep(T)$, we have $\Tu(e) \le \r + \delta' < \r + \delta \le T(e)$,
		\item for all other edges $e$, we have $\Tu(e)=T(e)$.
	\end{itemize}

	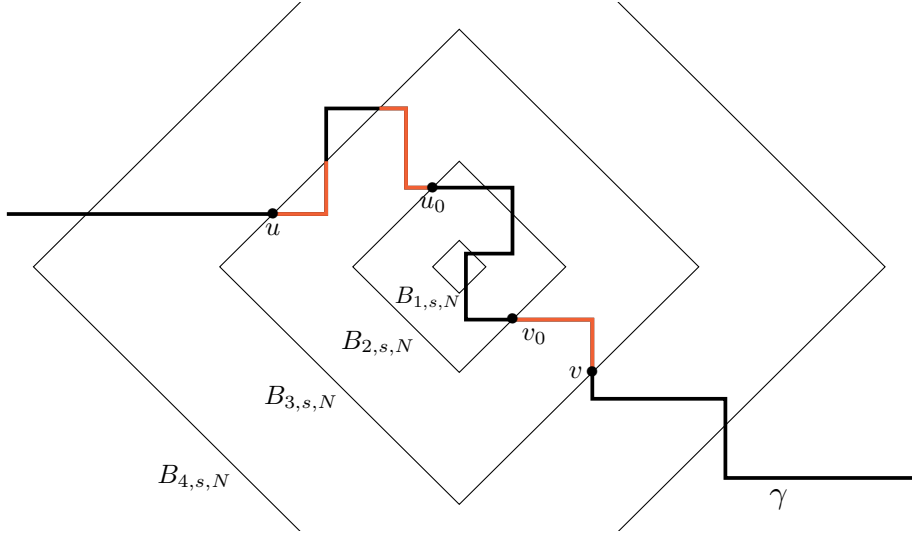
\begin{figure}
		\begin{center}
			\begin{tikzpicture}[scale=0.35]
				\clip (-17,-10) rectangle (17,10);
				\draw (0,16) -- (16,0) -- (0,-16) -- (-16,0) -- cycle;
				\draw (-8.2,-8) node[left] {$\Bqs$};
				\draw (0,9) -- (9,0) -- (0,-9) -- (-9,0) -- cycle;
				\draw (-4.2,-5) node[left] {$\Bts$};
				\draw (0,4) -- (4,0) -- (0,-4) -- (-4,0) -- cycle;
				\draw (-1.3,-2.9) node[left] {$\Bds$};
				\draw (0,1) -- (1,0) -- (0,-1) -- (-1,0) -- cycle;
				\draw (0.4,-1.2) node[left,scale=0.9] {$\Bus$};
				\draw[line width=1.4pt] (-17,2) -- (-5,2) -- (-5,6) -- (-2,6) -- (-2,3) -- (2,3) -- (2,0.5) -- (0.25,0.5) -- (0.25,-2) -- (5,-2) -- (5,-5) -- (10,-5) -- (10,-8) -- (20,-8);
				\draw (12,-8) node[below,scale=1.2] {$\gamma$};
				\draw[color=RedOrange,line width=1.4pt] (-7,2) -- (-5,2) -- (-5,4);
				\draw[color=RedOrange,line width=1.4pt] (-3,6) -- (-2,6) -- (-2,3) -- (-1,3);
				\draw[color=RedOrange,line width=1.4pt] (2,-2) -- (5,-2) -- (5,-4);
				\draw (-7,2) node[below] {$u$};
				\draw (5,-4) node[left] {$v$};
				\draw (-1,3) node[below] {$u_0$};
				\draw (2,-2) node[below right] {$v_0$};
				\draw (-7,2) node {$\bullet$};
				\draw (5,-4) node {$\bullet$};
				\draw (-1,3) node {$\bullet$};
				\draw (2,-2) node {$\bullet$};
			\end{tikzpicture}
			\caption{Elements involved in the first modification. The edges whose time can be modified by the first modification are represented in red.}\label{Figure première modification cas borné.}
		\end{center}
	\end{figure}
	
	\begin{lemma}\label{17. Première modification.}
		We have the following properties. 
		\begin{enumerate}[label=(\roman*)]
			\item There are at least $\alpha (r_3-r_2) N$ edges of $\gamma_{u,u_0}$ and $\alpha (r_3-r_2) N$ edges of $\gamma_{v_0,v}$ that belong to $\Ep(T)$. Thus, \[\min \left( T(\gamma_{0,u_0})-T^*(\gamma_{0,u_0}), T(\gamma_{v0,x})-T^*(\gamma_{v_0,x}) \right) \ge \alpha (r_3-r_2) N (\delta - \delta').\]
			\item In the environment $\Tu$, every geodesic from $0$ to $x$ visits every edge of $\Ep(T)$. 
			\item In the environment $\Tu$, $\gamma$ is a geodesic from $0$ to $x$.  
		\end{enumerate}
	\end{lemma}
	
	\begin{proof}
		\begin{enumerate}[label=(\roman*)]
			\item This item follows from the first property of a typical box applied to the portion of $\gamma_{u,u_0}$ entirely contained in $\Bts$ going from $\partial \Bts$ to $u_0$ and to the portion of $\gamma_{v_0,v}$ entirely contained in $\Bts$ going from $v_0$ to $\partial \Bts$ (note that these two portions does not have edges in common).
			\item To prove the second point, let $\gu$ be a geodesic from $0$ to $x$ in the environment $T^*$. Assume that there exists an edge of $\Ep(T)$ which is not an edge of $\gu$. Then \[T(\gu)-T^*(\gu) < T(\gamma)-T^*(\gamma).\]
			
			Since $\gamma$ is a geodesic from $0$ to $x$ in the environment $T$, we have $T(\gamma) \le T(\gu)$, which implies \[0 \le T(\gu)- T(\gamma) < T^*(\gu) - T^*(\gamma).\] 
			Thus $T^*(\gamma)<T^*(\gu)$, which contradicts the fact that $\gu$ is a geodesic from $0$ to $x$ in the environment $\Tu$.
			\item Let us now assume that $\gamma$ is not a geodesic in the environment $T^*$. Hence, if we denote by $\gu$ a geodesic from $0$ to $x$, we have $T^*(\gu) < T^*(\gamma)$. By item $(ii)$, \[ T(\gu) - T^*(\gu) = T(\gamma)-T^*(\gamma).\] Thus, $T(\gu) < T(\gamma)$, which contradicts the fact that $\gamma$ is a geodesic in the environment $T$.
		\end{enumerate}
	\end{proof}
	
	\subsubsection{Construction of $\pi$}\label{Sous-section construction du chemin pi.}
	
	Here, we shall identify an oriented path $\pi$ between two vertices of $\gamma$ in $\Bds$. This oriented path is later used to place the pattern.  Let $c_0$ denote the entry point of $\gamma$ in $\Bus$ and recall the definition of $\nabla$ in \eqref{On fixe t0 ou nabla0.}. Since by \eqref{on fixe r2, cas borné}, $\displaystyle r_2 \ge r_1+\frac{\nabla}{c_\mu}$, $B_\mu (c_0,N\nabla) \cap \Z^d$ is entirely contained in $\Bds$. We introduce
	\begin{equation}
		\text{$u_1$ the entry point of $\gamma$ in $B_\mu (c_0,N\nabla) \cap \Z^d$ and $v_1$ the exit point.} \label{On introduit u1 et v1.}
	\end{equation} 
	
	\begin{lemma}\label{17. minoration de mu(u1-v1) et arêtes de la zone rouges contenues dans B4.}
		We have $\mu(u_1-v_1) \ge N\nabla$ and $\gamma_{u_1,v_1}$ is contained in $\Bds$.
	\end{lemma}
	
	\begin{remk}
		The idea of the proof is that $\mu(u_1-v_1)$ is roughly equal to $t(u_1,v_1)$. Furthermore, since $\gamma$ is a geodesic visiting $u_1$, $c_0$ and $v_1$ in this order, we have $t(u_1,v_1)=t(u_1,c_0)+t(c_0,v_1)$. So $\mu(u_1-v_1) \approx \mu(u_1-c_0) + \mu(v_1-c_0) \approx 2N\nabla$. We have a sufficient control over the approximations to guarantee a lower bound by $N\nabla$. 
	\end{remk}
	
	\begin{proof}[Proof]
		Using item $(ii)$ of Lemma \ref{17. lemme technique partie argument de modification} with $z=c_0$ and $\rr = \nabla(1-\epsilon) - 1$ leads to
		\begin{equation}
		B(c_0,N(\nabla(1-\epsilon)-1)) \cap \Bts \subset B_\mu(c_0,N\nabla) \cap \Bts = B_\mu(c_0,N\nabla) \cap \Z^d. \label{Inclusion preuve lemme 4.10.}	
		\end{equation} 
		Since $u_1$ is the entry point of a path in $B_\mu(c_0,N\nabla)$, there exists a vertex $\ou_1 \in \Bts$ such that $\|\ou_1-u_1\|_1 =1$ and $\ou_1 \notin B_\mu(c_0,N\nabla)$. By \eqref{Inclusion preuve lemme 4.10.}, $t(c_0,\ou_1) \ge N(\nabla(1-\epsilon)-1)$, so \[t(c_0,u_1) \ge t(c_0,\ou_1)-\underbrace{t(\ou_1,u_1)}_{\le \|\ou_1-u_1\|_1 \ST }.\]
		The same argument holds for $t(c_0,v_1)$. Hence \[t(u_1,v_1)=t(u_1,c_0) + t(c_0,v_1) \ge 2N(\nabla(1-\epsilon)-1) - 2 \ST.\]
		Thus, since $\displaystyle \nabla > \frac{3 + 2 \ST}{1-3\epsilon}$, by the third item of the definition of a typical box,  \[\mu(u_1-v_1) \ge \frac{N(2\nabla(1-\epsilon)- 3) - 2\ST}{1+\epsilon} \ge N\nabla. \]
		For the second part of the proof, using the third item of Lemma \ref{17. lemme technique partie argument de modification} with $z=c_0$ and $\rr=\nabla$ leads to \[B_\mu(c_0,N\nabla) \cap \Z^d \subset B(c_0,((1+\epsilon)\nabla+1)N). \] 
		Then, by \eqref{on fixe r2, cas borné} we have $\displaystyle r_2 > r_1 + \frac{2(\nabla+2)}{c_\mu}$. So, for all $z \in B_\mu (c_0,2N(\nabla + 2))$, we have \[\|z-sN\|_1 \le \|z-c_0\|_1 + \|c_0-sN\|_1 \le \left(\frac{2(\nabla + 2)}{c_\mu}+r_1\right)N < r_2N.\] So, we have $B_\mu (c_0,2N(\nabla + 2)) \cap \Z^d \subset \Bds$ and since $\epsilon < \frac{1}{3}$, we have \[B_\mu \left( c_0, N \frac{(1+\epsilon)\nabla+2}{1-\epsilon} \right) \subset B_\mu (c_0,2N(\nabla + 2)).\] Thus, by the second item of Lemma \ref{17. lemme technique partie argument de modification} with $z=c_0$ and $\rr=\nabla$, we get \[B(c_0,((1+\epsilon)\nabla + N)) \subset B_\mu \left( c_0, N \frac{(1+\epsilon)\nabla+2}{1-\epsilon} \right) \cap \Z^d \subset B_\mu (c_0,2N(\nabla + 2)) \cap \Z^d. \] To sum up,
		\[B_\mu(c_0,N\nabla) \cap \Z^d \subset B(c_0,(1+\epsilon)N\nabla+N) \subset \Bds. \]
		Now, assume that $\gamma_{u_1,v_1}$ visits a vertex which is not in $\Bds$. Let denote by $z$ such a vertex and assume for example that $z$ is visited by $\gamma_{u_1,c_0}$. Then, we have, thanks to these inclusions, $t(c_0,z)>t(c_0,u_1)$, which is impossible since $\gamma_{u_1,c_0}$ is a geodesic.
	\end{proof}
	
	\begin{lemma}\label{Inclusion des boules des temps T étoile dans les boules des temps T.}
		Recall the definition of $B^*$ given at \eqref{Définition des boules des temps dans les environnements Tu et Td.} and that $\Tu(\gamma_{0,u_1})=t^*(0,u_1)$ and $\Tu(\gamma_{v_1,x})=t^*(v_1,x)$. 
		\begin{enumerate}[label=(\roman*)]
			\item We have the following inclusions:
			\[B^*(0,T^*(\gamma_{0,u_1})) \subset B(0,t(0,u_1)) \mbox{ and } B^*(x,T^*(\gamma_{v_1,x})) \subset B(x,t(v_1,x)).\]
			\item We have \[B^*(0,T^*(\gamma_{0,u_1})) \cap B^*(x,T^*(\gamma_{v_1,x})) = \emptyset. \]
		\end{enumerate}
	\end{lemma}
	
	\begin{proof}[Proof]
		\begin{enumerate}[label=(\roman*)]
			\item Let $s'$ be a vertex in $B^*(0,T^*(\gamma_{0,u_1}))$. The aim is to show \[t(0,s') \le t(0,u_1).\]  Let $\ogu$ be a geodesic from $0$ to $s'$ in the environment $T^*$. We denote by $s^*$ the last vertex visited by $\ogu$ among those visited by $\gamma$ (note that $0$ is such a vertex). 
			First, since $\ogu_{s^*,s'}$ does not take an edge of $\gamma$, $T^*(\ogu_{s^*,s'}) = T(\ogu_{s^*,s'})$.
			Then, by Lemma \ref{17. minoration de mu(u1-v1) et arêtes de la zone rouges contenues dans B4.}, $\gamma_{u_1,v_1}$ is entirely contained in $\Bds$, so $T^*(\gamma_{u_1,v_1})=T(\gamma_{u_1,v_1})$. Thus, also by Lemma \ref{17. minoration de mu(u1-v1) et arêtes de la zone rouges contenues dans B4.}, $T^*(\gamma_{u_1,v_1}) > 0$. So $\ogu$ does not take an edge of $\gamma_{v_1,x}$. Otherwise, since $\gamma$ is a geodesic in the environment $T^*$, $t^*(0,s') \ge T^*(\gamma_{0,v_1}) > T^*(\gamma_{0,u_1})$. Hence, the time saved by $\ogu$ after the modification comes only from the edges of $\gamma_{0,u_1}$. So,  
			\begin{equation}\label{17. inégalité preuve inclusion des boules des temps T* dans celles des temps T.}
				T(\ogu_{0,s^*})-T^*(\ogu_{0,s^*}) \le T(\gamma_{0,u_1})-T^*(\gamma_{0,u_1}).
			\end{equation}	
			
			Hence, 
			
			\begin{align*}
				t(0,s') & \le T(\ogu_{0,s^*})+T(\ogu_{s^*,s'}) \\
				& \le T(\gamma_{0,u_1})-T^*(\gamma_{0,u_1}) + \underbrace{T^*(\ogu_{0,s^*})+T^*(\ogu_{s^*,s'})}_{=t^*(0,s')} \mbox{ by \eqref{17. inégalité preuve inclusion des boules des temps T* dans celles des temps T.},} \\
				& = t^*(0,s') + T(\gamma_{0,u_1})-T^*(\gamma_{0,u_1}) \\
				& \le T^*(\gamma_{0,u_1}) + T(\gamma_{0,u_1}) -T^*(\gamma_{0,u_1}) \mbox{ since $s' \in B^*(0,T^*(\gamma_{0,u_1}))$,} \\
				& = t(0,u_1),
			\end{align*}
			which proves the inclusion and the same proof gives us the second inclusion. 
			\item Let $s'$ be a point of $B^*(0,T^*(\gamma_{0,u_1})) \cap B^*(x,T^*(\gamma_{v_1,x}))$. Then
			\begin{align*}
				t(0,x) & \le t(0,s')+t(s',x) \\
				& \le t(0,u_1) + t(v_1,x) \mbox{ by $(i)$} \\
				& < t(0,x)
			\end{align*}
			since $\gamma$ is a geodesic visiting $0$, $u_1$, $v_1$ and $x$ in this order and since $t(u_1,v_1)>0$, which is a contradiction.
		\end{enumerate}
	\end{proof}
	
	Now, we can make the construction of $\pi$, which is the path on which we would like to put the pattern that the geodesics have to take. We begin by two definitions:
	\begin{itemize}
		\item We say that a path between two vertices $y_1$ and $y_2$ is oriented if its number of edges is equal to $\|y_1-y_2\|_1$.
		\item A step of length $\l$ is a path of $\l$ consecutive edges in the same direction. 
	\end{itemize}
	We state the following lemma whose proof is left to the reader and where $[u_1,v_1]$ is the segment in $\R^d$ and $\|.\|_1$ is the norm on $\R^d$.
	
	\begin{lemma}\label{Lemme construction de pi.}
		We can construct a path $\pi$ with a deterministic rule such that:
		\begin{enumerate}[label=(\roman*)]
			\item $\pi$ is an oriented path connecting $u_1$ and $v_1$, 
			\item $\pi$ is the concatenation of steps of length $10 \lll$ except in $B_\infty(v_1,10\lll)$, where $\pi$ takes steps of length $1$,
			\item there exists a constant $L_1 \in \R_+$ only depending on $\lll$ and $d$ such that for all $z \in \Z^d$ which is in $\pi$, the distance for the norm $\|.\|_1$ between $z$ and $[u_1,v_1]$ is bounded by $L_1$, and for all $y \in [u_1,v_1]$, the distance for the norm $\|.\|_1$ between $y$ and $\pi$ (seen as a set of vertices in $\Z^d$) is bounded by $L_1$.
		\end{enumerate}
	\end{lemma}
	Let $\pi$ be the path given by Lemma \ref{Lemme construction de pi.}.
	We introduce
	\begin{align}
		& \text{$u_2$ the first vertex of $\pi$ starting from $v_1$ in $B^*(0,\Tu(\gamma_{0,u_1}))$,} \label{On introduit u2.} \\
		& \text{$v_2$ the first vertex of $\pi$ starting from $u_1$ in $B^*(x,T^*(\gamma_{v_1,x}))$.} \label{On introduit v2.}
	\end{align}
	Remark that $t^*(0,u_2) \le t^*(0,u_1)$ and $t^*(v_2,x) \le t^*(v_1,x)$.
	
	\begin{lemma}\label{Chemin pi contenu dans B2.}
		The path $\pi$ is contained in $\Bds$.
	\end{lemma}

	\begin{proof}
		Let $z$ be a vertex of $\pi$. Using Lemma \ref{Lemme construction de pi.}, there exists $y \in [u_1,v_1]$ such that $\|z-y\|_1 \le L_1$. The vertices $u_1$ and $v_1$ belong to $B_\mu(c_0,N\nabla)$ which is convex, thus $y \in B_\mu(c_0,N\nabla)$. So, we have
		\begin{align*}
			\|z-sN\|_1 & \le \underbrace{\|z-y\|_1}_{\le L_1} + \underbrace{\|y-c_0\|_1}_{\le \frac{\mu(y-c_0)}{c_\mu}} + \underbrace{\|c_0-sN\|_1}_{\le r_1 N} \\
			& \le L_1 + \frac{N \nabla}{c_\mu} + r_1 N \text{ since $y \in B_\mu(c_0,N\nabla)$,} \\
			& \le r_2 N \text{ by \eqref{on fixe r2, cas borné},}
		\end{align*}
		which proves that $z \in \Bds$.
	\end{proof}
	
	\subsubsection{Second modification}
	
	Now, let us define $\EeM$, $\Eep$, $\Eem$ and $\cO$. Let $c_1 \in \R^d$ denote the midpoint of $[u_1,v_1]$ and let us consider the set of all vertices of $\pi_{u_2,v_2}$ at distance at most $L_1$ for the norm $\|.\|_1$ from $c_1$. This set is not empty, hence we can choose one such vertex with a deterministic rule. Recall that $L_2 = L_1+(10+d)\lll$. Since $\nabla \ge 6dL_2C_\mu$ (see \eqref{On fixe t0 ou nabla0.}), we have that the distance for the norm $\|.\|_\infty$ between the chosen vertex and $v_1$ is greater than $10 \lll$. So there exist one or two steps of the path $\pi$ of length $10 \lll$ that contain the chosen vertex. We chose one step (between these one or two steps) with a deterministic rule to put the pattern. We denote the midpoint of this chosen step by $c_P$. Then, 
	\begin{equation}
		\mbox{$\cO(T,T')$ is equal to the direction of this step,} \label{On définit O(T,T').}
	\end{equation}
	\begin{equation}
		\text{and }	\EeM(T,T') = \left\lbrace \text{edges connecting vertices belonging to $B_\infty(c_P,\lll)$}\right\rbrace . \label{On définit EP(T,T').} 
	\end{equation}
	Note that by the direction of a step, we mean the integer $j$ such that for every distinct vertices $z_1$ and $z_2$ of this step, $z_1-z_2= \pm \|z_1-z_2\|_1 \epsilon_j.$ 
	Note also that for all vertex $z$ in $B_\infty(c_P,\lll)$, \[\|z-c_1\|_1 \le L_2.\]
	We define $u_3$ and $v_3$ as the endpoints of the oriented pattern such that $\pi$ visits $u_2$, $u_3$, $v_3$ and $v_2$ in this order.
	Recall that for a pattern, the entry and exit points correspond to the endpoints given by the definition of this pattern. Note that, thanks to the construction, $u_3$ and $v_3$ are two vertices of the chosen step of length $10 \lll$. We denote by $\gamma^\pi$ the path composed by the first geodesic in the lexicographic order from $0$ to $u_2$ in the environment $T^*$, then $\pi_{u_2,v_2}$ and then the first geodesic in the lexicographic order from $v_2$ to $x$ in the environment $T^*$. 
	Then, we have to define $\Eep(T,T')$ and $\Eem(T,T')$:
	\begin{align}
		\Eep(T,T') = \left\{\text{edges of $\pp$ except the one connected to $u_2$ and}\right. & \left.\text{the one connected to $v_2$} \right\}, \label{On définit Eeplus(T,T').} \\
		\Eem(T,T') = \{\text{edges $e$ in $\Bds$ such that $T(e) < \nu$ and which are not}& \nonumber \\
		\text{ in } \text{$\EeM(T,T')$, $\pp$, $B^*(0,\Tu(\gamma_{0,u_1}))$ } & \text{or $B^*(x,T^*(\gamma_{v_1,x}))$} \}. \label{On définit Eemoins(T,T').}
	\end{align}
	For all the sequel, we assume that the event  \[\{T \in \cM(\l) \} \cap \left\{ \Sun_\l(T)=s \right\} \cap \{T' \in \cB^*(T) \} \cap \{T'' \in \cB^{**}(T,T')\} \mbox{ occurs}.\] Recall the definition of the environment $\Td$ given at \eqref{Définition de T**, cas borné.} and of $\cB^{**}(T,T')$ given in Lemma \ref{17. ancien lemme 12.9.}. In particular,
	\begin{itemize}
		\item for all $e \in \Ep(T)$, $\Td(e) = \Tu(e) < T(e)$,
		\item for all $e \in \Eep(T,T')$, $\Td(e) \le \r + \delta'$,
		\item for all $e \in \Eem(T,T')$, $\Td(e) \ge \nu > T(e)=\Tu(e)$,
		\item since by Lemma \ref{Chemin pi contenu dans B2.}, $\pi$ is contained in $\Bds$, we have that $\Eep(T,T')$, $\Eem(T,T')$ and $\EeM(T,T')$ are pairwise disjoint and are contained in $\Bds$. Thus, since $\Ep(T)$ is contained in $\Bts \setminus \Bds$, item (i) of Lemma \ref{17. ancien lemme 12.9.} is satisfied.
	\end{itemize}
	In what follows, when we talk about edges whose time has been reduced, it means the edges $e$ such that $\Td(e) < T(e)$.
	Before proving item (ii) of Lemma \ref{17. ancien lemme 12.9.}, we state the following lemma which completes the vision of the sets $\Eep(T,T')$, $\Eem(T,T')$ and $\EeM(T,T')$. Recall the definition of $\Bd$ given at \eqref{Définition des boules des temps dans les environnements Tu et Td.}.
	
	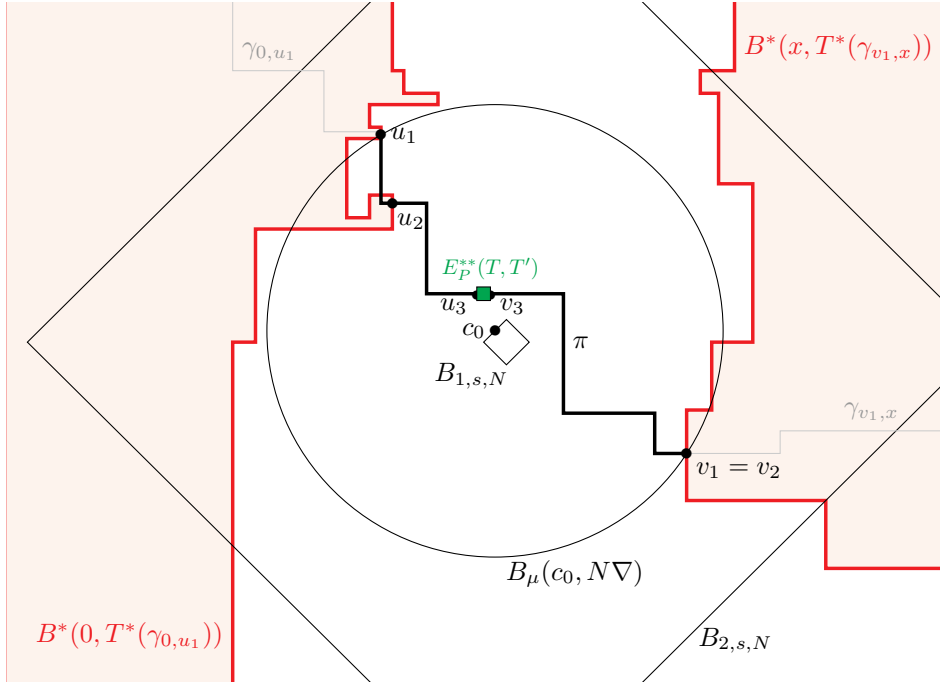
\begin{figure}\label{Figure de la deuxième modification.}
		\centering
		\begin{tikzpicture}[scale=0.3]
			\clip (-21.9,-15) rectangle (19,15);
			\filldraw[fill=Red!5,draw=Red,line width=1.3pt] (-22,22) -- (-5,22) -- (-5,12) -- (-4.5,12) -- (-4.5,11) -- (-3,11) -- (-3,10.5) -- (-6,10.5) -- (-6,9.5) -- (-5.5,9.5) -- (-5.5,9) --  (-7,9) -- (-7,5.5) -- (-6,5.5) -- (-6,6.5) -- (-5,6.5) -- (-5,5) -- (-11,5) -- (-11,0) -- (-12,0) -- (-12,-20) -- (-22,-20) -- cycle; 
			\filldraw[fill=Red!5,draw=Red,line width=1.3pt] (22,22) -- (10,22) -- (10,12) -- (8.5,12) -- (8.5,11) -- (9.3,11) -- (9.3,7) -- (10.8,7) -- (10.8,3) -- (10.8,0) -- (9,0) -- (9,-3) -- (7.9,-3) -- (7.9,-7) -- (14,-7) -- (14,-10) -- (22,-10) -- cycle;
			\draw[color=Red] (-12,-14) node[above left] {$B^*(0,\Tu(\gamma_{0,u_1}))$};
			\draw[color=Red] (10,12) node[above right] {$B^*(x,\Tu(\gamma_{v_1,x}))$};
			\draw (0,1) -- (1,0) -- (0,-1) -- (-1,0) -- cycle;
			\draw (0.5,-0.5) node[below left] {$\Bus$};
			\draw (-0.5,0.5) node {$\bullet$};
			\draw (-0.5,0.5) node[left] {$c_0$};
			\draw (-0.5,0.5) circle (10);
			\draw (3,-9.2) node[below] {$B_\mu(c_0,N\nabla)$};
			\draw (0,21) -- (21,0) -- (0,-21) -- (-21,0) -- cycle;
			\draw (8,-12.2) node [below right] {$\Bds$};
			\draw (-5,6.11) node {$\bullet$};
			\draw (-5.2,6.11) node[below right] {$u_2$};
			\draw[line width=1.3pt] (-5.5,9.14) -- (-5.5,6.14) -- (-3.5,6.14) -- (-3.5,2.14) -- (-1.5,2.14) -- (2.5,2.14) -- (2.5,-3.14) -- (6.5,-3.14) -- (6.5,-4.92) -- (7.9,-4.92);
			\draw (2.5,0) node[right] {$\pi$};
			\draw (-1.3,2.04) node[scale=0.9] {$\bullet$};
			\draw (-1.3,2.34) node[below left] {$u_3$};
			\draw (-0.7,2.044) node[scale=0.9] {$\bullet$};
			\draw (-0.7,2.34) node[below right] {$v_3$};
			\draw[fill=Green] (-1.3,1.84) rectangle (-0.7,2.44) node[above,Green,scale=0.8] {$\EeM(T,T')$};
			\draw[color=gray!40] (-5.5,9.3) -- (-8,9.3) -- (-8,12) -- (-12,12) -- (-12,17);
			\draw[color=gray!40] (7.9,-4.92) -- (12,-4.92) -- (12,-3.92) -- (20,-3.92);
			\draw[color=gray!80] (-12,12) node[above right] {$\gamma_{0,u_1}$};
			\draw[color=gray!80] (16,-3.92) node[above] {$\gamma_{v_1,x}$};
			\draw (-5.5,9.14) node {$\bullet$};
			\draw (-5.5,9.14) node[right] {$u_1$};
			\draw (7.9,-4.95) node {$\bullet$};
			\draw (7.9,-4.92) node[below right] {$v_1=v_2$};
		\end{tikzpicture}
		\centering
		\caption{The construction of $\pi$, of the vertices $u_1$, $u_2$, $u_3$, $v_1$, $v_2$, $v_3$ and of the pattern-location in dimension 2.}
	\end{figure}

	\begin{lemma}\label{17. Egalité des boules des deuxième et troisième configurations, et aucune arête du motif dans ces boules.}
		We have $\Bd(0,T^*(\gamma_{0,u_1})) = B^*(0,T^*(\gamma_{0,u_1}))$, $\Bd(x,T^*(\gamma_{v_1,x})) = B^*(x,T^*(\gamma_{v_1,x}))$ and there is no edge of $\EeM(T,T')$ in any of these balls. 
	\end{lemma}
	
	\begin{proof}[Proof]
		We begin by proving that there is no edge of $\EeM(T,T')$ in $B^*(0,T^*(\gamma_{0,u_1}))$. To this aim, we prove that there is no vertex of $B_\infty (c_P,\lll)$ in this ball. Let $z$ be a vertex of $B_\infty (c_P,\lll)$.
		The idea is the following. Since $\gamma$ is a geodesic in the environment $\Tu$, since the time of the edges of $\gamma_{u_1,v_1}$ is not modified by the first modification and since $\Bqs$ is a typical box in the environment $T$, for every $w$ vertex of $\gamma_{u_1,v_1}$,
		\[t^*(0,w) = t^*(0,u_1) + t^*(u_1,w) = t^*(0,u_1) + t(u_1,w) \approx t^*(0,u_1) + \mu(u_1-w).\]
		We can not guarantee that $z$ is a vertex of $\gamma_{u_1,v_1}$ but using a similar argument: 
		\[t^*(0,z) \ge t^*(0,v_1) - t^*(v_1,z) = t^*(0,u_1) + t^*(u_1,v_1)-t^*(v_1,z),\] since $\gamma$ is a geodesic in the environment $\Tu$.
		But \[t^*(u_1,v_1)=t(u_1,v_1) \approx 2N\nabla \text{ and } t^*(v_1,z) \le t^*(v_1,c_1) + t^*(c_1,z) \approx N \nabla,\] since $t^*(c_1,z)$ is negligible compared to $N \nabla$, which gives that $t^*(0,z)$ is roughly greater than $t^*(0,u_1) + N \nabla$.
		
		Now, we make the proof rigorous. Since $z \in B_\infty(c_P,\lll)$, we have $z \in B_\mu(c_1, C_\mu L_2)$. Since by \eqref{On fixe t0 ou nabla0.}, $\nabla \ge 2 C_\mu L_2$, $\displaystyle z \in B_\mu \left(c_1, \frac{N \nabla}{2} \right)$. Furthermore, since $c_1$ is the midpoint of $[u_1,v_1]$ and $\mu(u_1-v_1) \le 2N\nabla$, we have $\mu(v_1-c_1) \le N\nabla$. Thus, since $\displaystyle \nabla \ge \frac{1}{\epsilon}$ (by \eqref{On fixe t0 ou nabla0.}), using the third item of Lemma \ref{17. lemme technique partie argument de modification} with $z=v_1$ and $\rr=\frac{3 \nabla}{2}$, we have \[z \in B_\mu \left(v_1, \frac{3N\nabla}{2} \right) \cap \Z^d \subset B \left( v_1, \frac{3N\nabla(1+2\epsilon)}{2} \right). \] Hence, using the lower bound for $t(u_1,v_1)$ of the proof of Lemma \ref{17. minoration de mu(u1-v1) et arêtes de la zone rouges contenues dans B4.} and since by \eqref{On fixe t0 ou nabla0.}, $\displaystyle \nabla > \frac{4(1+\ST)}{1-10\epsilon}$, we have \[t(0,v_1)-t(0,u_1)=t(u_1,v_1) \ge 2N(\nabla(1-\epsilon)-1) - 2\ST > \frac{3N\nabla(1+2\epsilon)}{2},\] and thus \[ t(0,z) \ge t(0,v_1)-t(v_1,z) \ge t(0,v_1)-\frac{3N\nabla(1+2\epsilon)}{2} > t(0,u_1). \] The first item of Lemma \ref{Inclusion des boules des temps T étoile dans les boules des temps T.} allows us to conclude. 
		
		Then, let us prove the first equality. The proof of the second one is the same. The inclusion $B^*(0,T^*(\gamma_{0,u_1})) \subset \Bd(0,T^*(\gamma_{0,u_1}))$ is easy to check. Let us take $z \in B^*(0,T^*(\gamma_{0,u_1}))$ and $\ogu$ a geodesic from $0$ to $z$ in the environment $T^*$. Then $\ogu$ is entirely contained in $B^*(0,T^*(\gamma_{0,u_1}))$ and there are no edges of $\Eem(T,T')$ or $\EeM(T,T')$ in $B^*(0,T^*(\gamma_{0,u_1}))$, so \[\td(0,z) \le \Td(\ogu) \le T^*(\ogu) \le T^*(\gamma_{0,u_1}). \]  
		
		For the other inclusion, assume that there exists a vertex $z \in \Bd(0,T^*(\gamma_{0,u_1})) \setminus B^*(0,T^*(\gamma_{0,u_1}))$ and let $\ogd$ be a geodesic from $0$ to $z$ in the environment $\Td$.  Let $\oz$ be the first vertex of $\ogd$ which is not in $B^*(0,T^*(\gamma_{0,u_1}))$. By construction, there is no edge of $\ogd_{0,\oz}$ in $\Eep(T,T')$ or $\EeM(T,T')$ and thus: \[\td(0,z) \ge \td(0,\oz) \ge t^*(0,\oz) > T^*(\gamma_{0,u_1}), \] which is a contradiction.
	\end{proof}

	 Now, to get item (ii) of Lemma \ref{17. ancien lemme 12.9.}, fix $\displaystyle \eta=\min_{j \in \{1,\dots,d\}} \P \left(T \in \cA^\Lambda_j \right) \tilde{p}^{|\Bts|}$, where \[\tilde{p} = \min(F([\r,\r+\delta']),F([\nu(N),\ST])) > 0.\] Thus, $\eta$ only depends on $F$, the pattern and $N$ and we have that \[\P \left( \left. T' \in \cB^*(T) \right| T \right)  \ge \tilde{p}^{|\Bts|} \ge \eta\] and \[ \P \left( \left. T'' \in \cB^{**}(T,T') \right| T,T' \right) \ge \min_{j \in \{1,\dots,d\}} \P \left(T \in \cA^\Lambda_j \right) \tilde{p}^{|\Bts|} = \eta.\]
	\medskip
	
	We end this section with two lemmas, one giving a lower bound on the time saved by the geodesics from $0$ to $x$ thanks to the second modification and the other on the distance between $u_1$ and any vertex of $\pi_{v_3,v_2}$. 

	\begin{lemma}\label{17. Temps gagné pendant la deuxième modification.}
		For all $N \in \N^*$, we have  \[T^*(\gamma)-\Td(\gamma^\pi) \ge \frac{N\nabla}{2C_\mu} (\delta-\delta') > T^\Lambda.\]
	\end{lemma}
	
	\begin{proof}[Proof]
		This result is an easy consequence of Lemma \ref{17. minoration de mu(u1-v1) et arêtes de la zone rouges contenues dans B4.}. Since $\nabla \ge C_\mu$ (by \eqref{On fixe t0 ou nabla0.}) and all edges of $\gamma_{u_1,v_1}$ are contained in $\Bds$, which implies that there is no edge of $\gamma_{u_1,v_1}$ whose time has been reduced between the environment $T$ and $T^*$, \[T^*(\gamma) \ge T^*(\gamma_{0,u_1})+\|u_1-v_1\|_1(\r+\delta)+T^*(\gamma_{v_1,x}),\] where we used $(i)$ of the definition of a typical box.
		Further \[\Td(\gamma^\pi) \le \Td(\gamma^\pi_{0,u_2}) + \|u_2-v_2\|_1(\r+\delta') + \Td(\gamma^\pi_{v_2,x})+ 2\ST + \tau^\Lambda, \] where the term $2\ST$ is an upper bound for the time for both, the edge connecting $u_2$ to $\Eep(T,T')$ and the one connecting $v_2$ to $\Eep(T,T')$, and the term $\tau^\Lambda$ is an upper bound for the time collected by $\gamma^\pi$ in $\EeM(T,T')$.
		Since we have $\Td(\gamma^\pi_{0,u_2}) = \Tu(\gamma^\pi_{0,u_2}) \le T^*(\gamma_{0,u_1})$, $\Td(\gamma^\pi_{v_2,x}) = \Tu(\gamma^\pi_{v_2,x}) \le T^*(\gamma_{v_1,x})$, $\|u_2-v_2\|_1 \le \|u_1-v_1\|_1$ since $\pi$ is an oriented path, and $\displaystyle \nabla \ge 4C_\mu(2\ST+\tau^\Lambda)$ (by \eqref{On fixe t0 ou nabla0.}), and using Lemma \ref{17. minoration de mu(u1-v1) et arêtes de la zone rouges contenues dans B4.}, we obtain for all $N \in \N^*$, \[
		T^*(\gamma)-\Td(\gamma^\pi) \ge \|u_1-v_1\|_1 (\delta-\delta')-\tau^\Lambda -2\ST \ge \frac{N\nabla}{2C_\mu} (\delta-\delta').\]
		Finally, since by \eqref{On fixe t0 ou nabla0.}, $\displaystyle \nabla > \frac{2 C_\mu T^\Lambda}{\delta-\delta'}$ we have the result.
	\end{proof}
	
	\begin{lemma}\label{17. Minoration distance a1 motif.}
		For all $N \in \N^*$, for all $w$ vertex of $\pi_{v_3,v_2}$, we have \[\|u_1-w\|_1 \ge \frac{\|u_1-v_1\|_1}{3}.\]
	\end{lemma}
	
	\begin{proof}[Proof]
		Let us note that since $\pi$ is an oriented path, for all $w$ vertex of $\pp$ after $\EeM(T,T')$, we have $\|u_1-w\|_1 \ge \|u_1-v_3\|_1$. So \[\|u_1-w\|_1 \ge \|u_1-c_1\|_1 - \|c_1-v_3\|_1 \ge \frac{1}{2} \|u_1-v_1\|_1-L_2.\]
		
		Then, since by \eqref{On fixe t0 ou nabla0.}, $\nabla \ge 6L_2C_\mu$, we have $\displaystyle \frac{N\nabla}{6C_\mu} \ge L_2$ and using Lemma \ref{17. minoration de mu(u1-v1) et arêtes de la zone rouges contenues dans B4.} leads to the result.  
	\end{proof}
	
	\subsubsection{Proof of the second and third items of Lemma \ref{17. gros lemme modification.}}
	
	Recall that we assume that  \[\{T \in \cM(\l) \} \cap \left\{ \Sun_\l(T)=s \right\} \cap \{T' \in \cB^*(T) \} \cap \{T'' \in \cB^{**}(T,T')\} \mbox{ occurs}.\]
	The aim of this section is to prove the following properties (which are the second and third items of Lemma \ref{17. gros lemme modification.}):
	\begin{enumerate}[label=(\roman*)]
		\setcounter{enumi}{1}
		\item for all geodesic $\ogd$ from $0$ to $x$ in the environment $\Td$, there exists a geodesic $\OG$ from $0$ to $x$ in the environment $T$ such that $\OG$ and $\ogd$ are associated in $B_{4,\Sun_\l(T),N}$,
		\item there exists a geodesic $\gd$ in the environment $\Td$ from $0$ to $x$ such that $\gd$ and the selected geodesic $\gamma$ in the environment $T$ are associated in $B_{4,\Sun_\l(T),N}$.
	\end{enumerate}
	
	To prove this, we use the following sequence of lemmas. 
	
	\begin{lemma}\label{17. Toute géodésique touche pi.}
		Every geodesic from $0$ to $x$ in the environment $\Td$ takes at least one edge of $\pp$. 
	\end{lemma}
	
	\begin{proof}[Proof]
		Let $\gd$ be a geodesic from $0$ to $x$ in the environment $\Td$. Since $\gamma$ is a geodesic in the environment $\Tu$, we have the following inequalities \[\Td(\gd) \le \Td(\gamma^\pi) < T^*(\gamma) \le T^*(\gd). \]
		So, using Lemma \ref{17. Temps gagné pendant la deuxième modification.}, \[T^*(\gd)-\Td(\gd) > T^\Lambda.\]
		It means that $\gd$ has to take at least one edge whose time has been reduced during the second modification which is not in $\EeM(T,T')$. Hence, since the only edges which are not in $\EeM(T,T')$ whose time has been reduced are edges of $\pp$, the result follows. 
	\end{proof}
	
	\begin{lemma}\label{17. Toute géodésique touche gamma avant et après B4sN.}
		Every geodesic from $0$ to $x$ in the environment $\Td$  takes at least one edge of $\gamma$ whose time has been reduced before taking its first edge of $\Bds$, and takes at least one edge of $\gamma$ whose time has been reduced after taking its last edge of $\Bds$.
	\end{lemma}
	
	\begin{proof}[Proof]
		The idea is to use that the time saved by the geodesics from $0$ to $x$ after the first modification is much greater than the geodesic time between any two vertices of $\Bds$ in any environment. 
		Let $\gd$ be a geodesic from $0$ to $x$ in the environment $\Td$. Let $\ud$ be the first vertex in $\Bds$ that $\gd$ visits. Its existence is guaranteed by Lemma \ref{17. Toute géodésique touche pi.}. The aim of the proof is to show \[\Td(\gd_{0,\ud})<T(\gd_{0,\ud}).\] Indeed, the definition of $\ud$ and the fact that the only edges whose time has been reduced which are in $\Bts$ but not in $\Bds$ are edges of $\gamma$ gives us the result. 
		Recall that $u_0$ is the entry point of $\gamma$ in $\Bds$. First, since $\gd$ is a geodesic in the environment $\Td$, \[\Td(\gd_{0,\ud}) \le \Td (\gamma_{0,u_0}) + 2r_2N\ST. \]
		Then, using the first item of Lemma \ref{17. Première modification.}, we obtain \[\Td(\gd_{0,\ud}) \le T (\gamma_{0,u_0}) + 2r_2N\ST - \alpha(r_3-r_2)N(\delta-\delta'). \]
		Finally, using the fact that $\gamma$ is a geodesic in the environment $T$ leads to 
		\[\Td(\gd_{0,\ud}) \le T (\gd_{0,\ud}) + 4r_2N\ST - \alpha(r_3-r_2)N(\delta-\delta'). \]
		To conclude, it is sufficient to observe that the condition $\displaystyle r_3 > \frac{7r_2(4\ST+\alpha \delta)}{\alpha \delta}$ implies that $\alpha(r_3-r_2)N(\delta-\delta')-4r_2N\ST >0$, so we have the desired strict inequality. The same proof gives us the second part of the lemma.  
	\end{proof}
	
	\begin{lemma}\label{17. Toute géodésique passe par s1 et s2 dans le bon ordre.}
		Let $\gd$ be a geodesic from $0$ to $x$ in the environment $\Td$. Then the first edge of $\gd$ whose time has been reduced is $e_1$ and the last is $e_2$. Moreover, the first vertex of $e_1$ taking by $\gd$ is $s_1$ and the last of $e_2$ is $s_2$.
	\end{lemma}
	
	\begin{proof}[Proof]
		Let $\gd$ be a geodesic from $0$ to $x$ in the environment $\Td$. Let $\zd$ denote the last vertex visited by $\gd$ before it takes for the first time an edge of $\gamma$ whose time has been reduced. We know by the construction and by Lemma \ref{17. Toute géodésique touche gamma avant et après B4sN.} that $\zd$ is a vertex visited by $\gamma_{u,u_0}$ or $\gamma_{v_0,v}$, and thus it is a vertex visited by $\gamma_{u,u_1}$ or $\gamma_{v_1,v}$. Let us prove that it is a vertex visited by $\gamma_{u,u_1}$. Assume, aiming at a contradiction that $\zd$ is a vertex visited by $\gamma_{v_1,v}$.
		
		On the one hand, by Lemma \ref{17. Toute géodésique touche gamma avant et après B4sN.}, $\gd_{0,\zd}$ does not take edges in $\Bds$ and since by Lemma \ref{Chemin pi contenu dans B2.}, $\pi$ is contained in $\Bds$, $\gd_{0,\zd}$ does not take any edge of $\pp$.
		
		On the other hand, all edges of $\gamma_{v_1,x}$ are in $B^*(x,T^*(\gamma_{v_1,x}))$. So, if $\zd$ is visited by $\gamma_{v_1,x}$, we have $\zd \in B^*(x,T^*(\gamma_{v_1,x}))$. But $\Bd(x,T^*(\gamma_{v_1,x})) = B^*(x,T^*(\gamma_{v_1,x}))$ by Lemma \ref{17. Egalité des boules des deuxième et troisième configurations, et aucune arête du motif dans ces boules.}. Since $\gd$ is a geodesic in the environment $\Td$, it implies that $\gd_{\zd,x}$ is entirely contained in $B^*(x,T^*(\gamma_{v_1,x}))$. Thus $\gd_{\zd,x}$ does not take any edge of $\pp$ since there is no edge of $\pp$ in $B^*(x,T^*(\gamma_{v_1,x}))$. 
		
		Combining these two conclusions, we get that $\gd$ does not visit any edge of $\pp$, which contradicts Lemma \ref{17. Toute géodésique touche pi.}. So $\zd$ is a vertex visited by $\gamma_{u,u_1}$. Knowing this, we can complete the proof. By the definition of $\zd$ and by Lemma \ref{17. Toute géodésique touche gamma avant et après B4sN.}, we have \[T(\gd_{0,\zd}) = \Td(\gd_{0,\zd}).\] Since $\gd$ is a geodesic, \[\Td(\gd_{0,\zd}) \le \Td(\gamma_{0,\zd}).\]
		Now, let us assume $\zd$ is not $s_1$. Then, by the definition of $s_1$, $\zd$ is visited by $\gamma_{u,u_1}$ after $s_1$ and thus, 
		\[\Td(\gamma_{0,\zd}) < T(\gamma_{0,\zd}).\]
		Combining these three inequalities yields \[T(\gd_{0,\zd}) < T(\gamma_{0,\zd}),\] which is impossible because $\gamma_{0,\zd}$ is a geodesic in the environment $T$. So, the result is proved and the same proof leads to the second part of this lemma. 
	\end{proof}
	
	We can now prove the two properties stated at the beginning of this subsection. For item $(ii)$, let $\ogd$ be a geodesic from $0$ to $x$ in the environment $\Td$. By Lemma \ref{17. Toute géodésique touche gamma avant et après B4sN.} and Lemma \ref{17. Toute géodésique passe par s1 et s2 dans le bon ordre.}, like $\gamma$, $\ogd$ visits $s_1$ and before that $\ogd$ does not visit any edge whose time has been changed when replacing $T$ by $\Td$. Hence,
	\begin{equation}
		T(\gamma_{0,s_1}) = \Td(\gamma_{0,s_1}) \ge \Td(\ogd_{0,s_1}) = T(\ogd_{0,s_1}) \ge T(\gamma_{0,s_1}). \label{Suite d'égalités et d'inégalités deuxième et troisième points du lemme 3.5.}
	\end{equation}
	This proves that $T(\gamma_{0,s_1})=T(\ogd_{0,s_1})$ and $\ogd_{0,s_1}$ is a geodesic in the environment $T$. Similarly, $\ogd_{s_2,x}$ is a geodesic in the environment $T$ and the path $\OG = \ogd_{0,s_1} \cup \gamma_{s_1,s_2} \cup \ogd_{s_2,x}$ is a geodesic in the environment $T$ that satisfies $(ii)$ in Lemma \ref{17. gros lemme modification.} if we prove that it is a self-avoiding path and it is contained in $\Bqs$. Assume, aiming at a contradiction that $\ogd_{0,s_1}$ visits a vertex of $\gamma_{s_1,s_2}$ which is not $s_1$ and denote it by $s_3$. Then $\ogd_{0,s_1} \cup \gamma_{s_1,x}$ is an optimal path for the passage time in the environment $T$ since $\ogd_{0,s_1}$ is a geodesic in this environment. It implies that $\ogd_{s_3,s_1} \cup \gamma_{s_1,s_3}$ has at least one edge which is $e_1$ and that $T(\ogd_{s_3,s_1} \cup \gamma_{s_1,s_3})=0$, which is impossible since $T(e_1)>0$. The same proof gives that $\ogd_{s_2,x}$ does not visit a vertex of $\gamma_{s_1,s_2}$ and thus $\OG$ is self-avoiding.
	To get item $(ii)$, it remains to prove that $\ogd_{s_1,s_2}$ is contained in $\Bqs$. This comes from the fact that, in any environment, the geodesic time between two vertices of $\Bts$ is bounded by $2 r_3 \ST$, and, since the edges in $\Bqs \setminus \Bts$ have the same times in the environments $T$ or $\Td$, by property $(i)$ of a typical box, the geodesic time to reach a vertex outside $\Bqs$ and to come back in $\Bts$ is bounded from below by $2(r_4-r_3)(\r+\delta)$. The condition on $r_4$ insures that $(r_4-r_3)(\r+\delta)>r_3\ST$. 
	
	For item $(iii)$, observe that from \eqref{Suite d'égalités et d'inégalités deuxième et troisième points du lemme 3.5.}, we also get $\Td(\ogd_{0,s_1})=\Td(\gamma_{0,s_1})$ and $\Td(\ogd_{s_2,x})=\Td(\gamma_{s_2,x})$. Thus, the path $\gamma_{0,s_1} \cup \ogd_{s_1,s_2} \cup \gamma_{s_2,x}$ is an optimal path for the passage time in the environment $\Td$. If it is not self-avoiding, we get a geodesic that satisfies the requirement of $(iii)$ in Lemma \ref{17. gros lemme modification.} by cutting its loops with the same process as in the proof of Lemma \ref{17. gros lemme modification, cas non borné.} in the unbounded case. If we denote by $s'_1$ and $s'_2$ the two vertices such that $\gamma_{0,s'_1} \cup \ogd_{s'_1,s'_2} \cup \gamma_{s'_2,x}$ is a geodesic in the environment $\Td$ obtained by cutting the loops, we have to justify that $\gamma_{s'_1,s'_2}$ is entirely contained in $\Bqs$. We know that $\gamma_{s_1,s_2}$ is entirely contained in $\Bqs$. Let us show that $\gamma_{s'_1,s_1}$ is also entirely contained in $\Bqs$, the proof for $\gamma_{s_2,s'_2}$ is the same. Since $\gamma_{0,s_1} \cup \ogd_{s_1,s_2} \cup \gamma_{s_2,x}$ is an optimal path for the passage time in the environment $\Td$, we have $\Td(\ogd_{s_1,s'_1})+\Td(\gamma_{s'_1,s_1})=0$ and thus in particular $\Td(\gamma_{s'_1,s_1})=0$. Since $s_1$ belongs to $\Bts$, if $\gamma_{s'_1,s_1}$ visits a vertex outside $\Bqs$, there exist two vertices $z$ and $z'$ of $\gamma_{s'_1,s_1}$ such that $z \in \partial \Bts$, $z' \in \partial \Bqs$ and $\gamma_{z,z'}$ except for $z$ is contained in $\Bqs \setminus \Bts$. So $\Td(\gamma_{s'_1,s_1})=0$ implies that $\Td(\gamma_{z,z'})=0$ but $T(\gamma_{z,z'})=\Td(\gamma_{z,z'})=0$ since the time of the edges of $\Bqs \setminus \Bts$ has not been changed. It makes a contradiction with \eqref{17. chemin pas anormalement courts} since $r_4-r_3 \ge 1$ by \eqref{On fixe r4.}. 
	
	\subsubsection{Every geodesic takes the pattern.}
	
	Recall that we assume that  \[\{T \in \cM(\l) \} \cap \left\{ \Sun_\l(T)=s \right\} \cap \{T' \in \cB^*(T) \} \cap \{T'' \in \cB^{**}(T,T')\} \mbox{ occurs}.\]
	The aim of this last subsection is to show that every geodesic in the environment $\Td$ takes the pattern in $\EeM(T,T')$. The proof is decomposed in two steps. The first step is to show that every geodesic takes an edge of $\pi_{u_2,u_3}$ and an edge of $\pi_{v_3,v_2}$. The second step is to show that every geodesic verifying this property takes the pattern in $\EeM(T,T')$. We begin with a technical lemma. 
	
	\begin{lemma}\label{17. lemme technique pour toute géodésique emprunte le motif.}
		For all $w$ vertex of $\pp$, \[|\mu(u_1-w)+\mu(v_1-w)-\mu(u_1-v_1)|\le 2 C_\mu L_1.\]
	\end{lemma}
	
	\begin{proof}[Proof]
		Let $w$ be a vertex of $\pp$. Then, by the construction of $\pi$, there exists a $\ow \in [u_1,v_1] \subset \R^d$ such that $\|w-\ow\|_1 \le L_1$. We have \[\mu(u_1-v_1)=\mu(u_1-\ow)+\mu(v_1-\ow).\]
		Then
		\begin{align*}
			|\mu(u_1-w)-\mu(u_1-v_1)+\mu(v_1-w)| & \le |\mu(u_1-w)-\mu(u_1-\ow)|+|\mu(v_1-w)-\mu(v_1-\ow)| \\
			& \le 2 C_\mu \|w-\ow\|_1 \\
			& \le 2 C_\mu L_1.
		\end{align*}
	\end{proof}
	
	\begin{lemma}\label{17. Toute géodésique touche pi avant et après le motif.}
		Let $\gd$ be a geodesic from $0$ to $x$ in the environment $\Td$. Then $\gd$ visits a vertex of $\pi_{u_2,u_3}$ and one of $\pi_{v_3,v_2}$. More precisely, the first vertex of $\gd$ that belongs to $\pp$ belongs to $\pi_{u_2,u_3}$ and the last to belongs to $\pi_{v_3,v_2}$. 
	\end{lemma}
	
	\begin{proof}[Proof]
		Let $\gd$ be a geodesic from $0$ to $x$ in the environment $\Td$. By Lemma \ref{17. Toute géodésique touche pi.}, there exists at least one vertex of $\pp$ visited by $\gd$. Let $w$ be the first vertex of $\gd$ that belongs to $\pp$ and assume that $w$ belongs to $\pi_{v_3,v_2}$. Note that $\gd_{0,w}$ does not visit any other vertex of $\pp$. The aim of the proof is to show that \[\Td(\gamma^\pi_{0,w}) < \Td(\gd_{0,w}),\] which is impossible since $\gd$ is a geodesic in the environment $\Td$. We start with 
		\begin{align*}
			T^*(\gd_{0,w}) & \ge t^*(0,w) \\
			& \ge t^*(0,v_1)-t^*(v_1,w) \\
			& = t^*(0,u_1) + t^*(u_1,v_1)-t^*(v_1,w),
		\end{align*}
		since, by Lemma \ref{17. Première modification.}, $\gamma$ is a geodesic in the environment $T^*$. By construction, there is no edge of $\gamma_{u_1,v_1}$ whose time has been changed at the first modification, thus $t^*(u_1,v_1)=t(u_1,v_1)$. Furthermore, since there is no edge whose time has been increased at the first modification,  $t^*(v_1,w) \le t(v_1,w)$, so \[T^*(\gd_{0,w}) \ge t^*(0,u_1)+t(u_1,v_1)-t(v_1,w).\]
		We now want to bound from below $\Td(\gd_{0,w})$. The only edges whose time has been reduced at the second modification are among those of $\pp$ and of $\EeM(T,T')$. So, since $\gd_{0,w}$ does not take any edge of $\pp$, it can only save time taking edges of $\EeM(T,T')$. So,
		
		\[\Td(\gd_{0,w}) \ge t^*(0,u_1) + t(u_1,v_1) - t(v_1,w) - T^\Lambda. \]
		Then, using the definition of a typical box, Lemma \ref{17. lemme technique pour toute géodésique emprunte le motif.} and the inequality $t^*(0,u_1) \ge t^*(0,u_2)$ (which comes from the definition of $u_2$) leads to  
		\[\Td(\gd_{0,w}) \ge t^*(0,u_2) + \mu(u_1-w)-\epsilon(\mu(u_1-v_1)+\mu(v_1-w))-2 N -2C_\mu L_1 - T^\Lambda.\]
		On the other hand, note that, $\pi$ being an oriented path, $\|u_1-w\|_1 \ge \|u_2-w\|_1$, so, using the knowledge of $\Td$ on edges of $\pi$,
		\begin{align*}
			\Td(\gamma^\pi_{0,w}) & = \Td(\gamma^\pi_{0,u_2})+\Td(\gamma^\pi_{u_2,w}) \\
			& \le t^*(0,u_2)+2\ST+2\ST \lll + (\r+\delta')\|u_1-w\|_1.
		\end{align*}
		To conclude, let us show that we have the inequality 
		\begin{align}
			& \hspace{-50pt} t^*(0,u_2)+2\ST+2\ST \lll + (\r+\delta')\|u_1-w\|_1 \nonumber \\
			& <  \, t^*(0,u_2) + \mu(u_1-w)-\epsilon(\mu(u_1-v_1)+\mu(v_1-v))-2 N -2C_\mu L_1 - T^\Lambda. \label{inégalité souhaitée pour la fin de la preuve du cas borné.}
		\end{align}
		First, combining Lemma \ref{17. minoration de mu(u1-v1) et arêtes de la zone rouges contenues dans B4.} and Lemma \ref{17. Minoration distance a1 motif.} leads to
		\begin{equation}
			\|u_1-w\|_1 \ge \frac{N\nabla}{3C_\mu}. \label{Minoration norme u1 moins w.}
		\end{equation} 
		Then, by \eqref{Lien entre mu et delta.}, we have \[\mu(u_1-w) \ge \left( \r + \delta \right) \|u_1-w\|_1.\]
		Recall that $\Ki = T^\Lambda+2(C_\mu L_1+\ST+\ST\lll)$, it is sufficient to have \[\delta' < \delta- \epsilon C_\mu \frac{\|u_1-v_1\|_1+\|v_1-w\|_1}{\|u_1-w\|_1 } - \frac{2 N}{\|u_1-w\|_1 } - \frac{\Ki}{\|u_1-w\|_1 }. \]
		Then, by Lemma \ref{17. Minoration distance a1 motif.}, $\displaystyle \frac{\|u_1-v_1\|_1+\|v_1-w\|_1}{\|u_1-w\|_1} \le 6$. So, since $\displaystyle \epsilon < \frac{\delta}{24 C_\mu}$, we have \[\epsilon C_\mu \frac{\|u_1-v_1\|_1+\|v_1-w\|_1}{\|u_1-w\|_1 } < \frac{\delta}{4}. \]
		By \eqref{Minoration norme u1 moins w.} and since $\displaystyle 1 < \frac{\delta \nabla}{24 C_\mu}$ (by \eqref{On fixe t0 ou nabla0.}) and $\displaystyle N > \frac{12 C_\mu \Ki}{\delta \nabla}$ (by \eqref{On fixe N, n et x dans le cas borné.}), the condition $\displaystyle \delta' \le \frac{\delta}{4}$ gives us \eqref{inégalité souhaitée pour la fin de la preuve du cas borné.}.
		Hence $\Td(\gamma^\pi_{0,w})<\Td(\gd_{0,w})$.
	\end{proof}
	
	Finally, let us prove the following lemma which completes the proof of Lemma \ref{17. gros lemme modification.}. 
	
	\begin{lemma}\label{17. Toute géodésique qui touche pi avant et après le motif emprunte le motif.}
		Any geodesic from $0$ to $x$ takes the pattern at the pattern-location $\EeM(T,T')$.
	\end{lemma}
	
	\begin{proof}[Proof]
		Let $\gd$ be a geodesic from $0$ to $x$ in the environment $\Td$. By Lemma \ref{17. Toute géodésique touche pi avant et après le motif.}, $\gd$ visits a vertex of $\pi_{u_2,u_3}$ and one of $\pi_{v_3,v_2}$. As a consequence, there exist a vertex $u_4$ of $\pi_{u_2,u_3}$ and a vertex $v_4$ of $\pi_{v_3,v_2}$ such that $\gd$ goes from $u_4$ to $v_4$ without taking edges of $\pp$. Let us remember that for all edge $e$ of $\EeM(T,T')$, we have $\Td(e) \le \nu$. We prove successive properties. 
		
		\begin{itemize}
			\item \textit{The edges of $\gd_{u_4,v_4}$ which are not in $\EeM(T,T')$ have a passage time greater than or equal to $\nu$.}
			
			\medskip
			
			Since there is no edge of $\gd_{u_4,v_4}$ in $\Bd(0,T^*(\gamma_{0,u_1}))$, $\Bd(x,T^*(\gamma_{v_1,x}))$ or $\pp$, it is sufficient to prove that $\gd_{u_4,v_4}$ is entirely contained in $\Bds$. By convexity, we have that all points of $[u_1,v_1]$ are contained in $B_\mu(c_0,N\nabla)$, so by Lemma \ref{Lemme construction de pi.}, $\displaystyle \|u_4-c_0\|_1 \le \frac{N \nabla}{c_\mu} + L_1$, and thus, \[\|u_4-sN\|_1 \le N \left( \frac{\nabla}{c_\mu} + r_1 \right) + L_1. \] So, if $\gd_{u_4,v_4}$ is not entirely contained in $\Bds$, the number of edges whose time is greater than or equal to $\nu$ that $\gd_{u_4,v_4}$ has to travel to leave $\Bds$ is bounded from below by $\left(r_2 - \frac{\nabla}{c_\mu} - r_1\right) N -2d\lll-L_1$, and we get  \[\Td(\gd_{u_4,v_4}) \ge \left( \left( r_2 - \frac{\nabla}{c_\mu} - r_1 \right)N-2d\lll-L_1 \right) \nu. \]
			But \[\Td(\gamma^\pi_{u_4,v_4}) \le \frac{2N\nabla}{c_\mu} (\r+\delta') +2\ST(1+\lll), \]
			and since $\displaystyle N \ge 1$, $\displaystyle \frac{\ST}{\nu} \ge 1$, $\displaystyle \frac{\r + \delta}{\nu} \le 1$ and by \eqref{on fixe r2, cas borné}, \[ r_2 > r_1 + L_1 + \frac{3\nabla}{c_\mu} + \frac{2\ST}{\nu}(1+(1+d)\lll), \] we have $\displaystyle \Td(\gamma^\pi_{u_4,v_4}) < \Td(\gd_{u_4,v_4})$, which is impossible since $\gd_{u_4,v_4}$ is a geodesic in the environment $\Td$.  
			
			\item \textit{We have that $\|u_4-u_3\|_1 \le 4\lll$ and $\|v_4-v_3\|_1 \le 4\lll$.}
			
			\medskip
			
			Assume that it is not the case. Let us show that $\Td(\gd_{u_4,v_4})-\Td(\pi_{u_4,v_4})>0$, which is a contradiction since $\gd_{u_4,v_4}$ is a geodesic in the environment $\Td$. 
			To this aim, let us compare the time that can save each path compared with the other in any direction. Recall the notation introduced in the proof of Lemma \ref{Lemme théorème 6.2 KRAS.}: for $i \in \{1,\dots,d\}$ and a path $\tilde{\pi}$, $\Td_i({\tilde{\pi}})$ denotes the sum of the passage times of the edges of $\tilde{\pi}$ which are in the direction $\epsilon_i$. In the direction $\cO(T,T')$, since $\|u_4-u_3\|_1>4\lll$ or $\|v_4-v_3\|_1>4\lll$, and by the construction, we have that
			\[\Td_{\cO(T,T')}(\gd_{u_4,v_4}) - \Td_{\cO(T,T')}(\pi_{u_4,v_4}) \ge 4 \lll \nu - 4 \lll (\r + \delta') + 2 \lll \r - 2\lll \nu, \]
			where the term $2\lll\r-2\lll\nu$ comes from the time potentially saved by $\gd_{u_4,v_4}$by taking edges of $\EeM(T,T')$. Then, in any other direction, $\gd_{u_4,v_4}$ can save a time lower than or equal to $2\lll\delta'$ compared with $\pi_{u_4,v_4}$ thanks to the edges in $\EeM(T,T')$.
			Hence, 
			\begin{align*}
				\Td(\gd_{u_4,v_4})-\Td(\pi_{u_4,v_4}) & \ge 2\lll\nu+2\lll \r-4\lll(\r+\delta') - 2 (d-1) \lll \delta' \\
				& \ge 2\lll\nu-2\lll \r-(4\lll + 2(d-1)\lll) \delta'> 0
			\end{align*}  
			since $\displaystyle \r+(1+d)\delta' < \r + \delta \le \nu$. 
			
			\item \textit{We have that $u_4=u_3$ and $v_4=v_3$.}
			
			\medskip
			
			Assume that it is not the case. First, assume that $\gd_{u_4,v_4}$ does not take any edge of $\EeM(T,T')$. According to the first property, $\Td(\gd_{u_4,v_4}) \ge \|u_4-v_4\|_1 \nu$. Then, since $\pi$ is an oriented path, \[ \Td(\gamma^\pi_{u_4,v_4}) \le (\|u_4-v_4\|_1-2\lll)(\r+\delta') + 2\lll \nu.\] Since $u_4 \neq u_3$ or $v_4 \neq v_3$, we have that $\|u_4-v_4\|_1>2\lll$ and we obtain \[\Td(\gamma^\pi_{u_4,v_4}) < \Td(\gd_{u_4,v_4}),\] which is a contradiction since $\gd_{u_4,v_4}$ is a geodesic in the environment $\Td$.
			So $\gd_{u_4,v_4}$ takes an edge of $\EeM(T,T')$. Let $\ud_0$ be the first entry point of $\gd_{u_4,v_4}$ in $\EeM(T,T')$ and consider the path $\pi^{**}$ following $\pi_{u_4,u_3}$, then going from $u_3$ to $\ud_0$ in one of the shortest way for the norm $\|.\|_1$ and then following $\gd_{\ud_0,v_4}$. Then the number of edges of $\pi^{**}_{u_4,\ud_0}$ is lower than or equal to the number of edges of $\gd_{u_4,\ud_0}$, for all $e \in \pi^{**}_{u_4,\ud_0}$, $\Td(e) \le \nu$, there exists $e' \in \pi^{**}_{u_4,\ud_0}$ such that $\Td(e') \le \r + \delta'$ and for all $e \in \gd_{u_4,\ud_0}$, $\Td(e) \ge \nu$. So we have $\Td(\pi^{**}) < \Td(\gd_{u_4,v_4})$, which is impossible since $\gd$ is a geodesic in the environment $\Td$. The same proof gives $v_4=v_3$.
			
			\item \textit{$\gd$ takes the pattern at the pattern-location $\EeM(T,T')$.}
			
			\medskip
			
			Assume that $\gd_{u_4,v_4}$ is not entirely contained in $\EeM(T,T')$. Let $\vd_0$ be the first exit point from $\EeM(T,T')$ of $\gd_{u_4,v_4}$ and $\ud_1$ the first entry point after $\vd_0$. Let us consider the shortcut $\pi^{**}$ going from $\vd_0$ to $\ud_1$ in one of the shortest way for the norm $\|.\|_1$. Then let $\vd_{0,+}$ denote the first vertex visited by $\gd$ after $\vd_0$, then \[\|\ud_1-\vd_{0,+}\|_1=\|\ud_1-\vd_0\|_1+1.\] Indeed, we have that $\|\ud_1-\vd_{0,+}\|_1-\|\ud_1-\vd_0\|_1$ is equal to $1$ or $-1$, and if it is equal to $-1$, it implies that $\vd_{0,+}$ is in $\EeM(T,T')$, which is impossible. So, $\gd_{\vd_0,\ud_1}$ has strictly more edges than $\pi^{**}$. Furthermore, all edges of $\gd_{\vd_0,\ud_1}$ have a time greater than or equal to $\nu$ although all edges of $\pi^{**}$ have a time lower than or equal to $\nu$. So, $\Td(\pi^{**}) < \Td(\gd_{\vd_0,\ud_1})$, which is a contradiction since $\gd_{\vd_0,\ud_1}$ is a geodesic in the environment $\Td$.
			
			Thus $\gd_{u_4,v_4}$ is a path entirely contained in $\EeM(T,T')$, going from $u_3$ to $v_3$ and with an optimal time. So, we have the result. 
		\end{itemize}
	\end{proof}

	\section{Proofs of generalizations of modification arguments in \cite{KRAS}}\label{Section preuve des généralisations de KRAS.}
	
	\subsection{Modification proof for the Euclidean length of geodesics}\label{Section preuve du théorème 6.2.}
	
	To prove Theorem \ref{Théorème 6.2 KRAS}, we begin by defining the valid pattern in three different cases. Recall that $k$ and $\l$ are given by the assumptions of this theorem and that $(\epsilon_i)_{i \in \{1,\dots,d\}}$ are the vectors of the canonical basis.
	
	\subparagraph*{Case where zero is an atom.}
	
	In addition to the assumptions of Theorem \ref{Théorème 6.2 KRAS}, we assume that zero is an atom for $F$. We set $L_1=k+2$, $L_2=\l+2$ and if $d \ge 3$, for all $i \in \{3,\dots,d\}$, $L_i=2$. We define a pattern in $\displaystyle \Lambda=\prod_{i=1}^{d} \{0,\dots,L_i\}$. We take the endpoints $u=\sum_{i=2}^d \epsilon_i$ and $v=u+(k+2) \epsilon_1$.
	We denote by $\PP$ the path going from $u$ to $v$ by $k+2$ steps in the direction $\epsilon_1$ and by $\PPP$ the path going from $u$ to $u+\epsilon_1$ by one step in the direction $\epsilon_1$, then to $u+\epsilon_1+\l \epsilon_2$ by $\l$ steps in the direction $\epsilon_2$, then to $u+(k+1) \epsilon_1+\l \epsilon_2$ by $k$ steps in the direction $\epsilon_1$, then to $u+(k+1) \epsilon_1$ by $\l$ steps in the direction $- \epsilon_2$ and then to $v$ by one step in the direction $\epsilon_1$. We define $\cA^\Lambda$ as follows:
	\begin{itemize}
		\item for all $e \in \PP \cup \PPP$, $T(e)=0$,
		\item for all $e \in \Lambda$ which is not in $\PP \cup \PPP$, $T(e)>0$.
	\end{itemize}
	Note that $\cA^\Lambda$ has a positive probability. Then, $\PP$ and $\PPP$ are the only two optimal self-avoiding paths from $u$ to $v$ entirely contained in the pattern. Furthermore, for every vertex $z \in \PP \cup \PPP$ different from $u$ and $v$, there exists no path from $\partial \Lambda \setminus \{u,v\}$ to $z$ whose passage time is equal to $0$.
	
	\subparagraph*{Unbounded case.}
	
	Here, in addition to the assumptions of Theorem \ref{Théorème 6.2 KRAS}, we assume that zero is not an atom and that the support of $F$ is unbounded. We set $L_1=k$, $L_2=\l$ and if $d \ge 3$, for all $i \in \{3,\dots,d\}$, $L_i=0$. We define a pattern in $\displaystyle \Lambda=\prod_{i=1}^{d} \{0,\dots,L_i\}$. We take the endpoints $u=0$ and $v=k \epsilon_1$. We denote by $\PP$ the path going from $0$ to $k \epsilon_1$ by $k$ steps in the direction $\epsilon_1$ and by $\PPP$ the path going from $0$ to $\l \epsilon_2$ by $\l$ steps in the direction $\epsilon_2$, then to $k \epsilon_1+\l \epsilon_2$ by $k$ steps in the direction $\epsilon_1$ and then to $v$ by $\l$ steps in the direction $-\epsilon_2$. Then, we index the edges of $\PP$ and the ones of $\PPP$ in the order in which they are taken by theses paths. We respectively denote them by $(e^1_i)_{i \in \{1,\dots,k\}}$ and $(e^2_i)_{i \in \{1,\dots,k+2\l\}}$. We fix $\displaystyle M > \sum_{j=1}^{k}s'_j$ and we define $\cA^\Lambda$ as follows: 
	\begin{itemize}
		\item for all $i \in \{1,\dots,k\}$, $T(e^1_i)=s'_i$ and for all $i \in \{1,\dots,k+2\l\}$, $T(e^2_i)=r'_i$,
		\item for all $e \in \Lambda$ which are not in $\PP$ or $\PPP$, $T(e)>M$. 
	\end{itemize}
	Since the support of $F$ is unbounded, $\P(\cA^\Lambda)$ is positive. Furthermore, we have \[T(\PP) = T(\PPP).\] Since $M > T(\PP)$, the optimal paths from $u$ to $v$ entirely contained in the pattern can not take other edges than those in $\PP \cup \PPP$. Hence $\PP$ and $\PPP$ are the only two optimal paths from $u$ to $v$ entirely contained in the pattern. 
	
	\subparagraph*{Bounded case.}
	
	In addition to the assumptions of Theorem \ref{Théorème 6.2 KRAS}, we assume that zero is not an atom and that the support of $F$ is bounded. We set $\ST=\sup(\text{support}(F))$. We denote 
	$\aM=\max(r'_1,\dots,r'_{k+2\l},s'_1,\dots,s'_k)$. Then, there are at least $2\l$ integers $j \in \{1,\dots,k+2\l\}$ such that $r'_j < \aM$. Indeed, assume that this is not the case. Then, we have \[ \sum_{i=1}^{k+2\l} r'_i \ge (k+1)\aM > k\aM \ge \sum_{j=1}^{k}s'_j, \] and this contradicts \eqref{Egalité avec les atomes}. Thus, even if it means changing the indexes, we can assume that for all $i \in \{1,\dots,2\l\}$, $r'_{i}<\aM$ and we denote $t_w=\aM-\max(r'_1,\dots,r'_{2\l})>0$. We fix $\alpha > 0$ an integer such that: 
	\begin{equation}
		\alpha > \max \left( \frac{k}{2\l}, \frac{k \aM}{\l t_w} \right). \label{On fixe alpha dans la preuve du théorème 6.2 de KRAS.}
	\end{equation}
	We set $k'=\alpha k$, $\l' = \alpha \l$, $L_1=k'$, $L_2=\l'$ and if $d \ge 3$, for all $i \in \{3,\dots,d\}$, $L_i=0$. We define a pattern in $\displaystyle \Lambda=\prod_{i=1}^{d} \{0,\dots,L_i\}$. We take the endpoints $u=0$ and $v=k' \epsilon_1$. Let $\PP$ be the path going from $u$ to $v$ by $k'$ steps in the direction $\epsilon_1$ and $\PPP$ be the path going from $u$ to $u'=\l' \epsilon_2$ by $\l'$ steps in the direction $\epsilon_2$, then to $v'=k' \epsilon_1+\l' \epsilon_2$ by $k'$ steps in the direction $\epsilon_1$ and then to $v$ by $\l'$ steps in the direction $-\epsilon_2$. Then we index the edges of $\PP$, $\PPP_{u,u'}$, $\PPP_{u',v'}$ and $\PPP_{v',v}$ in the order in which they are taken by these paths. We respectively denote them by $\left(e^1_i\right)_{i \in \{1,\dots,k'\}}$, $\left(e^2_i\right)_{i \in \{1,\dots,\l'\}}$, $\left(e^3_i\right)_{i \in \{1,\dots,k'\}}$ and $\left(e^4_i\right)_{i \in \{1,\dots,\l'\}}$. The idea for the event $\cA^\Lambda$ is just to alternate the atoms on every boundary of the rectangle whose vertices are $u$, $u'$, $v'$ and $v$. It allows us a better control of the time of a path taking both vertices of $\PP$ and vertices of $\PPP_{u',v'}$ (see Figure \ref{Figure Théorème 6.2 KRAS.}).
	
	So, we define $\cA^\Lambda$ as follows : 
	\begin{itemize}
		\item for all $i \in \{1,\dots,k'\}$, $T(e^1_i)=s'_{i[k]}$ and $T(e^3_i)=r'_{2\l+i[k]}$ where $i[k]$ is the integer in $\{1,\dots,k\}$ such that $i-i[k]$ is divisible by $k$,
		\item for all $i \in \{1,\dots,\l'\}$, $T(e^2_i)=r'_{i[\l]}$ and $T(e^4_i)=r'_{\l+i[\l]}$,
		\item for all other edges $e \in \Lambda$, $T(e) = \aM$. 
	\end{itemize}
	
	Note that $\cA^\Lambda$ has a positive probability and that on this event, $T(\PP)=T(\PPP)$. 
	
	\begin{figure}
		\begin{center}
			\begin{tikzpicture}[scale=1.4]
				\draw[color=gray!40] (0.5,4.5) grid [xstep=0.5,ystep=0.5] (8.5,8.5);
				\draw (1,8) node[fill=white,scale=1.2] {$\Lambda$};
				\draw[color=RedOrange,line width=1.5pt] (0.5,4.5) -- (8.5,4.5);
				\draw[color=PineGreen,line width=1.5pt] (0.5,4.5) -- (0.5,8.5) -- (8.5,8.5) -- (8.5,4.5);
				\draw[color=RedOrange] (4.5,4.5) node[below] {$\PP$};
				\draw[color=PineGreen] (4.5,8.4) node[fill=white,below] {$\PPP$};
				\draw (0.5,4.5) node[below left] {$u$};
				\draw (0.5,4.5) node {$\bullet$};
				\draw (0.5,8.5) node[above left] {$u'$};
				\draw (0.5,8.5) node {$\bullet$};
				\draw (8.5,8.5) node[above right] {$v'$};
				\draw (8.5,8.5) node {$\bullet$};
				\draw (8.5,4.5) node[below right] {$v$};
				\draw (8.5,4.5) node {$\bullet$};
				\foreach \x in {0,2,4,6} \foreach \y in {1,2,3,4} \draw (0.25+\x+0.5*\y,4.5) node[above,scale=0.8] {$s'_\y$};
				\foreach \x in {0,2,4,6} \foreach \y in {5,6,7,8} \draw (0.25-2+\x+0.5*\y,8.5) node[above,scale=0.8] {$r'_\y$};
				\foreach \x in {0,1,2,3} \foreach \y in {1,2} \draw (0.5,4.25+\x+0.5*\y) node[left,scale=0.8] {$r'_\y$};
				\foreach \x in {0,1,2,3} \foreach \y in {3} \draw (8.5,4.25+\x+0.5*\y-1) node[right,scale=0.8] {$r'_4$};
				\foreach \x in {0,1,2,3} \foreach \y in {4} \draw (8.5,4.25+\x+0.5*\y-1) node[right,scale=0.8] {$r'_3$};
			\end{tikzpicture}
			\caption{Paths $\PP$ and $\PPP$ with their passage times in the pattern for the proof of Theorem \ref{Théorème 6.2 KRAS} in the bounded case with $d=2$, $k=4$, $\l=2$ and $\alpha=4$.}\label{Figure Théorème 6.2 KRAS.}
		\end{center}
	\end{figure}
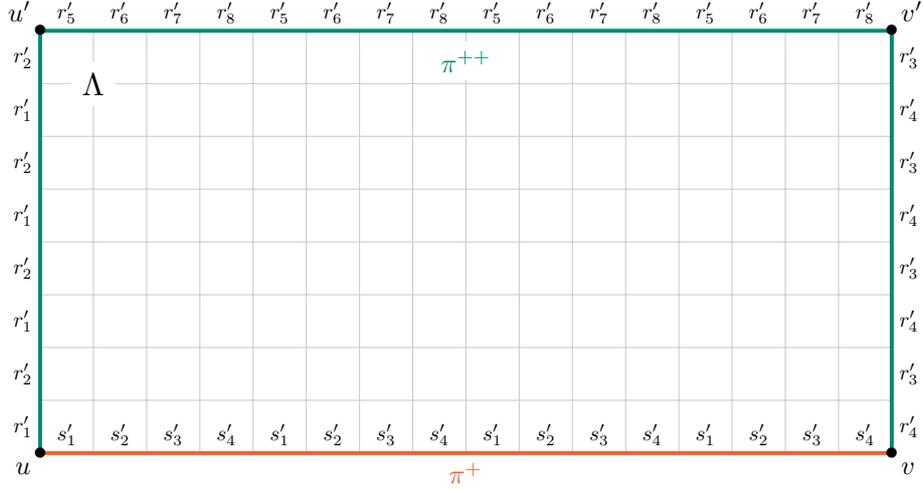
	
	\begin{lemma}\label{Lemme théorème 6.2 KRAS.}
		The paths $\PP$ and $\PPP$ belong to the family of optimal paths from $u$ to $v$ which are entirely contained in the pattern. 
	\end{lemma}
	
	\begin{proof}[Proof of Lemma \ref{Lemme théorème 6.2 KRAS.}]
		Let us begin by introducing some notations. For $i \in \{1,\dots,d\}$ and a path $\pi$, $T_i(\pi)$ denotes the sum of the passage times of the edges of $\pi$ which are in the direction $\epsilon_i$. So we have $T(\pi)=T_1(\pi)+\dots+T_d(\pi)$. Furthermore, for $j \in \{0,\dots,\l'-1\}$, we denote by $S^j_2$ the set of edges of $\Lambda$ which can be written $\{x,x+\epsilon_2\}$ where the second coordinate of $x$ is equal to $j$.
		
		
		Now, to prove the lemma, assume for a contradiction that there exists an optimal self-avoiding path $\pi$ entirely contained in $\Lambda$ such that 
		\begin{equation}
			T(\pi)<T(\PP)=T(\PPP). \label{Preuve lemme de la preuve du théorème 6.2 de KRAS.}
		\end{equation} 
		
		\textbf{First step:}
		$\pi$ takes at least one edge in $\PPP_{u',v'}$.
		Indeed, assume that it is not the case. Then \[T(\pi) \ge T_1(\pi) \ge T_1(\PP)=T(\PP),\] which contradicts \eqref{Preuve lemme de la preuve du théorème 6.2 de KRAS.}. Note that the second inequality comes from the fact that edges in the direction $\epsilon_1$ whose passage time is strictly smaller than those in $\PP$ must belong to $\PPP_{u',v'}$. 
		
		\textbf{Second step:}
		We denote by $x$ (resp. $y$) the first vertex of $\PPP_{u',v'}$ visited by $\pi$. Let us prove that $\pi_{x,y} = \PPP_{x,y}$. Note that, since $\pi$ is a self-avoiding path from $u$ to $v$ and since every path entirely contained in the pattern can only take edges of directions $\epsilon_1$ and $\epsilon_2$, 
		\begin{itemize}
			\item $x$ (resp. $y$) is also the first vertex of $\pi$ visited by $\PPP_{u',v'}$,
			\item $\pi_{x,y}$ does not take any edge of $\PPP_{u,x}$ and $\PPP_{y,v}$.
		\end{itemize}
		Assume for a contradiction that $\pi_{x,y} \neq \PPP_{x,y}$. Denote by $x'$ and $y'$ two distinct vertices of $\PPP_{x,y} \cap \pi_{x,y}$ such that $\pi_{x',y'}$ does not take any edge of $\PPP$ and such that $x'$ is visited by $\PPP_{u',v'}$ before $y'$. We have $|\pi_{x',y'}| > |\PPP_{x',y'}|$ and thus $\pi_{x',y'}$ has to take at least one edge in $\PP$, else
		\[T(\pi_{x',y'}) = |\pi_{x',y'}| \aM > |\PPP_{x',y'}| \aM \ge T(\PPP_{x',y'}),\] which contradicts the fact that $\pi$ is an optimal path. 
		Since $\pi_{x',y'}$ has to take edges in $\PP$, and since it cannot take edges of $\PPP$, we get 
		\[T_2(\pi_{x',y'}) \ge 2 \l' \aM.\]
		Furthermore, for each edge $e$ in $\PPP_{x',y'}$, $\pi_{x',y'}$ has to take an edge in the direction $\epsilon_1$ such that this edge is the edge $e-\l'\epsilon_2 \in \PP$ or such that its passage time is equal to $\aM$. Hence
		\[T_1(\pi_{x',y'}) \ge \left\lfloor \frac{\|x'-y'\|_1}{k} \right\rfloor \sum_{i=1}^{k} s'_i.\]
		But,
		\[T_1(\PPP_{x',y'}) \le \left\lceil \frac{\|x'-y'\|_1}{k} \right\rceil \sum_{i=1+2\l}^{k+2\l} r'_i.\]
		Thus, since $\displaystyle \sum_{i=1+2\l}^{k+2\l} r'_i \le \sum_{i=1}^{k} s'_i \le k \aM,$ we have
		\[ T(\PPP_{x',y'})-T(\pi_{x',y'}) \le \sum_{i=1}^{k} s'_i - 2 \alpha \l \aM \le (k-2\alpha\l) \aM < 0 \text{ by \eqref{On fixe alpha dans la preuve du théorème 6.2 de KRAS.},}\]
		which contradicts the fact that $\pi$ is an optimal path.
		
		\textbf{Third step:}
		We have just proven that $\pi_{x,y} = \PPP_{x,y}$. Hence, \eqref{Preuve lemme de la preuve du théorème 6.2 de KRAS.} implies that $T(\pi_{u,x}) < T(\PPP_{u,x})$ or that $T(\pi_{y,v}) < T(\PPP_{y,v})$. Assume for a contradiction that
		\begin{equation}
			T(\pi_{u,x}) < T(\PPP_{u,x}), \label{Troisième step dans la preuve du lemme du théorème 6.2 de KRAS.}
		\end{equation}
		the other case being the same. 
		First, we have $x \ne u'$ else, using again that $\pi$ is a self avoiding path from $u$ to $v$ and the fact that every path entirely contained in the pattern can only take edges of directions $\epsilon_1$ and $\epsilon_2$, $\pi_{u,x}$ can not take any edge of $\PPP_{x,v}$, we have $T_2(\pi_{u,u'}) \ge T_2(\PPP_{u,u'})$ and thus \[T(\pi_{u,x})=T(\pi_{u,u'}) \ge T_2(\pi_{u,u'}) \ge T_2(\PPP_{u,u'}) = T(\PPP_{u,u'})=T(\PPP_{u,x}),\] which contradicts \eqref{Troisième step dans la preuve du lemme du théorème 6.2 de KRAS.}.
		
		Denote by $x''$ the last vertex of $\PPP_{u,u'}$ visited by $\pi_{u,x}$ (note that we can have $x''=u$). We have that $x'' \ne u'$, else $x''=x=u'$.
		Now, since $\pi$ is a self-avoiding path and using the definitions of $x''$ and $x$, we get that $\pi_{x'',x}$ cannot take any edge of $\PPP_{u,x}$. Hence, $\pi_{x'',x}$ takes at least one edge of $\PP$. Indeed, if it is not the case, we have $T_1(\pi_{x'',x}) \ge T_1(\PPP_{x'',x})$ and we get (using that $\|x''-u'\|_1 >0$):
		\begin{align*}
			T(\pi_{x'',x}) & = T_1(\pi_{x'',x}) + T_2(\pi_{x'',x}) \ge T_1(\PPP_{x'',x}) + \|x''-u'\|_1 \aM \\
			& > T_1(\PPP_{x'',x}) + \|x''-u'\|_1 \max(r'_1,\dots,r'_{2\l}) \ge  T_1(\PPP_{x'',x}) +  T_2(\PPP_{x'',x}) =  T(\PPP_{x'',x}),
		\end{align*}
		which contradicts the fact that $\pi$ is an optimal path. 
		
		Now, since $\pi_{x'',x}$ takes edges of $\PP$, $\pi_{x'',x}$ has to take at least $\l'$ edges in the direction $\epsilon_2$ of passage time equal to $\aM$. Thus,
		\[T_2(\pi_{x'',x}) \ge \l' \aM.\]
		Then,  for each edge $e$ in $\PPP_{u',x}$, $\pi_{x'',x}$ has to take an edge in the direction $\epsilon_1$ such that this edge is the edge $e-\l'\epsilon_2 \in \PP$ or such that its passage time is equal to $\aM$. We get 
		\[T_1(\pi_{x',y'}) \ge \left\lfloor \frac{\|u'-x\|_1}{k} \right\rfloor \sum_{i=1}^{k} s'_i.\]
		On the other hand, we have
		\[T_1(\PPP_{x'',x}) \le\left\lceil \frac{\|u'-x\|_1}{k} \right\rceil \sum_{i=1+2\l}^{k+2\l} r'_i,\]
		and
		\[T_2(\PPP_{x'',x}) \le \|x''-u'\|_1 \max(r'_1,\dots,r'_{2\l}) \le \l' \max(r'_1,\dots,r'_{2\l}) \le \l' (\aM-t_w). \]
		Hence, since $\displaystyle \sum_{i=1+2\l}^{k+2\l} r'_i \le \sum_{i=1}^{k} s'_i$,
		\begin{align*}
			T(\PPP_{x'',x}) - T(\pi_{x'',x}) & \le T_1(\PPP_{x'',x}) - T_1(\pi_{x'',x}) + T_2(\PPP_{x'',x}) - T_2(\pi_{x'',x}) \\
			& \le \underbrace{\sum_{i=1}^{k} s'_i}_{\le k \aM} - \l' t_w \\
			& \le k \aM - \alpha \l t_w < 0 \text{ by \eqref{On fixe alpha dans la preuve du théorème 6.2 de KRAS.}},
		\end{align*}
		which contradicts the fact that $\pi$ is an optimal path. Thus, there is no optimal self-avoiding path $\pi$ entirely contained in $\Lambda$ such that \eqref{Preuve lemme de la preuve du théorème 6.2 de KRAS.} holds and the proof is complete.
	\end{proof}
	
	\subparagraph*{Conclusion in the three cases.}
	
	\begin{proof}[Proof of Theorem \ref{Théorème 6.2 KRAS}]
		For $x \in \Z^d$, we denote by $\pi(x)$ the first geodesic from $0$ to $x$ in the lexicographical order\footnote{The lexicographical order is based on the directions of the consecutive edges of the geodesics.} among those who have the minimal number of edges. We say that two patterns are disjoint if they have no vertex in common. We denote by $\cN^\mathfrak{P}(\pi(x))$ the maximum number of disjoint patterns defined above in the three cases visited by $\pi(x)$. Simple geometric considerations provide a constant $c>0$ such that for all path $\pi$, $\cN^\mathfrak{P}(\pi) \ge c N^\mathfrak{P}(\pi)$. Further, note that, in each pattern visited by $\pi(x)$, $\pi(x)$ takes the $\PP$ segment since the $\PP$ segment belongs to the set of optimal paths entirely contained in the pattern and is the only optimal path for the norm $\|.\|_1$. We can define a self-avoiding path $\widehat{\pi}(x)$ from $0$ to $x$ by replacing each $\PP$ segment of $\pi(x)$ with the $\PPP$ segment in each disjoint pattern visited by $\pi(x)$. This path $\widehat{\pi}(x)$ has the same passage time and hence both $\pi(x)$ and $\widehat{\pi}(x)$ are geodesics. Note that in the case where zero is not an atom, $\widehat{\pi}$ is obviously self-avoiding (since every geodesic is self avoiding), and in the case where zero is an atom, $\widehat{\pi}$ is self-avoiding since in the patterns, the passage times of all edges which are not in the $\PP$ or $\PPP$ segments are strictly positive, and thus $\pi(x)$ can not visit vertices in the $\PPP$ segments except those which also belong to the corresponding $\PP$ segment. In the three cases, we have: \[|\widehat{\pi}(x)| \ge |\pi(x)|+2\cN^\mathfrak{P}(\pi(x)).\] Finally, we get:
		\begin{equation}
			\Lo_{0,x} \ge |\widehat{\pi}(x)| \ge |\pi(x)|+2\cN^\mathfrak{P}(\pi(x)) \ge \Lu_{0,x}+2\cN^\mathfrak{P}(\pi(x)). \label{Inégalité preuve théorème 6.2 KRAS.}
		\end{equation} 
		The hypothesis of Theorem \ref{17. Théorème à démontrer.} are satisfied, thus there exist $\alpha>0$, $\beta_1>0$ and $\beta_2>0$ such that for all $x \in \Z^d$, \[ \P \left(\exists \mbox{ a geodesic $\gamma$ from $0$ to $x$ such that } N^\mathfrak{P}(\gamma) < \alpha \|x\|_1 \right) \le \beta_1 \mathrm{e}^{-\beta_2 \|x\|_1}. \]
		Then, taking $D=2 \alpha c$, we get by \eqref{Inégalité preuve théorème 6.2 KRAS.}, 
		\[\P \left(\Lo_{0,x}-\Lu_{0,x} \ge D \|x\|_1 \right) \ge \P \left(\cN^\mathfrak{P}(\pi(x)) \ge c \alpha \|x\|_1 \right) \ge \P \left(N^\mathfrak{P}(\pi(x)) \ge \alpha \|x\|_1 \right) \ge 1-\beta_1 \mathrm{e}^{-\beta_2 \|x\|_1}.\]	
	\end{proof}
	
	\subsection{Modification proof for the strict concavity of the expected passage times as a function of the weight shifts}
	
	\begin{proof}[Proof of Theorem \ref{Théorème 5.4 KRAS}.]
		Recall that we assume that the support of $F$ is bounded and that there are two different positive points in this support. We set $\ST=\sup(\text{support}(F))$. Let $0<r<s$ be two points in the support of $F$. Fix $b \in (0,r)$. Applying Lemma 5.5 in \cite{KRAS}, we get positive integers $k$, $\l$ fixed for the rest of the proof such that 
		\begin{equation}
			k(s+\delta) < (k+2\l)(r-\delta) < (k+2\l)(r+\delta) < k(s-\delta) + (2\l-1)b \label{Inégalités preuve Théorème 5.4 KRAS.}
		\end{equation}
		holds for sufficiently small $\delta>0$. Fix $L=k+\l+1$, $L_1=k+2L$, $L_2=\l+2L$ and if $d \ge 3$, for all $i \in \{3,\dots,d\}$, $L_i=2L$. We define a pattern in $\displaystyle \Lambda=\prod_{i=1}^{d} \{0,\dots,L_i\}$. We take the endpoints $\displaystyle \ulm = \sum_{i=2}^d L \epsilon_i$ and $\displaystyle \vlm = \ulm + (k+2L) \epsilon_1$. Then we define
		\begin{equation}
			\Lambda_0=\{L,\dots,k+L\} \times \{L,\dots,\l+L\} \times \prod_{i=3}^d \{L\}. \label{on fixe l lambda dans la preuve du théorème 5.4 de KRAS.}
		\end{equation}
		With this definition, every path from the boundary of $\Lambda$ to the boundary of $\Lambda_0$ has to take at least $L$ edges of $\Lambda \setminus \Lambda_0$.
		We take $\PP$ the path going from $u^\Lambda$ to $v^\Lambda$ by $k+2L$ steps in the direction $\epsilon_1$, and $\PPP$ the one going from $u^\Lambda$ to $u_1=\ulm+L\epsilon_1 \in \Lambda_0$ by $L$ steps in the direction $\epsilon_1$, then going to $u_2=u_1+\l\epsilon_2 \in \Lambda_0$ by $\l$ steps in the direction $\epsilon_2$, then to $u_3=u_2+k\epsilon_1 \in \Lambda_0$ by $k$ steps in the direction $\epsilon_1$, then to $u_4=u_3-\l\epsilon_2 \in \Lambda_0$ by $\l$ steps in the direction $-\epsilon_2$, and then to $v^\Lambda$ by $L$ steps in the direction $\epsilon_1$.
		
		For a deterministic family $(t_e)_{e\in \cE_\Lambda}$ of passage times on the edges of $\Lambda$ and for a path $\pi$, we use the abuse of writing $T(\pi)$ to denote $\displaystyle \sum_{e \in \pi} t_e$. For all $\delta\ge0$, we consider the set $G(\delta)$ of families $(t_e)_{e\in \cE_\Lambda}$ of passage times on the edges of $\Lambda$ which satisfy the following two conditions:
		\begin{itemize}
			\item for all $e \in \PPP \cap \Lambda_0$, $t_e \in [r-\delta,r+\delta]$, 
			\item for all other edges $e$ in $\Lambda$, $t_e \in [s-\delta,s+\delta]$.
		\end{itemize}
		
		Then, consider the set $H$ of families $(t_e)_{e\in \cE_\Lambda}$ such that:
		\begin{enumerate}[label=(P\arabic*)]
			\item\label{Propriété 1 théorème 5.4.} $\PP$ is the unique optimal path from $u^\Lambda$ to $v^\Lambda$ among the paths entirely contained in $\Lambda$,
			\item\label{Propriété 2 théorème 5.4.} for all path $\pi_1$ from a vertex $w_1$ of the boundary of $\Lambda_0$ to a vertex of the boundary of $\Lambda$, for all $w_2$ in $\Lambda_0$, for all path $\pi_2$, optimal for the norm $\|.\|_1$, going from $w_1$ to $w_2$, we have $T(\pi_2)<T(\pi_1)$.
		\end{enumerate}
		Note that for all $\delta > 0$, since $r$ and $s$ belong to the support of $F$, we have that $\P((T(e))_{e \in \cE_\Lambda} \in G(\delta)) > 0$.
		Let us prove that $G(0) \subset H$. Consider a family $(t_e)_{e\in \cE_\Lambda}\in G(0)$. So there are only two different passage times in $\Lambda$ which are $r$ and $s$. Assume for a contradiction that \ref{Propriété 1 théorème 5.4.} does not hold. Then there exists an optimal path $\pi$ going from $u^\Lambda$ to $v^\Lambda$, different from $\PP$. Recall the notation $T_i$ for $i \in \{1,\dots,d\}$ introduced in the proof of Lemma \ref{Lemme théorème 6.2 KRAS.}. Since $\PP$ is the unique path between $\ulm$ and $\vlm$ taking only edges in the direction $\epsilon_1$ and since there is no passage time equal to zero in $\Lambda$, we have $T_1(\pi) < T(\pi)$. Hence 
		\[T_1(\PP) = T(\PP) \ge T(\pi) > T_1(\pi).\]
		Since the only edges in the direction $\epsilon_1$ whose passage time is smaller than $s$ are in $\PPP_{u_2,u_3}$, $\pi$ takes an edge of $\PPP_{u_2,u_3}$ and thus $T_2(\pi) \ge 2 \l r$. Furthermore \[T_1(\pi) \ge T_1(\PP)-k(s-r).\] Hence we get \[T(\pi) \ge T_1(\pi) + T_2(\pi) \ge T_1(\PP) + 2\l r - k(s-r) > T_1(\PP)=T(\PP),\] where the strict inequality comes from  \eqref{Inégalités preuve Théorème 5.4 KRAS.}. Thus, it contradicts the fact that $\pi$ is an optimal path and \ref{Propriété 1 théorème 5.4.} holds.
		Now, to prove that \ref{Propriété 2 théorème 5.4.} holds, let $\pi_1$ be a path from a vertex $w_1$ of the boundary of $\Lambda_0$ to a vertex of the boundary of $\Lambda$ and let $w_2$ a vertex of $\Lambda_0$. Let $\pi_2$ be an optimal path for the norm $\|.\|_1$ going from $w_1$ to $w_2$. Then, by the definition of $\Lambda_0$ (see \eqref{on fixe l lambda dans la preuve du théorème 5.4 de KRAS.}), $\pi_1$ has to take at least $L$ edges whose passage time is equal to $s$ although $\pi_2$ takes at most $k+\l$ edges whose passage time is smaller than or equal to $s$. Thus \[T(\pi_1) \ge Ls > (k+\l) s \ge T(\pi_2,) \] where the strict inequality comes from the fact that $L=k+\l+1$. Hence \ref{Propriété 2 théorème 5.4.} holds and $G(0) \subset H$.
		
		Furthermore, $H$ is an open set since for a family $(t_e)_{e\in \cE_\Lambda}$ to belong to $H$, it is required that the time of a finite family of paths is strictly smaller than the time of every path of another finite family of paths. Hence, for $\delta>0$ small enough, we have 
		\begin{equation}
			G(\delta) \subset H. \label{théorème 5.4 preuve G delta contenu dans H.}
		\end{equation}
		Fix $\delta>0$ such that \eqref{Inégalités preuve Théorème 5.4 KRAS.} and \eqref{théorème 5.4 preuve G delta contenu dans H.} hold.
		Consider the pattern $\mathfrak{P}=(\Lambda,u^\Lambda,v^\Lambda,\cA^\Lambda)$ with $\cA^\Lambda=\{(T(e))_{e \in \cE_\Lambda} \in G(\delta)\}$. Now, we denote by $\pi(x)$ the first geodesic from $0$ to $x$ in the lexicographical order among those of maximal Euclidean length and we denote by $\cN^\mathfrak{P}(\pi(x))$ the maximum number of disjoint patterns visited by $\pi(x)$. Recall the existence of a constant $c>0$ small enough such that for all path $\pi$, $\cN^\mathfrak{P}(\pi) \ge c N^\mathfrak{P}(\pi)$. Since $\P(\cA^\Lambda)>0$, we can apply Theorem \ref{17. Théorème à démontrer.}. Let $\alpha, \beta_1, \beta_2 > 0$ be the constants given by Theorem \ref{17. Théorème à démontrer.}. Then, we have
		\begin{align*}
			\E[\cN^\mathfrak{P}(\pi(x))] & \ge \lfloor c \alpha \|x\|_1 \rfloor \P(\cN^\mathfrak{P}(\pi(x)) \ge c \alpha \|x\|_1) \\
			& \ge \lfloor c \alpha \|x\|_1 \rfloor \P(N^\mathfrak{P}(\pi(x)) \ge \alpha \|x\|_1) \ge \lfloor c \alpha \|x\|_1 \rfloor (1 - \beta_1 \mathrm{e}^{-\beta_2 \|x\|_1}) \ge C \|x\|_1.
		\end{align*}
		Now, let us follow the end of Stage 3 of the proof of Theorem 5.4 in \cite{KRAS}. By \ref{Propriété 1 théorème 5.4.}, $\pi(x)$ takes the $\PP$ segment in each pattern that it takes. Furthermore, by \ref{Propriété 2 théorème 5.4.}, $\pi(x)$ does not take any edge in the $\PPP$ segment which is not in the $\PP$ segment. So, we can define a self-avoiding path $\widehat{\pi}(x)$ from $0$ to $x$ by replacing each $\PP$ segment with the $\PPP$ segment in each pattern visited by $\pi(x)$. Reduce the weights on each edge $e$ from $T(e)$ to $T^{(-b)}(e)=T(e)-b$. By the definition of the pattern, the $T^{(-b)}$-passage times of the segments $\PP$ and $\PPP$ obey the inequality:
		\[T^{(-b)}(\PPP) = T(\PPP) - b|\PPP| < T(\PP) + (2\l-1)b - b|\PPP|=T^{(-b)}(\PP)-b. \]
		Then, following the proof of Theorem 5.4 in \cite{KRAS}, we get
		\begin{align}
			t^{(-b)}(0,x) & \le T^{(-b)}(\widehat{\pi}(x)) < T^{(-b)}(\pi(x)) - b \cN^\mathfrak{P}(\pi(x)) \nonumber \\
			& = T(\pi(x)) - b|\pi(x)|-b\cN^\mathfrak{P}(\pi(x)) \nonumber \\
			& =t(0,x)-b \Lo_{0,x}-b\cN^\mathfrak{P}(\pi(x)).\label{Dernière inégalité preuve théorème 5.4 KRAS.}
		\end{align}
		Since $b \in (0,\r + \epsilon_0)$, $\E[t^{(-b)}(0,x)]$ is finite. Thus, taking expectation in \eqref{Dernière inégalité preuve théorème 5.4 KRAS.}, we get the result.
	\end{proof}
	
	\appendix
	
	
	
	\section{Unbounded case}
	
	\subsection{Assumptions on the pattern in the proof of the unbounded case}\label{Annexe sur l'hypothèse sur les motif dans le cas non borné.}
	
	\begin{lemma}\label{Lemme annexe A.1 sur les overlapping pattern.}
		Let $\mathfrak{P}=(\Lambda,u^\Lambda,v^\Lambda,\cA^\Lambda)$ be a valid pattern.
		There exists a positive integer $\lll$ and a valid pattern $\mathfrak{P}_0=(\Lambda_0,u^\Lambda_0,v^\Lambda_0,\cA^\Lambda_0)$ such that:
		\begin{itemize}
			\item $\Lambda_0=B_\infty(0,\lll)$,
			\item for every self-avoiding path $\pi$, $\displaystyle N^\mathfrak{P}(\pi) \ge N^{\mathfrak{P}_0}(\pi)$.
		\end{itemize} 
	\end{lemma}
	
	\begin{proof}
		Denote by $L_1,\dots,L_d$ the integers such that $\displaystyle \Lambda=\prod_{i=1}^{d} \{0,\dots,L_i\}$.
		Fix $\lll=\max(L_1,\dots,L_d)$.
		Let $M^\Lambda>0$ such that 
		\[\P \left( \cA^\Lambda \cap \{\forall e \in \Lambda, \, T(e) \le M^\Lambda\} \right) > 0.\]
		Consider the pattern $\mathfrak{P}_0=(\Lambda_0,u^\Lambda_0,v^\Lambda_0,\cA^\Lambda_0)$ defined by:
		\begin{itemize}
			\item $\Lambda_0=B_\infty(0,\lll)$
			\item $u^\Lambda_0$ (resp. $v^\Lambda_0$) one vertex of $\partial \Lambda_0$ such that there exists a path $\pi_u$ (resp. $\pi_v$) fixed for the remaining of the proof linking $u^\Lambda_0$ and $\ulm$ (resp. $v^\Lambda_0$ and $\vlm$) which satisfies the following two conditions:
			\begin{itemize}
				\item $\pi_u$ (resp. $\pi_v$) does not visit any vertex of $\Lambda$ except $\ulm$ (resp. $\vlm$),
				\item $\pi_u$ (resp. $\pi_v$) only takes edges in the direction of one external normal unit vector associated with $\ulm$ (resp. $\vlm$) chosen arbitrarily if there are several such vectors,
			\end{itemize}
			\item $\cA^\Lambda_0$ the event such that:
			\begin{itemize}
				\item $\cA^\Lambda \cap \{\forall e \in \Lambda, \, T(e) \le M^\Lambda\}$ occurs,
				\item for all $e$ belonging to $\pi_u \cup \pi_v$, we have $T(e) \le M^\Lambda$,
				\item for all $e$ which does not belong to $\Lambda \cup \pi_u \cup \pi_v$, \[T(e) > |\Lambda_0|_e M^\Lambda,\]
				where $|\Lambda_0|_e$ is the number of edges in $\Lambda_0$.
			\end{itemize}
		\end{itemize}
		Note that, since $\ulm$ and $\vlm$ are distinct, $\pi_u$ and $\pi_v$ are disjoint. We get:
		\begin{itemize}
			\item $\P(\cA^\Lambda_0)$ is positive since $\mathfrak{P}$ is valid and since the support of $F$ is unbounded, and then $\mathfrak{P}_0$ is a valid pattern.
			\item On $\cA^\Lambda_0$, any path from $\ulm_0$ to $\vlm_0$ optimal for the passage time among the paths entirely inside $\Lambda_0$ contains a subpath from $\ulm$ to $\vlm$ entirely inside $\Lambda$. Indeed, let $\pi_0$ be a path from $\ulm_0$ to $\vlm_0$ which does not contain a subpath from $\ulm$ to $\vlm$ entirely inside $\Lambda$. It implies that $\pi_0$ takes an edge whose time is greater than $|\Lambda_0|_e M^\Lambda$. Let $\pi'_0$ be a path following $\pi_u$, then going from $\ulm$ to $\vlm$ inside $\Lambda$, and then following $\pi_v$. We have $T(\pi'_0) \le |\Lambda_0|_e M^\Lambda < T(\pi_0)$ and thus $\pi_0$ is not an optimal path. Hence, for every path $\pi$, if a vertex $x \in \Z^d$ satisfies the condition $(\pi;\mathfrak{P}_0)$, $x$ satisfies the condition $(\pi;\mathfrak{P})$. 
		\end{itemize}
		
		Thus we get that $\displaystyle N^\mathfrak{P}(\pi) \ge N^{\mathfrak{P}_0}(\pi)$.
	\end{proof}
	
	\subsection{Construction of the path $\pi$ for the modification in the unbounded case}\label{Annexe construction de pi pour le cas non borné.}
	
	Recall that, here, $u^\Lambda$ and $v^\Lambda$ are the vertices defined at the beginning of Section \ref{Sous-section modification argument, cas non borné}. The aim is to prove that we can construct a path $\pi$ in a deterministic way such that:
	\begin{enumerate}[label=(\roman*)]
		\item $\pi$ goes from $u$ to $u^\Lambda$ without visiting a vertex of $B_\infty(sN,\lll)$, then going from $u^\Lambda$ to $v^\Lambda$ in a shortest way for the norm $\|.\|_1$ (and thus being contained in $B_\infty(sN,\lll)$) and then goes from $v^\Lambda$ to $v$ without visiting a vertex of $B_\infty(sN,\lll)$,
		\item $\pi$ is entirely contained in $\Bds$ and does not have vertices on the boundary of $\Bds$ except $u$ and $v$, 
		\item $\pi$ is self-avoiding,
		\item the length of $\pi_{u,u^\Lambda}  \cup \pi_{v^\Lambda,v}$ is bounded from above by $2r_2 N+\Ka$, where $\Ka$ is the number of edges in $B_\infty(0,\lll+3)$.
	\end{enumerate}
	
	As it is said in Section \ref{Sous-section modification argument, cas non borné}, we want to construct a path from $u$ to $sN$ and a path from $sN$ to $v$ which have no vertex in common except $sN$ and such that their lengths are bounded from above by $r_2 N$. To get them, we use the following lemma whose proof is left to the reader.
	
	\begin{lemme}
		Let $m \in \N^*$ and $x$, $y$ two vertices of $\Z^d$ such that $\|x\|_1=\|y\|_1=m$ and $x \neq y$. Then we can build in a deterministic way two paths $\pi_x$ and $\pi_y$ linking respectively $x$ and $y$ to $0$ and such that their length is equal to $m$, they have only $0$ as a common vertex and they have respectively only $x$ and $y$ as vertices whose norm $\|.\|_1$ is greater than or equal to $m$.
	\end{lemme}
	
	
	Using this lemma and replacing $0$ by $sN$ using a translation, we get two paths linking $u$ to $sN$ and $v$ to $sN$ with the stated properties. Recall that $B_\infty(sN,\lll+3) \subset \Bds$ and let $u_0$ (resp. $v_0$) denote the first vertex in $B_\infty(sN,\lll+3)$ visited by the path going from $u$ to $sN$ (resp. the one going from $v$ to $sN$). Then we get two paths $\pi_{u,u_0}$ and $\pi_{v_0,v}$ respectively from $u$ to $u_0$ and from $v_0$ to $v$ both constructed in a deterministic way such that $\pi_{u,u_0}$ and $\pi_{v_0,v}$ do not have any vertex in common, are entirely contained in $\Bds$, have only $u$ or $v$ as points on the boundary of $\Bds$, and their lengths are bounded from above by $r_2 N - (\lll+3)$.

	\begin{figure}
		\begin{center}
			\begin{tikzpicture}[scale=0.6]
				\draw[dashed] (0,0) rectangle (10,10);
				\draw[dashed] (0.5,0.5) rectangle (9.5,9.5);
				\draw (1.5,1.5) rectangle (8.5,8.5);
				\draw[line width=1.2pt,color=Green] (5,10) -- (5,9.5) -- (0.5,9.5) -- (0.5,0.5) -- (8,0.5) -- (8,1.5);
				\draw[line width=1.2pt,color=Red] (0,6.5) -- (0,0) -- (10,0) -- (10,5.5) -- (8.5,5.5); 
				\draw (8.5,8.25) node[below,left,scale=0.5] {$B_\infty(sN,\lll)$};
				\draw (9.5,9.2) node[above,left,scale=0.5] {$B_\infty(sN,\lll+2)$};
				\draw (10,9.75) node[above,left,scale=0.5] {$B_\infty(sN,\lll+3)$};
				\draw[color=Red] (0,6.5) node[left,scale=0.8] {$u_0$};
				\draw[color=Red] (10,5.5) node[right,scale=0.8] {$u^{\Lambda,0}$};
				\draw[color=Red] (8.5,5.5) node[left,scale=0.8] {$u^{\Lambda}$};
				\draw[color=Red] (2,0) node[below,scale=1] {$\pi_{u_0,u^\Lambda}$};
				\draw[color=Green] (5,10) node[above,scale=0.8] {$v_0$};
				\draw[color=Green] (8,0.5) node[below,right,fill=white,scale=0.8] {$v^{\Lambda,0}$};
				\draw[color=Green] (8,1.5) node[above,scale=0.8] {$v^{\Lambda}$};
				\draw[color=Green] (3,9.5) node[below,scale=1] {$\pi_{v^\Lambda,v_0}$};
				\draw[dashed] (15,0) rectangle (25,10);
				\draw[dashed] (15.5,0.5) rectangle (24.5,9.5);
				\draw[dashed] (16,1) rectangle (24,9);
				\draw (16.5,1.5) rectangle (23.5,8.5);
				\draw[line width=1.4pt,color=gray] (12.5,-2) -- (12.5,12);
				\draw[line width=1.2pt,color=Green] (17,0) -- (25,0) -- (25,10) -- (15,10) -- (15,6) -- (16,6) -- (16,5.5) -- (16.5,5.5);
				\draw[line width=1.2pt,color=Red] (15,5.5) -- (15,0) -- (16.5,0) -- (16.5,0.5) -- (17,0.5) -- (17,1.5);
				\draw (23.5,8.25) node[below,left,scale=0.5] {$B_\infty(sN,\lll)$};
				\draw (24,8.75) node[above,left,scale=0.5] {$B_\infty(sN,\lll+1)$};
				\draw (24.5,9.25) node[above,left,scale=0.5] {$B_\infty(sN,\lll+2)$};
				\draw (25,9.75) node[above,left,scale=0.5] {$B_\infty(sN,\lll+3)$};
				\draw[color=Red] (15,5.5) node[left,scale=0.8] {$u_0$};
				\draw[color=Red] (16.5,0) node[below,scale=0.8] {$u''_0$};
				\draw[color=Red] (17,1.5) node[above,scale=0.8] {$u^{\Lambda}$};
				\draw[color=Red] (15,3) node[left,scale=1] {$\pi_{u_0,u^\Lambda}$};
				\draw[color=Green] (17.2,-0.1) node[below,scale=0.8] {$v_0$};
				\draw[color=Green] (15,6.1) node[left,scale=0.8] {$v''_0$};
				\draw[color=Green] (16.5,5.5) node[right,scale=0.8] {$v^{\Lambda}$};
				\draw[color=Green] (21,0) node[below,scale=1] {$\pi_{v^\Lambda,v_0}$};
			\end{tikzpicture}
			\caption{Example of the construction of the portion of $\pi$ which is contained in $B_\infty(sN,\lll+3) \setminus B_\infty(sN,\lll)$ in dimension $2$. On the left, this is an example of the case where $\|u_0-v^\Lambda\|_1>3$ or $\|v_0-u^\Lambda\|_1>3$ and on the left the other case.}\label{Figure de l'annexe de la construction du chemin pi. }
		\end{center}
	\end{figure}
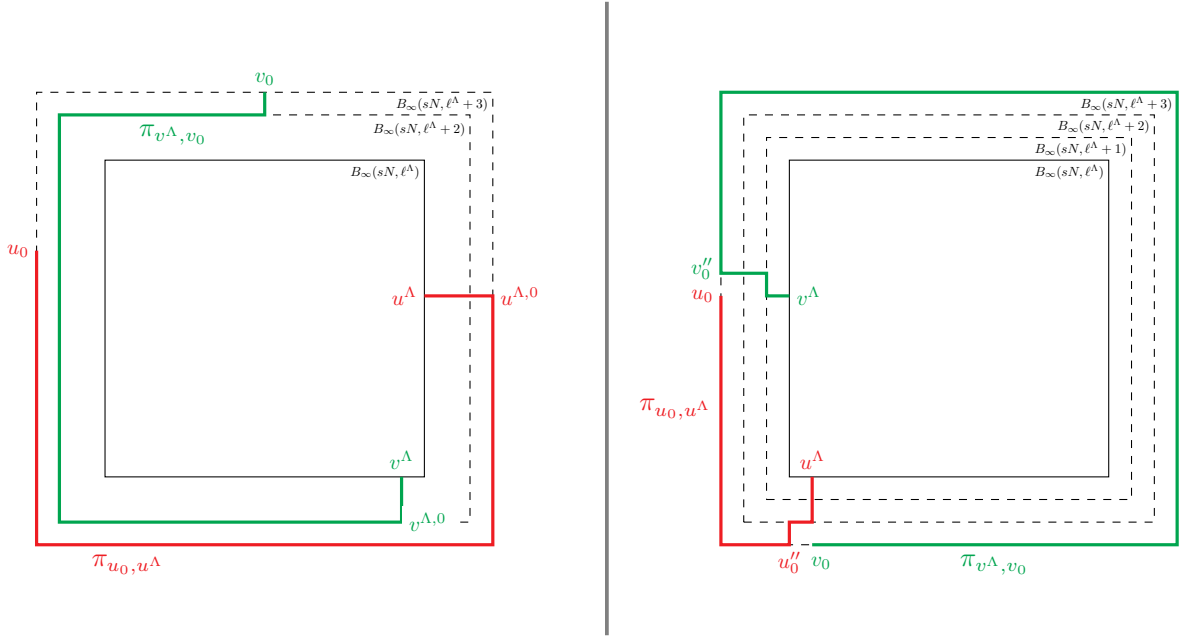 
	
	Then we build two paths $\pi_{u_0,u^\Lambda}$ and $\pi_{v^\Lambda,v_0}$ respectively from $u_0$ to $u^\Lambda$ and from $v^\Lambda$ to $v_0$ contained in $B_\infty(sN,\lll+3) \setminus B_\infty(sN,\lll)$ except for $u^\Lambda$ and $v^\Lambda$, such that they do not have any vertex in common. By the definition of $\Ka$, this implies that the sum of their lengths is bounded from above by $\Ka$ (recall that $\Ka$ is the number of edges in $B_\infty(0,\lll+3)$).
	To get these paths, assume first that $\|u_0-v^\Lambda\|_1>3$ or $\|v_0-u^\Lambda\|_1>3$. Assume that we have $\|v_0-u^\Lambda\|_1>3$, the other case being the same.
	We begin by considering a path $\tpm$ going in a shortest way from $v_0$ to a vertex of the boundary of $B_\infty(sN,\lll+2)$, denoted by $v'_0$. This path has at most $d$ edges.
	Then, let $u^{\Lambda,0}$ be a vertex on the boundary of  $B_\infty(sN,\lll+3)$ such that $\|u^\Lambda-u^{\Lambda,0}\|_1=3$. We get $\pi_{u_0,u^\Lambda}$ by going from $u_0$ to $u^{\Lambda,0}$ in a shortest way on the boundary of $B_\infty(sN,\lll+3)$ avoiding all vertices of $\tpm$, and then by going from $u^{\Lambda,0}$ to $u^\Lambda$ by three steps. This is possible since $\|v_0-u^\Lambda\|_1>3$.
	To get $\pi_{v_0,v^\Lambda}$, we begin by following $\tpm$. Let $v^{\Lambda,0}$ be a vertex on the boundary of $B_\infty(sN,\lll+2)$ such that $\|v^\Lambda-v^{\Lambda,0}\|_1=2$. Then, $\pi_{v_0,v^\Lambda}$ goes from $v'_0$ to $v^{\Lambda,0}$ in a shortest way on the boundary of $B_\infty(sN,\lll+2)$ avoiding the unique vertex of $\pi_{u_0,v^\Lambda}$ belonging to $B_\infty(sN,\lll+2)$. Finally, $\pi_{v_0,v^\Lambda}$ goes from $v^{\Lambda,0}$ to $v^\Lambda$ by two steps (see the left side of Figure \ref{Figure de l'annexe de la construction du chemin pi. } for an example of this construction in dimension $2$).
	
	Now, if $\|u_0-v^\Lambda\|_1=3$ and $\|v_0-u^\Lambda\|_1=3$, the construction has to be slightly different. From $u_0$, $\pi_{u_0,u^\Lambda}$ goes to a vertex $u''_0$ belonging to the boundary of $B_\infty(sN,\lll+3)$ and such that $\|u''_0-v_0\|_1=1$ in a shortest way on the boundary of $B_\infty(sN,\lll+3)$. After, it makes one step to go to the boundary of $B_\infty(sN, \lll + 2)$ and goes to $u^\Lambda$ in a shortest way by taking only one vertex in $B_\infty(sN,\lll+1)$. Then, $\pi_{v_0,v^\Lambda}$ goes on the boundary of $B_\infty(sN,\lll+3)$ to a vertex $v''_0$ belonging to the boundary of $B_\infty(sN,\lll+3)$ and such that $\|v''_0-u_0\|_1=1$ by avoiding every vertex which belongs to $\pi_{u_0,u^\Lambda}$, and then makes two steps to go to the boundary of $B_\infty(sN,\lll+1)$ and goes to $v^\Lambda$ in a shortest way (see the right side of Figure \ref{Figure de l'annexe de la construction du chemin pi. } for an example of this construction in dimension $2$). In this case, we also have that the sum of their lengths is bounded from above by $\Ka$.
	
	Finally, $\pi$ is the path obtained by concatenating $\pi_{u,u_0}$, $\pi_{u_0,u^\Lambda}$, a path going from $u^\Lambda$ to $v^\Lambda$ in a shortest way for the norm $\|.\|_1$, $\pi_{v^\Lambda,v_0}$ and $\pi_{v_0,v}$ in this order. We have that $\pi$ is a self-avoiding path contained in $\Bds$ and has only $u$ and $v$ on the boundary of $\Bds$.
	
	\section{Overlapping pattern in the bounded case}\label{Annexe overlapping pattern}
	
	\begin{proof}[Proof of Lemma \ref{17. Lemme motif valable}]
		Recall that $\nu$ is fixed at the beginning of Section \ref{Section cas borné.} and that $\displaystyle \delta'=\min \left(\frac{\delta}{8}, \frac{\delta}{1+d} \right)$. Then, we fix a positive real $\nu_0$ such that: 
		\begin{itemize}[label=\textbullet]
			\item $\r+\delta \le \nu_0 \le \nu \le \ST$,
			\item the event $\cA^\Lambda \cap \{\forall e \in \Lambda, \, t_e \le \nu_0\}$ has a positive probability, 
			\item $F([\nu_0,\nu])>0$. 
		\end{itemize}
		Notice that, if $F$ has an atom, one could have $\nu_0=\nu$, or even $\nu_0=\nu=\ST$. Then, fix
		\begin{equation}
			\lll = \max(L_1,\dots,L_d)  , \label{17. on fixe llambda}
		\end{equation}
		
		\begin{equation}
			\l_1 > \frac{4d\lll\left(\nu_0-\r-\frac{\delta'}{2}\right)}{\nu_0-\r-\delta'}  , \label{17. on fixe l2}
		\end{equation}
		
		\begin{equation}
			\mbox{and } \l_0 > \frac{\lll(\left(2d+1) \nu_0 + (2d-1)\r + 2d\delta'\right) + \l_1(\nu_0+3\r+4\delta')}{\nu_0-\r-2\delta'}. \label{17. on fixe l1}
		\end{equation}
		Fix also 
		\begin{equation}
			\delta'' < \min \left( \delta', \frac{\nu_0-\r}{2d\l_0} \right). \label{17. on fixe delta''}
		\end{equation}
		
		Let $j \in \{1,...,d\}$. Let us construct the overlapping pattern in $\Lambda_0 = \{-\l_0,...,\l_0\}^d$ with endpoints $\l_0 \epsilon_j$ and $- \l_0 \epsilon_j$. We denote by $u_3$ and $v_3$ the endpoints of the original pattern. The first step is the construction of an oriented path going from the face of $\Lambda_0$ containing $\l_0 \epsilon_j$ to the face containing $-\l_0 \epsilon_j$ and visiting $u_3$ and $v_3$. We denote by $\Lambda_1$ the set $\displaystyle \prod_{i=1}^d \{-\l_1,\dots,L_i+\l_1\}$. Then we define $u_2$ (resp. $v_2$) the vertex of $\partial \Lambda_1$ which can be linked to $u_3$ (resp. $v_3$) by a path using exactly $\l_1$ edges in only one direction (the direction of the external normal vector associated to $u_3$ (resp. $v_3$)). If $u_3$ or $v_3$ are associated to several external normal unit vectors, then the choice of the direction is not unique. We choose one external normal unit vector in an arbitrary way but we ensure that the two normal unit vectors chosen to build $u_2$ and $v_2$ are different. Note that this is possible since the pattern is valid (recall Definition \ref{Définition motif valable.}). For a vertex $z \in \Z^d$, denote by $z(j)$ its $j$-th coordinate. Even if it means exchanging the roles of $u_3$ and $v_3$, we can assume that $u_2(j) \ge v_2(j)$. Then, we define $u_1 = u_2 + (\l_0-u_2(j)) \epsilon_j$ and $v_1 = v_2 - (\l_0+v_2(j)) \epsilon_j$. Note that $u_1$ (resp. $v_1$) belongs to the face of $\Lambda_0$ containing $\l_0 \epsilon_j$ (resp. $-\l_0 \epsilon_j$).
		We define a path $\tpm$ going from $\l_0 \epsilon_j$ to $u_1$ in a shortest way, then to $u_2$ in the shortest way, then to $u_3$ in the shortest way, then to $v_3$ in a shortest way, then to $v_2$ in the shortest way, then to $v_1$ in the shortest way and to $v_0$ in a shortest way. Note that $\tpm_{u_1,v_1}$ is an oriented path. Indeed, $\tpm_{u_3,v_3}$ is an oriented path. Extending it outside $\Lambda$ following the direction of an external normal unit vector preserves the fact that it is oriented. Thus $\tpm_{u_2,v_2}$ is an oriented path. Then, we assume that $u_2(j) \ge v_2(j)$ on purpose to guarantee that, with the definitions of $u_1$ and $v_1$, the path $\tpm_{u_1,v_1}$ remains an oriented path.

		\begin{figure}
			\begin{center}
				\begin{tikzpicture}[scale=1.3]
					\draw (-3,-3) rectangle (3,3);
					\draw (-3,3) node[below right] {$\Lambda_0$};
					\draw[pattern= north west lines, pattern color=gray!70] (0,0) rectangle (1,0.6);
					\draw (0,0.6) node[below right] {$\Lambda$};
					\draw[dashed] (-0.8,-0.8) rectangle (1.8,1.4);
					\draw (-0.8,1.4) node[below right] {$\Lambda_1$};
					\draw[line width=1.5pt] (0,3) -- (1.8,3) -- (1.8,0.2) -- (1,0.2);
					\draw[line width=1.5pt] (0,-3) -- (0.7,-3) -- (0.7,0);
					\draw[line width=1.5pt] (0.7,0) -- (0.7,0.2) -- (1,0.2);
					\draw (0,3) node[above] {$\l_0 \epsilon_2$};
					\draw (0,3) node {$\bullet$};
					\draw (1.8,3) node[above right] {$u_1$};
					\draw (1.8,3) node {$\bullet$};
					\draw (1.8,0.2) node[right] {$u_2$};
					\draw (1.8,0.2) node {$\bullet$};
					\draw (1,0.2) node[above right] {$u_3$};
					\draw (1,0.2) node {$\bullet$};
					\draw (0,-3) node[below] {$-\l_0 \epsilon_2$};
					\draw (0,-3) node {$\bullet$};
					\draw (0.7,-3) node[below right] {$v_1$};
					\draw (0.7,-3) node {$\bullet$};
					\draw (0.7,-0.8) node[below left] {$v_2$};
					\draw (0.7,-0.8) node {$\bullet$};
					\draw (0.7,0) node[below left] {$v_3$};
					\draw (0.7,0) node {$\bullet$};
					\draw (1.8,1.65) node[right] {$\tpm$};
					\draw [<->] (1,3.3) -- (1.8,3.3) node[midway,above] {$\l_1$};
					\draw [<->] (3.3,3) -- (3.3,0) node[midway,above,sloped] {$\l_0$};
				\end{tikzpicture}
				\caption{Example of the construction of $\tpm$ in an overlapping pattern in dimension $2$.}\label{Figure du overlapping pattern.}
			\end{center}
		\end{figure}
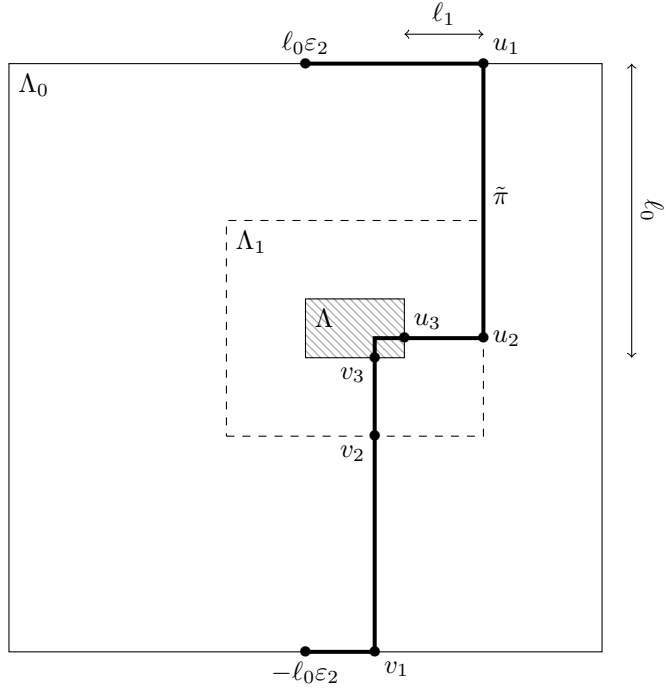
		
		The event (whose probability is positive) of the overlapping pattern, denoted by $\cA^{\Lambda_0}_j$ is the following:  
		
		\begin{itemize}
			\item for all $e \in \tpm_{\l_0 \epsilon_j,u_3}$ and $e \in \tpm_{v_3,-\l_0 \epsilon_j}$, $T(e) < \r + \delta''$,
			\item the event $\cA^\Lambda \cap \{ \forall e \in \Lambda, \, T(e) \le \nu_0 \}$ is realized,
			\item for all $e$ in $\Lambda_0 \setminus (\tpm \cup \Lambda)$, $\nu_0 \le T(e) \le \nu$.
		\end{itemize}
		
		On this event, let $\OG$ be one of the fastest path from $\l_0 \epsilon_j$ to $-\l_0 \epsilon_j$ among the path entirely contained in $\Lambda_0$ and let us show that $\OG$ visits $u_3$ and $v_3$ and that $\OG_{u_3,v_3}$ is entirely contained in $\Lambda$. We proceed by proving successive properties.
		
		\begin{itemize}
			\item \textit{There exist one vertex $a_0$ of $\tpm_{u_1,u_2}$ and one vertex $b_0$ of $\tpm_{v_2,v_1}$ visited by $\OG$. Further, $\OG_{\l_0 \epsilon_j,a_0}=\tpm_{\l_0 \epsilon_j,a_0}$ and $\OG_{b_0,-\l_0 \epsilon_j}=\tpm_{b_0,-\l_0 \epsilon_j}$.}
			
			\medskip
			
			Let us assume that $\OG$ does not visit any vertex of $\tpm_{u_1,u_2}$, the other case being the same. The path $\OG$ has to take at least $\l_0-\lll-\l_1$ edges connecting vertices such that at least one of them has its $j$-th coordinate strictly between $\lll+\l_1$ and $\l_0$. The only edges in this set whose passage time is smaller than $\nu_0$ are those of $\tpm_{u_1,u_2}$. Hence \[T(\OG) \ge (\l_0-\lll-\l_1) \nu_0 + (\l_0 + \lll + \l_1)\r. \]
			But, thanks to our construction, we have
			\begin{align*}
				T(\tpm) & \le \underbrace{T(\tpm_{\l_0 \epsilon_j,u_1})+T(\tpm_{v_1,-\l_0 \epsilon_j})}_{\le (2d\lll+2\l_1)(\r+\delta'')} + \underbrace{T(\tpm_{u_1,u_2}) + T(\tpm_{v_2,v_1})}_{\le 2 \l_0 (\r+\delta'')} + \underbrace{T(\tpm_{u_2,u_3}) + T(\tpm_{v_3,v_2})}_{\le 2 \l_1 (\r+\delta'')} + \underbrace{T(\tpm_{u_3,v_3})}_{\le 2d\lll\nu_0} \\
				& \le 2(d\lll+2\l_1+\l_0)(\r+\delta') + 2d\lll\nu_0 \text{ since $\delta'' < \delta'$.}
			\end{align*}
			
			So, \eqref{17. on fixe l1} leads to $T(\tpm) < T(\OG)$, which is impossible.
			
			To prove that $\OG_{\l_0 \epsilon_j,a_0}=\tpm_{\l_0 \epsilon_j,a_0}$, we use the fact that by the construction, if $\OG_{\l_0 \epsilon_j,a_0}\ne\tpm_{\l_0 \epsilon_j,a_0}$, $\OG_{\l_0 \epsilon_j,a_0}$ has to take at least one edge whose time is greater than or equal to $\nu_0$. Hence, \[T(\OG_{\l_0 \epsilon_j,a_0}) \ge \nu_0 + (\|\l_0 \epsilon_j-a_0\|_1-1)\r,\] although 
			\[T(\tpm_{\l_0 \epsilon_j,a_0}) \le \|\l_0 \epsilon_j-a_0\|_1 (\r + \delta'').\]
			Since $\|\l_0 \epsilon_j-a_0\|_1 \le 2d\l_0$, \eqref{17. on fixe delta''} leads to $T(\tpm_{\l_0 \epsilon_j,a_0}) < T(\OG_{\l_0 \epsilon_j,a_0})$, which is impossible sing $\OG$ is an optimal path. 
			
			\item \textit{If $a_1$ (resp. $b_1$) is a vertex of $\tpm_{u_2,u_3}$ (resp. $\tpm_{v_3,v_2})$ visited by $\OG$, then $\OG_{\l_0 \epsilon_j,a_1}=\tpm_{\l_0 \epsilon_j,a_1}$ (resp. $\OG_{b_1,-\l_0 \epsilon_j}=\tpm_{b_1,-\l_0 \epsilon_j}$).}
			
			\medskip
			
			Indeed, let $a_1$ be such a vertex and let $a_0$ be the vertex of the preceding property. We only have to prove that $\OG_{a_0,a_1}=\tpm_{a_0,a_1}$ and the proof is the same as the one for $\OG_{\l_0 \epsilon_j,a_0}=\tpm_{\l_0 \epsilon_j,a_0}$.
		\end{itemize}
		
		Among the vertices visited by $\OG$, we denote by $a$ (resp. $b$) the last vertex of $\tpm_{\l_0 \epsilon_j,u_3}$ (resp. the first vertex of $\tpm_{v_3,-\l_0 \epsilon_j}$). Then, $\OG_{a,b}$ does not visit any vertex of $\tpm_{\l_0 \epsilon_j,u_3} \cup \tpm_{v_3,-\l_0 \epsilon_j}$ (except $a$ and $b$). Indeed, by the two properties above, we have that $\OG_{\l_0 \epsilon_j,a}=\tpm_{\l_0 \epsilon_j,a}$ and $\OG_{b,-\l_0 \epsilon_j}=\tpm_{b,-\l_0 \epsilon_j}$ and $\OG_{a,b}$ can not visit a vertex of $\tpm_{a,u_3} \cup \tpm_{v_3,b}$ thanks to the definition of $a$ and $b$. 
		
		\begin{itemize}
			\item \textit{The vertex $a$ (resp. $b$) belongs to $\tpm_{u_2,u_3}$ (resp. $\tpm_{v_3,v_2}$).}
			
			\medskip
			
			Assume that $a$ or $b$ does not satisfy this property. Then $\| a - u_3 \|_1 \ge \l_1$ or $\| b - v_3 \|_1 \ge \l_1$. Since $\tpm_{a,b}$ is oriented and since $\|u_3-v_3\|_1 \le 2d\lll$, \[T(\tpm_{a,b}) \le (\|a-b\|_1-2d\lll)(\r+\delta'')+2d\lll \nu_0 \le (\|a-b\|_1-2d\lll)(\r+\delta')+2d\lll \nu_0.\]
			
			Then, since the edges whose time is smaller than $\nu_0$ taken by $\OG_{a,b}$ are those in $\Lambda$, \[T(\OG_{a,b}) \ge (\|a-b\|_1-2d\lll) \nu_0 + 2d\lll \r.\]
			
			Using \eqref{17. on fixe l2}, we get $T(\OG_{a,b}) > T(\tpm_{a,b})$, which is impossible.
			\item \textit{We have that $a=u_3$ and $b=v_3$.}
			
			\medskip
			Assume that $a \neq u_3$, the other case being the same. 
			We have \[T(\tpm_{a,b}) \le (\|a-b\|_1-\|u_3-v_3\|_1)(\r+\delta') + \|u_3-v_3\|_1 \nu_0,\] and thus $\OG_{a,b}$ takes at least one edge of $\Lambda$ otherwise $T(\OG_{a,b}) \ge \|a-b\|_1 \nu_0 > T(\tpm_{a,b})$ since $\|a-b\|_1-\|u_3-v_3\|_1 > 0$.
			Let $u_0$ be the first entry point of $\OG_{a,b}$ in $\Lambda$ and let us consider the path $\pi$ following $\tpm_{a,u_3}$, then going from $u_3$ to $u_0$ in a shortest way and then following $\OG_{u_0,b}$. Then, the number of edges of $\pi_{a,u_0}$ is lower than or equal to the number of edges of $\OG_{a,u_0}$, for all $e \in \pi_{a,u_0}$, $T(e) \le \nu_0$, there exists $e' \in \pi_{a,u_0}$ such that $T(e') \le \r + \delta'$ and for all $e \in \OG_{a,u_0}$, $T(e) \ge \nu_0$. So, we have $T(\pi) < T(\OG_{a,b})$ which is impossible since $\OG$ is an optimal path among paths entirely contained in $\Lambda_0$. 
			
			\item \textit{$\OG_{a,b}$ takes the original pattern.}
			
			\medskip
			
			Assume that $\OG_{a,b}$ is not entirely contained in $\Lambda$. Let $v_0$ be the first exit point from $\Lambda$ of $\OG_{a,b}$ and $u_0$ the first entry point after $v_0$. Let us consider the shortcut $\pi$ going from $v_0$ to $u_0$ in a shortest way. Then, using the same argument as in the proof of Lemma \ref{17. Toute géodésique qui touche pi avant et après le motif emprunte le motif.}, we have that $\OG_{v_0,u_0}$ has strictly more edges than $\pi$. Furthermore, all edges of $\OG_{v_0,u_0}$ have a time greater than or equal to $\nu_0$ although all edges of $\pi$ have a time lower than or equal to $\nu_0$. So, $T(\pi)<T(\OG_{v_0,u_0})$, which is impossible since $\OG$ is an optimal path among paths entirely contained in $\Lambda_0$. Since $\OG_{a,b}$ is a path entirely contained in $\Lambda$, going from $u_3$ to $v_3$ and with an optimal time, so we have the result. 
			
			
		\end{itemize}
	\end{proof}


\begin{thebibliography}{1}
	
	\bibitem{AndjelVares}
	Enrique~D. Andjel and Maria~E. Vares.
	\newblock {First passage percolation and escape strategies}.
	\newblock {\em {Random Structures and Algorithms}}, 47(3):414--423, October
	2015.
	
	\bibitem{50years}
	A.~Auffinger, M.~Damron, and J.~Hanson.
	\newblock {\em 50 years of first-passage percolation}.
	\newblock American Mathematical Society, 2017.
	
	\bibitem{Boucheron}
	S.~Boucheron, G.~Lugosi, and P.~Massart.
	\newblock {\em Concentration Inequalities: A Nonasymptotic Theory of
		Independence}.
	\newblock OUP Oxford, 2013.
	
	\bibitem{Grimmett}
	Geoffrey Grimmett.
	\newblock {\em Percolation}, volume 321 of {\em Grundlehren der Mathematischen
		Wissenschaften [Fundamental Principles of Mathematical Sciences]}.
	\newblock Springer-Verlag, Berlin, second edition, 1999.
	
	\bibitem{KRAS}
	Arjun Krishnan, Firas Rassoul-Agha, and Timo Sepp{\"a}l{\"a}inen.
	\newblock Geodesic length and shifted weights in first-passage percolation.
	\newblock {\em arXiv preprint arXiv:2101.12324}, 2021.
	
	\bibitem{LSS}
	T.~M. Liggett, R.~H. Schonmann, and A.~M. Stacey.
	\newblock {Domination by product measures}.
	\newblock {\em The Annals of Probability}, 25(1):71 -- 95, 1997.
	
	\bibitem{Nakajima}
	Shuta Nakajima.
	\newblock On properties of optimal paths in first-passage percolation.
	\newblock {\em Journal of Statistical Physics}, 174(2):259–275, Oct 2018.
	
	\bibitem{VdBK}
	J.~van~den Berg and H.~Kesten.
	\newblock {Inequalities for the Time Constant in First-Passage Percolation}.
	\newblock {\em The Annals of Applied Probability}, 3(1):56 -- 80, 1993.
	
\end{thebibliography}
\end{document}